\documentclass[10pt]{article}
\usepackage[english]{babel}
\usepackage{natbib}
\usepackage{url}
\usepackage{inputenc}
\usepackage{amsfonts}
\usepackage{amsmath}
\usepackage{empheq}
\usepackage{graphicx}
\graphicspath{{images/}}
\usepackage{parskip}
\usepackage{fancyhdr}
\usepackage{bbold}
\usepackage{amssymb}
\usepackage{vmargin}
\usepackage{array}
\usepackage{amsthm}
\usepackage[ruled,vlined]{algorithm2e}
\usepackage{hyperref}
\usepackage{caption, subcaption}
\usepackage{float}
\usepackage{fourier-orns}
\usepackage{listings}
\usepackage{epsfig}
\usepackage{amsmath}               
\newtheorem{assumption}{Assumption}
\newtheorem{theorem}{Theorem}[section]

\newtheorem{lemma}[theorem]{Lemma}
\newtheorem{remark}{Remark}[section]
\theoremstyle{definition}

\DeclareMathAlphabet\mathbfcal{OMS}{cmsy}{b}{n}

\usepackage{listings}
\usepackage[scr=boondoxo,scrscaled=1.05]{mathalfa}

\usepackage[shortlabels]{enumitem}
\mathtoolsset{showonlyrefs=true}

\usepackage[bbgreekl]{mathbbol}

\newcommand{\cha}{\mathbb{1}_{\Omega_0}}

\newcommand{\om}{{\Omega_{0}}}

\newcommand{\R}{\mathbb{R}}

\newcommand{\bu}{\mathbf{u}}
\newcommand{\bU}{\mathbf{U}}
\newcommand{\bydef}{\stackrel{\mbox{\tiny\textnormal{\raisebox{0ex}[0ex][0ex]{def}}}}{=}}
\newcommand{\bgam}{\text{\boldmath{$\gamma$}}}
\newcommand{\bGam}{\text{\boldmath{$\Gamma$}}}

\newcommand{\bpi}{\text{\boldmath{$\Pi$}}}

\newcommand{\mbf}[1]{\mathbf{#1}}
\newcommand{\mbb}[1]{\mathbb{#1}}

\usepackage[dvipsnames]{xcolor}

\makeatletter
\newcommand\subsubsubsection{\@startsection{paragraph}{4}{\z@}{-2.5ex\@plus -1ex \@minus -.25ex}{1.25ex \@plus .25ex}{\normalfont\normalsize\bfseries}}
\title{Proving the existence of localized patterns, periodic solutions, and branches of periodic solutions in the 1D Thomas model
}
\author{
Dominic Blanco \footnote{McGill University, Department of Mathematics and Statistics, 805 Sherbrooke Street West, Montreal, QC, H3A 0B9, Canada. {\tt dominic.blanco@mail.mcgill.ca}}}
\begin{document}

\maketitle
\begin{abstract}
    In this paper, we present a general framework for constructively proving  the existence of stationary localized solutions, spatially periodic solutions, and branches of spatially periodic solutions in the 1D Thomas model. Specifically, we develop the necessary analysis to compute explicit upper bounds required in a Newton--Kantorovich approach. Given an approximate solution $\bar{\mathbf{u}}$, this approach relies on establishing that a well-chosen fixed point map is contracting on a neighborhood $\bar{\mathbf{u}}$. For this matter, we construct an approximate inverse of the linearization around $\bar{\mathbf{u}}$, and establish sufficient conditions under which the contraction is achieved. This provides a framework for which computer-assisted analysis can be applied to verify the existence and local uniqueness of solutions in a vicinity of $\bar{\mathbf{u}}$, and control the linearization around $\bar{\mathbf{u}}$. Furthermore, as the Thomas model has a non-polynomial nonlinearity, we will need to use different techniques to handle it during our analysis. Our contributions are to provide a partial answer to how one can approach rigorously verifying results in the Thomas model, to adapt and combine previously developed techniques to apply to the Thomas model, and to perform the computer-assisted analysis to obtain such results. The code to perform the rigorous proofs is available on Github at \cite{ThomasProofs.jl}.
\end{abstract}
\begin{center}
\end{center}
\section{Introduction}\label{sec : introduction}
In this paper, we constructively prove the existence of stationary localized solutions, spatially periodic solutions, and branches of spatially periodic solutions in the 1D Thomas partial differential equation (PDE). The Thomas model is a type of Reaction–Diffusion (RD) system. These systems have attracted significant interest due to their physical relevance, with applications spanning biology~\cite{root_hair_og, selkov_schankenberg_og}, chemistry~\cite{GrayScott1985, brusselator_og,schankenberg_og}, and ecology~\cite{SITEUR201481, ZELNIK201727}. These systems are particularly compelling because they often exhibit rich and complex behaviors while maintaining a relatively simple mathematical structure. A summary of several mechanisms in RD systems was outlined in the work of Champneys et al.\cite{Champneys_Bistability}.  In \cite{AlSaadi_unifying_framework}, the authors studied a general class of activator-inhibitor models and provided various results about existence of solutions and bifurcation theory. In concluding, the authors raise the critical open question of providing a rigorous proof for their results. Addressing this challenge is essential, as it bridges the gap between formal asymptotics, numerical exploration, and mathematical rigor. This question was partially answered in \cite{blanco_cadiot_fassler_saddle} where some of their results were verified.
\par We now focus on the Thomas model specifically. It is an enzyme system originally proposed in \cite{original_thomas}. As a PDE, it can be written as
\begin{equation}\label{eq:thomas}
    \begin{split}
        \partial_t u &= \nu\Delta u + \nu_4 - u - \frac{\nu_1 uv}{1+u+\nu_2 u^2},  \\
        \partial_t v &= \Delta v + \nu_3(\nu_5 - v) -  \frac{\nu_1 uv}{1+u+\nu_2 u^2}\\
    \end{split}
    \quad ~~ (u,v) = \left(u(x,t), v(x,t)\right), ~~ x \in \R,
\end{equation}
where $\nu$, and $\nu_j$ for $j = 1,2,3,4,5$ are positive  parameters, and $\Delta$ is the Laplacian operator. We also restrict $\nu_2 > \frac{1}{4}$.
Some of the earliest studies on the model were \cite{thomas_1,thomas_2,thomas_3}. In \cite{thomas_1}, the model was studied numerically and experimentally. The authors concluded that diffusion of metabolites and an autocatalysis in reaction can induce oscillations. The authors of \cite{thomas_2} then studied the model again numerically using finite element methods. Conclusions were made regarding the bifurcation structure of the model. In \cite{thomas_3}, more numerical results were obtained and the model was also studied from a more analytical perspective. The authors also provided stability analysis. Later, the model was studied by James Murray in \cite{murray_1,murray_2}. Murray proposed the model as a way to study animal coat markings. The importance of Murray's work has led some to refer to the model as the Murray model or Thomas-Murray model (cf. \cite{thomas_sanderwanner} and \cite{gen_rd} respectively). We refer the interested reader to the books \cite{murray_book1,murray_book2} for more details on Murray's work with the model. The authors of \cite{thomas_1d} then studied the model using asymptotic approximations and provided a framework on its bifurcation structure in 1D. In their conclusion, the authors posed a question regarding if their results could be made rigorous. We wish to partially answer this question in this manuscript.

\subsection{A Computer-assisted approach}

One way to approach solving PDEs is to use computer-assisted proofs (CAPs) techniques. CAPs can help derive a constructive methodology for rigorously validating solutions to PDEs. While there are various strategies for CAPs, ours will be one that relies on a Newton-Kantorovich argument. Our approach will rely on the construction of a polynomial. It involves the computation of specific bounds, which can be evaluated rigorously on the computer by the use of interval arithmetic for instance \cite{julia_interval}. By computing these bounds, one can then verify that for some value of $r$, the polynomial in question is strictly negative. If this is the case, we obtain a unique solution in a ball of radius $r$ around an approximate solution. This strategy is sometimes referred to as the Radii-Polynomial approach. Note that the Radii-Polynomial approach is not the only way to approach CAPs with a Newton-Kantorovich argument. We refer the interested reader to the works \cite{arioli_sug2,continue3,nakao_sug3}. Let us now turn our attention to the problems we wish to solve using CAPs. As we are treating periodic and localized patterns, we discuss current approaches to tackle both problems.
\par To prove periodic solutions, the theory is well developed. More specifically, Fourier analysis can be used where the PDE is turned into a zero finding problem $F(U) = 0$ on sequence spaces.  For an approximate solution $\overline{U}$, the difficulty becomes in handling $DF(\overline{U})^{-1}$. To do so, one can exploit the
fact that the Fréchet derivatives are (asymptotically) diagonally
dominant (see \cite{gabriel_pfc,MR3633778,period_kuramoto,Sander_equilibrium,JB_symmetries_1}
for some illustrations). As
$DF(0)$ (the linear part) dominates for high-order modes, the tail of $DF(\overline{U})$
can be seen as a diagonally-dominant infinite dimensional matrix. Due to compactness, one can then approximate $DF(\overline{U})^{-1}$ as a finite matrix stored on the computer and a diagonal tail which can be controlled theoretically. There are also approaches when the tail is not diagonally dominant (see \cite{MR3392647,MR4379799,cyranka2018construction} for instance). 
\par Additionally, the ability to prove branches of periodic solutions is well studied. For some of the first examples, we refer the interested reader to \cite{arioli_sug2,continue3,continue2,continue1}. Recently, a new approach was developed by the author of \cite{polynomial_chaos}. Using the constructive approach, the author first identifies an approximate branch of periodic solutions. Then, using Chebyshev series, one can obtain an approximation of this branch to use as part of a Newton-Kantorovich argument. Then, using a computer-assisted proofs, one can obtain a rigorous enclosure of the true branch of periodic solutions. The approach has been used in a variety of studies, including \cite{continuation_1,continuation_2,continuation_3}. For instance in \cite{whitham_cadiot}, the author used this approach to compute a branch of solitons in the Whitham equation. In particular, the solution can be written as a Chebyshev series where the dependency is in the wave speed parameter. Then, by sampling the branch at the Chebyshev nodes, the author uses Fourier transforms to compute the branch. When there are saddle-node (fold) bifurcations present on the branch, one can write the Chebyshev series depending on its pseudo-arclength. This requires one to use an augmented system and solve for a parameter along with the solution itself. This was demonstrated in \cite{dominic_sh_periodic,maxime_general,maxime_paper_continuation,marschal,olivier_kevin_paper,continue1} where the augmented system was considered in order to perform the proof. Note that the approach does not directly lead to a proof of the saddle-node bifurcation itself, but the augmented system allows for one to move past such bifurcations. This is possible due to the underlying mechanics being pseudo-arclength continuation. \par On unbounded domains, less theory exists. For the  ODE case, one can set up a projected boundary value problem via the parameterization method \cite{MR1976079, MR1976080, MR2177465}. This method relies on the rigorous parameterization of invariant manifolds at the zero solution. The method was first combined with finite elements in \cite{jay_jp_gs} and later combined with Chebyshev series in \cite{suspension_bridge,chebyshev_parameterization}. Still for ODEs, the analysis derived in \cite{ breden2024constructiveproofssemilinearpdes,breden2025solutionsdifferentialequationsfreudweighted} allows to produce constructive existence proofs in a well-chosen class of equations. Indeed, such equations possess a confining potential, allowing the resolvent of the differential operator to be compact on a well-chosen space of functions. Note that the method derived in \cite{breden2024constructiveproofssemilinearpdes} also applies to PDEs possessing a linear part of the form $\Delta u + \frac{x}{2} \cdot \nabla u$.
In the weak formulation for PDEs, the authors of \cite{plum_numerical_verif} presented a method for proving weak solutions to second and fourth-order PDEs. Their approach relies on the rigorous control of the spectrum of the linearization around an approximate solution. Then, they use a homotopy argument and the Temple-Lehmann-Goerisch method (see Section 10.2 in \cite{plum_numerical_verif}). This approach was then applied by Wunderlich in \cite{plum_thesis_navierstokes} where he proved the existence of a weak solution to the Navier-Stokes equations defined on an infinite strip with an obstacle. As far as strong solutions are concerned, \cite{olivier_radial} provides a methodology for proving the existence of radially symmetric solutions. Using the radially symmetric ansatz,  the PDE is transformed into an ordinary differential equation. The approach then relies on  a rigorous enclosure of the center stable manifold by using a Lyapunov-Perron operator. This allows one to solve a boundary value problem on $(0,\infty)$. Removing the radially symmetric assumption, the authors of \cite{unbounded_domain_cadiot} provide a general method for proving the existence and local uniqueness of localized solutions to semilinear autonomous PDEs on $\mathbb{R}^m$. The approach is based on Fourier series and describes the necessary tools to construct an approximate solution, $u_0$, and an approximate inverse $\mathbb{A}$ of the linearization $D\mathbb{F}(u_0)$ about $u_0$. An application to the 1D Kawahara equation was provided in the same paper. In \cite{sh_cadiot}, the method of \cite{unbounded_domain_cadiot} was applied to the planar Swift-Hohenberg PDE. The method was then extended further by the authors of \cite{gs_cadiot_blanco}. In \cite{unbounded_domain_cadiot}, the method was presented for scalar PDEs. In \cite{gs_cadiot_blanco}, the authors generalized the approach to systems of PDEs. This allows for rigorous proofs of localized solutions to systems of PDEs. The approach was demonstrated on the 2D Gray-Scott system of equations. Then, in \cite{symmetry_blanco_cadiot}, the authors extended further this methodology to treat existence proofs of symmetric solutions, where the symmetry might be given by a general group of symmetries. The 2D planar Swift-Hohenberg equation was again used as an illustration, allowing for the authors to obtain existence  proofs of solutions with various dihedral symmetries. 
\par A further challenge posed by \eqref{eq:thomas} is the presence of the nonpolynomial nonlinearity. As the approach we will use relies on Fourier series, the nonpolynomial nonlinearity has infinitely many Fourier coefficients. Various methodologies for tackling nonpolynomial nonlinearities in CAPs have been used throughout the years. For ODEs, there is the method of polynomial embedding \cite{olivier_embed,automatic_differentiation}. The approach relies on one being able to recognize the nonlinearity as the solution of an ODE. By then coupling this ODE along with the original system, one can perform a substitution and solve a polynomial system using classical CAPs techniques based on Banach algebra estimates. By coupling additional ODEs to the system, the computational cost increases. A more direct approach was developed in \cite{KAM,lindsay_suspensionbrdige}, and the thesis \cite{marco_thesis}. Here, the authors outlines a general approach for computing the Fourier coefficients of the nonpolynomial nonlinearity. This involves computing finitely many coefficients and bounding the aliasing error up to a fixed number $N$. The remaining coefficients are then rigorously upper bounded by a constant, $C_{\overline{\rho}}$ where $\overline{\rho}$ is the radius of analyticity of the nonpolynomial nonlinearity. One must then carefully verify that the choice of $\overline{\rho}$ was valid to ensure that the solution does not cross any poles in the complex plane. This approach is general and applicability for a wide variety of nonpolynomial nonlinearities. In more specific cases, different estimates can be performed. For instance, in \cite{matthieu_lindsay,lindsay_suspensionbrdige}, the authors directly treated the nonpolynomial nonlinearity $e^u$ in the case of the suspension bridge equation. In \cite{maxime_paper_continuation,olivier_kevin_paper}, the authors consider the case of inversion. That is, given a sequence $\overline{\phi}$, the authors provide estimates regarding its true inverse, $\overline{\phi}^{-1}$. The approach relies on computing an approximate inverse, $\phi_{\mathrm{inv}}$, and using it to obtain estimates allowing one to control the true inverse. As the Thomas model \eqref{eq:thomas} features inversion, we will rely on the theory provided by the authors of \cite{maxime_paper_continuation,olivier_kevin_paper} for our purposes.
\subsection{Our Contributions}
\par Our goal in this paper is to provide a partial answer to some of the questions asked by the authors of \cite{thomas_1d}. This is of theoretical significance as the authors asked in their conclusion how many of their results could be rigorously justified. With this in mind, we wish to provide the necessary tools and machinery to rigorously prove various results for the 1D Thomas model studied in \cite{thomas_1d}. Note that we will not be providing additional information on the dynamics of the Thomas model or the physical interpretation of its solutions or bifurcation structure. While our study is motivated by \cite{thomas_1d}, the methodology provided in our manuscript is applicable to the Thomas model in general. Our contribution for the Thomas model itself is to provide a constructive methodology for rigorously verifying results, such as those from \cite{thomas_1d}. Some aspects of our numerical study did rely on the insight provided by the authors of \cite{thomas_1d}; however, we did not directly use their numerical candidates as part of our CAPs. We will focus on three things: localized patterns, periodic solutions, and branches of periodic solutions. All three of these cases will contribute towards our goal of partially answering some of the questions asked by the authors of \cite{thomas_1d}.  We will focus first on localized patterns in Section \ref{sec : localized solutions} as it will introduce almost all of the necessary notation for this paper. From the perspective of validated numerics, our contributions in this section are to combine the methodology for proving localized patterns with the theory and estimates for handling the inverse of an element in a Banach space. More specifically, we will rely on the methodology developed by the authors of \cite{unbounded_domain_cadiot}. That is, we will use Fourier analysis to construct our approximate solution and approximate inverse. As we are treating a system, we use the extension of the method developed in \cite{gs_cadiot_blanco}. Our analysis will be similar to that performed in \cite{blanco_cadiot_fassler_saddle} as the authors of the aforementioned paper were treating a 1D example. Hence, the theory and methology is adapted from \cite{unbounded_domain_cadiot}. In addition to this approach, we will need to handle the nonpolynomial nonlinearity. For this, we rely on the approach used in \cite{maxime_paper_continuation,olivier_kevin_paper}. In particular, the authors of these papers provided two estimates for bounding the inverse element in the Banach space $\ell^1_{e,\tau}$ (see \eqref{ell1Gnu} for a definition). We will perform the needed analysis in order to apply these estimates for our purposes. As a result, handing the nonpolynomial nonlinearity is adapted from \cite{maxime_paper_continuation,olivier_kevin_paper}. What is a specialized implementation for the Thomas model are the estimates themselves performed in Section \ref{sec : Bounds of patterns}. Our next step will be periodic solutions, which is treated in Section \ref{sec : periodic solutions}. In this section, we will aim to address the periodic solutions provided by the authors of \cite{thomas_1d} by providing the tools to rigorously prove them. As proofs of periodic solutions are classical in CAPs, our main contribution will be to handle the nonpolynomial nonlinearity once again. We will improve the estimates provided by the authors of \cite{maxime_paper_continuation,olivier_kevin_paper} in order to obtain sharper results. Hence, a majority of our contributions in this section are to adapt previous results to the Thomas model; however, our generalization (cf. Lemma \ref{lem : phi bound stronger}) is applicable to a wider class of models and could be useful for studying other models. Finally, in Section \ref{sec : branches}, we will move to branches of periodic solutions. This section uses the theory of the previous Section \ref{sec : periodic solutions}, but parameterizes a branch of periodic solutions using Chebyshev series. We follow the previous illustrations in \cite{dominic_sh_periodic,marschal,olivier_kevin_paper} and others where an augmented zero finding problem is introduced in order to pass through saddle-node (fold) bifurcations. We are then able to perform a rigorous proof of a branch. Therefore, this section's goal is to adapt results from previous literature and sections of this manuscript to prove branches of periodic solutions. In all three problems we consider, while the computer-assisted methodology is adapted from the previous aforementioned works; however, the adaptation itself is nontrivial. The theoretical bounds we compute are specific implementations for the Thomas model. This makes the computations novel in the sense that one can attempt to apply our results generally for other localized patterns, periodic solutions, and branches of periodic solutions in the Thomas model. This makes the computation aspect of the paper significant from the perspective of validated numerics as well.

\par This manuscript is organized as follows. We begin with localized patterns in Section~\ref{sec : localized solutions}. Here, we  introduce the notation, assumptions, and restrictions used throughout this section. In Section~\ref{sec : patterns}, we introduce a Newton-Kantorovich approach associated to the construction of an approximate solution $\bar{\mathbf{u}}$ and of an approximate inverse $\mathbb{A}$ for $D\mathbb{F}(\bar{\mathbf{u}})$. Furthermore, in Section \ref{sec : Bounds of patterns}, we discuss the necessary ingredients to handle the nonpolynomial nonlinearity.  In the same section, we present applications of our analysis and derive multiple constructive existence proofs of localized solutions. In Section \ref{sec : periodic solutions}, we present the equivalent approach for periodic solutions. This includes the numerical aspects in Section \ref{sec : numerical aspects of periodic solutions} and the computation of the bounds in Section \ref{sec : Bounds of periodic solutions}. We then present the application of our analysis on periodic solutions and perform various constructive proofs of such solutions. In Section \ref{sec : branches}, we introduce the necessary tools and notations to prove branches of periodic solutions. This includes the construction of an approximate branch of periodic solutions and the needed augmented zero finding problem for which we will rigorously establish a solution of. We compute the bounds for these branches in Section \ref{sec : bounds branches}, which are based on those computed for periodic solutions. We then present a constructive proof of one such branch.

\section{Localized Solutions}\label{sec : localized solutions}
Our goal is to provide rigorous validations to some of the observations made by the authors of \cite{thomas_1d}, and we will begin with localized solutions for the spatial domain $\R$. For this matter, we introduce usual notations of functional analysis on $\R$. We first define the Lebesgue notation on a product space as $L^2 = L^2(\mathbb{R}) \times L^2(\mathbb{R})$ and $L^2(\om)$ on a bounded domain $\om$ in $\mathbb{R}$. More generally, $L^p$ denotes the usual $p$ Lebesgue product space with two components on $\mathbb{R}$ associated to its norm $\| \cdot \|_{p}$. Moreover, given $s \in \mathbb{R}$, denote by $H^s \bydef H^s(\mathbb{R}) \times H^s(\mathbb{R})$ the usual Sobolev space on $\mathbb{R}$.  For a bounded linear operator $\mathbb{K} : L^2 \to L^2$, denote by $\mathbb{K}^*$ the adjoint of $\mathbb{K}$ in $L^2$. Moreover, if $\mathbf{u} \in L^2$ where $\mathbf{u} = (u_1,u_2)$, denote by $\hat{\mathbf{u}} \bydef \mathcal{F}(\mathbf{u}) \bydef (\mathcal{F}(u_1),\mathcal{F}(u_2))$ the Fourier transform of $\mathbf{u}$. More specifically,  $\displaystyle \hat{u}_j(\xi)  \bydef  \int_{\mathbb{R}}u_j(x)e^{-i2\pi x \cdot \xi}dx$ for all $\xi \in \mathbb{R}$ and $j \in \{1,2\}$.

\par As we are looking for localized solutions, we wish to have $\bu : \R \to \R^2$ such that $\mathbb{F}(\bu) = 0$ and  satisfying $\bu(x) \to 0$ as $|x| \to \infty.$  Given a well-chosen Hilbert space $\mathcal{H}_e$ (cf. Section \ref{sec : setup}), containing even localized functions, we look for zeros of $\mathbb{F}$ in $\mathcal{H}_e$. In fact, we demonstrate that $\mathbb{F} : \mathcal{H}_e \to L^2_e$ is a smooth operator, where $L^2_e$ is the restriction of even functions of $L^2(\R) \times L^2(\R)$. Then, following the analysis derived in \cite{blanco_cadiot_fassler_saddle,gs_cadiot_blanco, unbounded_domain_cadiot}, we prove the existence of zeros of $\mathbb{F}$ using a Newton-Kantorovich approach (cf. Section \ref{sec : patterns}). In fact, our analysis relies on the construction of an approximate solution $\bar{\bu} \in \mathcal{H}_e$ and of an approximate inverse $\mathbb{A} : {L}^2_e \to \mathcal{H}_e$ for the Fréchet derivative $D\mathbb{F}(\bar{\bu})$. Indeed, we can then define a Newton-like fixed point operator $\mathbb{T}$ given as 
\begin{align*}
    \mathbb{T}(\bu) \bydef \bu - \mathbb{A}\mathbb{F}(\bu)
\end{align*}
and we prove that there exists $r>0$ such that $\mathbb{T}$ is contracting from $\overline{B_r(\bar{\bu})}$ to itself, where $\overline{B_r(\bar{\bu})}$ is the closed ball in $\mathcal{H}_e$ of radius $r$ centered at $\bar{\bu}$. Using the Banach-fixed point theorem, this implies the existence of a unique zero $\tilde{\bu}$ of $\mathbb{F}$ in $\overline{B_r(\bar{\bu})}$. The verification of the contractivity of $\mathbb{T}$ on $\overline{B_r(\bar{\bu})}$ involves the rigorous computation of different quantities, which we expose in Section \ref{sec : Bounds of patterns}. This allows us to provide novel constructive existence proofs of localized solutions in \eqref{eq:thomas}.  Let us now setup the problem at hand.

\subsection{Setup of the Problem}\label{sec : setup} 
\par Going back to the study of  \eqref{eq:thomas}, we set $\partial_t u = \partial_t v = 0$ as we are looking for stationary solutions. Then, we introduce a change of variable to ensure we can look for $\mathbf{u}(x) \to 0$ as $|x| \to \infty$. 
 \begin{align}
     \lambda_2 \bydef \frac{(1 + \lambda_1 + \nu_2 \lambda_1^2)(\nu_4 - \lambda_1)}{\nu_1 \lambda_1}
 \end{align}
 where $\lambda_1$ solves the cubic equation
 \begin{align}
     \nu_3 \nu_2 \lambda_1^3 - ((\nu_4 \nu_2 - 1)\nu_3 - \nu_1) \lambda_1^2 + ((\nu_5 \nu_1 - (\nu_4 - 1))\nu_3 - \nu_4 \nu_1) \lambda_1 - \nu_4 \nu_3 = 0.
 \end{align}
 Observe that $(\lambda_1,\lambda_2)$ is a steady state of \eqref{eq:thomas}. Therefore, we can make the change of variables
 \begin{align}
     u_1 \bydef u - \lambda_1 ~ \text{and} ~ u_2 \bydef \nu u_1 - v + \lambda_2 
 \end{align}
 to equivalently obtain
\begin{equation}\label{eq : gen_model}
\begin{split}
   0 &= \nu \Delta u_1 - u_1 - \frac{\lambda_3 u_1^2 + \lambda_4 u_1 - \nu_1 u_1 u_2 - \nu_1 \lambda_1 u_2}{1 + u_1 + \lambda_1 + \nu_2(u_1 + \lambda_1)^2},  \\
      0 &=\Delta u_2 - \nu_3 u_2 + \lambda_5 u_1
      \end{split}
\end{equation}
where
\begin{align}
    &\lambda_3 \bydef \nu_1 \nu + (\lambda_1 - \nu_4)\nu_2,~ \lambda_4 \bydef \nu_1 \lambda_2 + (1 + 2\nu_2 \lambda_1)(\lambda_1 - \nu_4) + \nu_1 \lambda_1 \nu,~\lambda_5 \bydef \nu_3 \nu - 1.
\end{align}
Using the above, $(u,v)$ solves \eqref{eq : gen_model} if and only if $(u_1,u_2)$ solves \eqref{eq:thomas}. Let us denote $\mbf{u}=(u_1, u_2)$ and define $\mathbb{F}$ as follows
\begin{align}
    \mbb{F}(\mbf{u}) \bydef \mathbb{L} \bu + \mathbb{G}(\bu),
\end{align}
where
\begin{equation}\label{def : def l and g intro}
    \mathbb{L} \bydef \begin{bmatrix}
        \nu \Delta + (\lambda_6 - 1)I_d & \lambda_7 I_d \\
        \lambda_5 I_d & \Delta - \nu_3 I_d
    \end{bmatrix}  = \begin{bmatrix}
        \mathbb{L}_{11} & \mathbb{L}_{12} \\ \mathbb{L}_{21} & \mathbb{L}_{22}
    \end{bmatrix}, ~~  \mathbb{G}(\mathbf{u}) \bydef \begin{bmatrix}
        \mathbb{g}(\mathbf{u}) \\
        0
    \end{bmatrix},
\end{equation}
and
\begin{align}
    &\lambda_6 \bydef -\frac{\lambda_1\lambda_4 + \lambda_4 + \lambda_1^2 \lambda_4 \nu_2}{(1 + \lambda_1 + \nu_2 \lambda_1^2)^2}, ~ \lambda_7 \bydef \frac{\nu_1 \lambda_1}{1+\lambda_1 + \nu_2 \lambda_1^2}, \\
    &\mathbb{g}(\mathbf{u}) \bydef -\frac{\lambda_3 u_1^2 + \lambda_4 u_1 - \nu_1 u_1 u_2 - \nu_1 \lambda_1 u_2}{1 + u_1 + \lambda_1 + \nu_2(u_1 + \lambda_1)^2} - \lambda_6 u_1 - \lambda_7 u_2.
\end{align}
Note that by defining $\mathbb{G}(\mathbf{u})$ in this way, we ensure that $\mathbb{g}(0) = 0$ and $D_{u_j}\mathbb{g}(0) = 0$ for $j = 1,2$. Adopting the notation of \cite{unbounded_domain_cadiot}, we look for $\mbf{u}=(u_1, u_2)$ such that
\begin{align}
    \mbb{F}(\mbf{u})  = 0
\end{align}
where
$\mathbb{F}$ is given in \eqref{def : def l and g intro}. Note that $\mathbb{L}$ possesses a symbol (Fourier transform) $l: \mathbb{R} \to \mathbb{R} \times \mathbb{R}$ given as 
\begin{equation}\label{def : definition of symbol l}
    l(\xi) \bydef \begin{bmatrix}
        l_{11}(\xi) & l_{12}(\xi) \\
        l_{21}(\xi) & l_{22}(\xi)
    \end{bmatrix}
    = \begin{bmatrix}
         -\nu|2\pi\xi|^2 +\lambda_6 - 1 & \lambda_7 \\
        \lambda_5 & -|2\pi\xi|^2 - \nu_3
    \end{bmatrix}
\end{equation}
for all $\xi \in \mathbb{R}$. In order to use the approach of \cite{unbounded_domain_cadiot}, we require two assumptions. The first is regarding the linear part, $\mathbb{L}$. Its analytical meaning is that we assume $\mathbb{L}$ to be invertible. As we have a system, we recall the generalized  assumption from \cite{gs_cadiot_blanco} that our framework requires. 

\begin{assumption}\label{assumption : fourier}
  Given $l$ as in \eqref{def : definition of symbol l},  assume there exists $\sigma_0 > 0$ such that 
    \begin{equation}\label{eq : fourier transform bounded away from 0}
        |\det(l(\xi))| \geq \sigma_0 \text{ for all } \xi \in \mathbb{R}.
    \end{equation}
    That, is, $\det(l(\xi))$ is bounded away uniformly from $0$.
\end{assumption}
Assuming the invertibility of $\mathbb{L}$ is a strong
assumption and it may seem unrelated to CAPs. We will justify this assumption later (cf. Remark \ref{rem : L inverse}). We now provide the values of the parameters $\nu_j$ ($j\in \{1,2,3\})$ and $\lambda_k$ ($k \in \{1,\dots,7\})$ such that $l$ satisfies Assumption \ref{assumption : fourier}. 
\begin{lemma}\label{lem : l_invertible}$l$ is invertible if
    \begin{equation}
        (\nu\nu_3-\lambda_6 + 1)^2 + 4\nu(\nu_3\lambda_6 - \nu_3 + \lambda_5\lambda_7) < 0,
    \end{equation}\label{eq : invertible condition 1}
    or if
    \begin{equation}
        \begin{cases}
            (\nu\nu_3-\lambda_6 + 1)^2 + 4\nu(\nu_3\lambda_6 - \nu_3 + \lambda_5\lambda_7) \geq 0 ~~ \text{ and } \\
            \nu\nu_3-\lambda_6 + 1 > \sqrt{(\nu\nu_3-\lambda_6 + 1)^2 + 4\nu(\nu_3\lambda_6 - \nu_3 + \lambda_5\lambda_7)}
        \end{cases}
    \end{equation}\label{eq : invertible condition 2}
\end{lemma}
\begin{proof}
    $l$ is invertible if and only if $l_{11}(\xi)l_{22}(\xi) - l_{12}(\xi)l_{21}(\xi) \neq 0, \forall \xi \in \mathbb{R}$. 
    This is equivalent to studying the roots of  the second order polynomial $x \mapsto \nu x^2 + (\nu\nu_3 - \lambda_6 + 1)x - \nu_3\lambda_6 + \nu_3 - \lambda_5 \lambda_7 $. The proof is then obtained using basic properties of second order polynomials.
\end{proof}
From now on, we assume that the parameters $\nu$, $\nu_3$, and $\lambda_k$ ($k \in \{5,6,7\}$) satisfy the conditions of Lemma \ref{lem : l_invertible}, yielding that $l$ is invertible.
Since $l$ is invertible, we can define the  following norm and inner product \begin{align}\label{def : definition of the norm and inner product Hl}
    \|\mathbf{u}\|_{\mathcal{H}} \bydef \|\mathbb{L}\mathbf{u}\|_{2} ~~ \text{ and } ~~(\mathbf{u},\mathbf{v})_{\mathcal{H}} \bydef (\mathbb{L}\mathbf{u},\mathbb{L}\mathbf{v})_2 
\end{align}
for all $\mathbf{u}, \mathbf{v} \in \mathcal{H}$, where   $\mathcal{H}$ is the Hilbert space
\begin{align}\label{def : Hilbert space Hl}
    \mathcal{H} \bydef \{ \mathbf{u} \in L^2, \|\mathbf{u}\|_{\mathcal{H}} < \infty \}.
\end{align} 
 Using Plancherel's identity, we have 
\begin{align*}
(\mathbf{u}, \mathbf{v})_{\mathcal{H}} = (l\hat{\mathbf{u}},l\hat{\mathbf{v}})_2 
\end{align*}
for all $\mathbf{u}, \mathbf{v} \in \mathcal{H}.$ In particular, note that $\mathbb{L} : \mathcal{H} \to L^2$ is a well-defined bounded linear operator, which is actually an isometric isomorphism. We now want to prove that the operator $\mathbb{G}: \mathcal{H} \to L^2$ is a smooth operator. This is a nontrivial task in our case since $\mathbb{G}$ is nonpolynomial. For this matter, we recall two spaces from \cite{gs_cadiot_blanco} :  $\mathcal{M}_1$ and $\mathcal{M}_2$ given as
\begin{align}
 &\mathcal{M}_1 \bydef \left\{ M = \left(M_{i,j}\right)_{i,j \in \{ 1,2 \} }, \text{ where } M_{i,j} \in L^{\infty}(\mathbb{R}) \text{ and } ~ \|M\|_{\mathcal{M}_1} < \infty \right\}  \\
&\mathcal{M}_2 \bydef \left\{ M = \left(M_{i,j}\right)_{i,j \in \{ 1,2 \} }, \text{ where } M_{i,j} \in L^{2}(\mathbb{R}) \text{ and } ~ \|M\|_{\mathcal{M}_2} < \infty \right\}
\end{align}
with their associated norms
\begin{equation}\label{M_2__21_norm_def}
 \|M\|_{\mathcal{M}_1} \bydef \sup_{\xi \in \mathbb{R}} \sup\limits_{\substack{x \in \mathbb{R}^2\\ |x|_2=1}}  |M(\xi)x|_2,~\text{and}~\|M\|_{\mathcal{M}_2} \bydef \max_{i \in \{1,2\}} \left\{\left(\sum_{j = 1}^2 \|M_{i,j}\|_{L^2(\mathbb{R})}^2\right)^{\frac{1}{2}} \right\}\end{equation}
In fact, using $\mathcal{M}_1$ and $\mathcal{M}_2$, \cite{gs_cadiot_blanco} provides sufficient conditions for which products on $\mathcal{H}\times \mathcal{H} \to L^2$ are well-defined. Using this result, we obtain the following lemma.
\begin{lemma}\label{corr : banach algebra}
Let $\mathbf{u} = (u_1,u_2)$ and $ \mathbf{v} = (v_1,v_2)$. Suppose Assumption~\ref{assumption : fourier} is verified and let 
{\footnotesize\begin{align}
    \kappa_0 \bydef \frac{4\nu_2}{4\nu_2 - 1}, ~\kappa_1 \bydef \|l^{-1}\|_{\mathcal{M}_1}\|l^{-1}\|_{\mathcal{M}_2}, ~ \kappa_2 \bydef \|l^{-1}\|_{\mathcal{M}_2}(|\lambda_3|+\nu_1), ~ \kappa_3 \bydef |\lambda_4| + \nu_1 |\lambda_1|, ~ \kappa_4 \bydef |\lambda_6|+|\lambda_7|.
\end{align}} Then, for $j = 1,2$, 
\begin{align}
\|\mathbb{g}(\mathbf{u})v_j\|_2 \leq \kappa_1(\kappa_0 (\kappa_2 \|\mathbf{u}\|_{\mathcal{H}} + \kappa_3) + \kappa_4)\|\mathbf{u}\|_{\mathcal{H}}\|\mathbf{v}\|_{\mathcal{H}}.
\end{align} 
\end{lemma}
\begin{proof}
    To begin, observe that
    \begin{align}
        \|\mathbb{g}(\mathbf{u})v_j\|_{2} \leq \|\mathbb{g}(\mathbf{u})\|_{\infty} \|v_j\|_{2} \leq \|l^{-1}\|_{\mathcal{M}_1}\|\mathbb{g}(\mathbf{u})\|_{\infty} \|\mathbf{v}\|_{\mathcal{H}}\label{first part kappa comp}
    \end{align}
where the last step followed from the proof of Lemma 2.1 of \cite{blanco_cadiot_fassler_saddle}. Now, consider the denominator of $\mathbb{g}(\mathbf{u})$. We define the function
    \begin{align}
        h(x) \bydef 1 + x + \nu_2 x^2.
    \end{align}
Now, since $h$ is a quadratic equation and $\nu_2 > \frac{1}{4} > 0$, its minimum is located at its vertex point $\left(-\frac{1}{2\nu_2},\frac{4\nu_2 - 1}{4\nu_2}\right)$.
Hence, we see that
    {\small\begin{align}
        \|\mathbb{g}(\mathbf{u})\|_{\infty} &\leq \left\| \frac{\lambda_3 u_1^2 + \lambda_4 u_1 - \nu_1 u_1 u_2 - \nu_1 \lambda_1 u_2}{1 + u_1 + \lambda_1 + \nu_2(u_1 + \lambda_1)^2}\right\|_{\infty} + |\lambda_6| \|u_1\|_{\infty} + |\lambda_7| \|u_2\|_{\infty}\\
        &\leq \frac{4\nu_2}{4\nu_2 - 1} \|\lambda_3 u_1^2 + \lambda_4 u_1 - \nu_1 u_1 u_2 - \nu_1 \lambda_1 u_2\|_{\infty} + |\lambda_6| \|u_1\|_{\infty} + |\lambda_7| \|u_2\|_{\infty}\\
        &\leq \kappa_0 \left(|\lambda_3| \|u_1\|_{\infty}^2 + |\lambda_4| \|u_1\|_{\infty} + \nu_1 \|u_1\|_{\infty} \|u_2\|_{\infty} + \nu_1 |\lambda_1| \|u_2\|_{\infty}\right) + |\lambda_6| \|u_1\|_{\infty} + |\lambda_7| \|u_2\|_{\infty} \\
        &\leq \kappa_0\|l^{-1}\|_{\mathcal{M}_2} \left(\|l^{-1}\|_{\mathcal{M}_2}(|\lambda_3|+\nu_1)\|\mathbf{u}\|_{\mathcal{H}} + |\lambda_4| + \nu_1 |\lambda_1|\right)\|\mathbf{u}\|_{\mathcal{H}} + \|l^{-1}\|_{\mathcal{M}_2}(|\lambda_6|+|\lambda_7|)\|\mathbf{u}\|_{\mathcal{H}} \\
        &\bydef \|l^{-1}\|_{\mathcal{M}_2}\left(\kappa_0 (\kappa_2 \|\mathbf{u}\|_{\mathcal{H}} + \kappa_3) + \kappa_4\right)\|\mathbf{u}\|_{\mathcal{H}}
    \end{align}}
where we used the steps of Lemma 2.1 of \cite{blanco_cadiot_fassler_saddle} once again. Combining with \eqref{first part kappa comp}, we obtain
\begin{align}
    \|\mathbb{g}(\mathbf{u})v_j\|_{2} &\leq \|l^{-1}\|_{\mathcal{M}_1}\|l^{-1}\|_{\mathcal{M}_2}(\kappa_0  (\kappa_2 \|\mathbf{u}\|_{\mathcal{H}} + \kappa_3) + \kappa_4)\|\mathbf{u}\|_{\mathcal{H}}\|\mathbf{v}\|_{\mathcal{H}} \\
    &\bydef \kappa_1 (\kappa_0  (\kappa_2 \|\mathbf{u}\|_{\mathcal{H}} + \kappa_3) + \kappa_4)\|\mathbf{u}\|_{\mathcal{H}}\|\mathbf{v}\|_{\mathcal{H}}
\end{align}
as desired. Note that $\|l^{-1}\|_{\mathcal{M}_1}$ and $\|l^{-1}\|_{\mathcal{M}_2}$ can be computed using the results from \cite{blanco_cadiot_fassler_saddle}.
\end{proof}
Using Lemma \ref{corr : banach algebra}, we obtain that $\mathbb{G} : \mathcal{H} \to L^2$ is a smooth operator. Supposing that $\mathbf{u} = (u_1,u_2)$ is a solution of \eqref{eq : gen_model}, then any translation in space provides a new solution. Therefore, in order to isolate a localized solution in the set of solutions, we choose to look for even solutions. That is, we enforce that $\mathbf{u}(x) = \mathbf{u}(-x)$ for all $x \in \R$.   With this in mind, we introduce the following even restriction $\mathcal{H}_{e} \subset \mathcal{H}$
\begin{equation}\label{H_l_e_definition}
         \mathcal{H}_{e} \bydef \{ \mathbf{u} \in \mathcal{H} \ | \ \mathbf{u}(x) = \mathbf{u}(-x)\}.
    \end{equation}
Similarly, denote $L_{e}^2$ the Hilbert subspace of $L^2$ satisfying the even symmetry.
In particular we notice that if $\mathbf{u} \in \mathcal{H}_{e}$, then $\mathbb{L}\mathbf{u} \in L^2_{e}$ and $\mathbb{G}(\mathbf{u}) \in L^2_{e}$ ; hence,  $\mathbb{L}$ and $\mathbb{G}$ are well-defined as operators from $\mathcal{H}_{e}$ to $L^2_{e}$. 
 Finally, we look for solutions of the following problem
\begin{equation}\label{eq : f(u)=0 on H^l_e}
    \mathbb{F}(\mathbf{u}) = 0 ~~ \text{ and } ~~ \mathbf{u} \in \mathcal{H}_{e}.
\end{equation}
\begin{remark}\label{rem : L inverse}
As mentioned previously, Assumption \ref{assumption : fourier} is a strong assumption that may seem unrelated; however,
a key challenge for our rigorous approach will be to control the inverse of $D\mathbb{F}(\overline{\mathbf{u}})$.The authors of \cite{unbounded_domain_cadiot} show that the essential spectrum of
$D\mathbb{F}(\overline{\mathbf{u}}))= \mathbb{L} + D\mathbb{G}(\overline{\mathbf{u}})$ is equal to the essential spectrum of $\mathbb{L}$. If we suppose that $\mathbb{L}$ is singular, then there exists a $\xi$
such that $l(\xi) = 0$; hence, $0$ is in the essential spectrum of both $\mathbb{L}$ and also that of $D\mathbb{F}(\overline{\mathbf{u}})$. As a result, $D\mathbb{F}(\overline{\mathbf{u}}) : \mathcal{H} \to L^2$
will also be singular. Relaxing this assumption is a possible future work as discussed in Section 7 of \cite{unbounded_domain_cadiot}.
\end{remark}
\subsection{Periodic Spaces}\label{sec : periodic spaces}
In this section, we recall some notations introduced in Section 2.4 of \cite{unbounded_domain_cadiot}. In particular, the objects defined in the above sections have a corresponding representation in Fourier series when restricted to a bounded domain. Specifically, we define our bounded domain of interest $\Omega_0 \bydef (-d,d)$  where $1 \leq d<\infty$.
Moreover, denote $\tilde{n} = \frac{n}{2d} \in \mathbb{R}$ for all $n \in \mathbb{Z}$. As for the continuous case, we want to restrict to Fourier series representing even functions. Let $\mathbf{u}_{n} = ((u_1)_{n},(u_2)_{n})$ be the $n$th Fourier coefficient of $\mathbf{u}$. In terms of Fourier coefficients, the even  restriction reads
\begin{align*}
    \mathbf{u}_{n} = \mathbf{u}_{-n} \text{ for all } n \in \mathbb{Z}.
\end{align*}
In particular, when enforcing the even symmetry, we can restrict the indices of Fourier coefficients from $\mathbb{Z}$ to the reduced set
$
 \mathbb{N}_0 \bydef \left\{n \in \mathbb{Z} \ | \ 0 \leq n \right\}.$
Now, let $(\alpha_n)_{n \in \mathbb{N}_0}$ be defined by 
\begin{equation}\label{def : alpha_n}
    \alpha_n \bydef \begin{cases}
        1 &\text{ if } n=0\\
       2 &\text{ if } n > 0. 
    \end{cases}
\end{equation}
In particular, $\alpha_n$ is the size of $\mathrm{orb}_{\mathbb{Z}_2}(n)$.  Next, let $\ell_{\tau}^p(\mathbb{N}_0)$ denote the following Banach space
\begin{align}
    \ell_{\tau}^p(\mathbb{N}_0) \bydef \left\{U = (u_n)_{n \in \mathbb{N}_0}: ~ \|U\|_p \bydef \left( \sum_{n \in \mathbb{N}_0} \alpha_n|u_n|^p\tau^{np}\right)^\frac{1}{p} < \infty \right\}.\label{ell1Gnu}
\end{align}
We also denote $\ell^p(\mathbb{N}_0) \bydef \ell^p_{1}(\mathbb{N}_0)$ the unweighted $\ell^p$ space.
In \eqref{ell1Gnu}, note that $U$ is a sequence of scalars, and not a sequence of vectors.
In particular, $\ell^2(\mathbb{N}_0)$ is an Hilbert space associated to its natural inner product $(\cdot, \cdot)_2$ given by
\[
(U,V)_2 \bydef \sum_{n \in \mathbb{N}_0} \alpha_n u_n v_n^*
\] 
for all $U = (u_n)_{n \in \mathbb{N}_0}, V = (v_n)_{n \in \mathbb{N}_0} \in \ell^2(\mathbb{N}_0)$ and $(\cdot)^*$ denotes complex conjugation.
Then, we define the Banach space $\ell^p_{e,\tau}$  as 
\begin{align}
   \ell^p_{e,\tau} \bydef \ell_{\tau}^p(\mathbb{N}_0) \times \ell_{\tau}^p(\mathbb{N}_0), ~ \text{ with norm } ~ \|\mathbf{U}\|_{p,\tau} = (\|U_1\|_{p,\tau}^p + \|U_2\|_{p,\tau}^p)^{\frac{1}{p}}.\label{def : norm on product space}
\end{align}
Additionally, $\ell^p_e \bydef \ell^p_{e,1}$.
 For a bounded operator $K : \ell^2_{e} \to \ell^2_{e}$ (resp. $K : \ell^2(\mathbb{N}_0) \to \ell^2(\mathbb{N}_0)$), $K^*$ denotes the adjoint of $K$ in $\ell^2_{e}$ (resp. $\ell^2(\mathbb{N}_0)$). Now, similarly as what is achieved \cite{gs_cadiot_blanco}, we  define $\gamma~:~L^2_{e}(\mathbb{R}) \to \ell^2(\mathbb{N}_0)$ as
\begin{align}
    \left(\gamma(u)\right)_n \bydef  \frac{1}{|\om|}\int_\om u(x) e^{-2\pi i \tilde{n}\cdot x}dx\label{single_gamma}
\end{align}
for all $n \in \mathbb{N}_0$ and all $u \in L^2_{e}(\R)$. Similarly, we define $\gamma^\dagger : \ell^2(\mathbb{N}_0) \to L^2_{e}(\mathbb{R})$ as 
\begin{align}
    \gamma^\dagger\left(U\right)(x) \bydef \cha(x) \sum_{n \in \mathbb{N}_0} \alpha_n u_n \cos(2\pi \tilde{m} x)\label{single_gamma_dagger}
\end{align}
for all $x \in \mathbb{R}$ and all $U =\left(u_n\right)_{n \in \mathbb{N}_0} \in \ell^2(\mathbb{N}_0)$, where $\cha$ is the characteristic function on $\om$. We now introduce a similar notation for functions in the product space $L^2_{e}$. In particular, define $\bgam : L^2_{e} \to \ell^2_{e}$ and 
$\bgam^\dagger : \ell^2_{e} \to L^2_{e}$  as
\begin{align}
\bgam(\mathbf{u}) = (\gamma(u_1),\gamma(u_2)),~\bgam^\dagger(\mathbf{U}) = (\gamma^\dagger(U_1),\gamma^\dagger(U_2)).\label{def : gamma and gamma dagger}
\end{align}
More specifically, given $\mathbf{u} \in L^2_{e}$, $\bgam(\mathbf{u})$ represents the Fourier coefficients indexed on $\mathbb{N}_0$ of the restriction of $\mathbf{u}$ on $\om$. Conversely, given a sequence $\mathbf{U}\in \ell^2_{e}$, $\bgam^\dagger\left(\mathbf{U}\right)$  is the function representation of $\mathbf{U}$ in $L^2_{e}.$ In particular, notice that $\bgam^\dagger\left(\mathbf{U}\right)(x) =0$ for all $x \notin \om.$ Then, recall similar notations from \cite{unbounded_domain_cadiot}
\begin{align}
    L^2_{e,\om} \bydef \left\{\mathbf{u} \in L^2_{e} : \text{supp}(\mathbf{u}) \subset \overline{\om} \right\}~~ \text{ and } ~~
   \mathcal{H}_{e,\om} \bydef \left\{\mathbf{u} \in \mathcal{H}_{e} : \text{supp}(\mathbf{u}) \subset \overline{\om} \right\}.
\end{align}
Moreover, define $\mathcal{B}(L^2_{e})$ (respectively $\mathcal{B}(\ell^2_{e})$) as the space of bounded linear operators on $L^2_{e}$ (respectively $\ell^2_{e}$) and denote by $\mathcal{B}_\om(L^2_{e})$ the following subspace of $\mathcal{B}(L^2_{e})$
\begin{equation}\label{def : Bomega}
    \mathcal{B}_\om(L^2_{e}) \bydef \{\mathbb{K} \in \mathcal{B}(L^2_{e}) :  \mathbb{K} = \cha \mathbb{K} \cha\}.
\end{equation}
Finally, define $\bGam : \mathcal{B}(L^2_{e}) \to \mathcal{B}(\ell^2_{e})$ and $\bGam^\dagger : \mathcal{B}(\ell^2_{e}) \to \mathcal{B}(L^2_{e})$ as follows
\begin{equation}\label{def : Gamma and Gamma dagger}
    \bGam(\mathbb{K}) \bydef \bgam \mathbb{K} \bgam^\dagger ~~ \text{ and } ~~  \bGam^\dagger(K) \bydef \bgam^\dagger {K} \bgam 
\end{equation}
for all $\mathbb{K} \in \mathcal{B}(L^2_{e})$ and all $K \in \mathcal{B}(\ell^2_{e}).$

The maps defined above in \eqref{def : gamma and gamma dagger} and \eqref{def : Gamma and Gamma dagger} are fundamental in our analysis as they allow to pass from the problem on $\mathbb{R}$ to the one in $\ell^2_{e}$ and vice-versa. Furthermore, the following lemma, provides that this passage is an isometric isomorphism when restricted to the relevant spaces.

\begin{lemma}\label{lem : gamma and Gamma properties}
    The map $\sqrt{|\om|} \bgam : L^2_{e,\om} \to \ell^2_{e}$ (respectively $\bGam : \mathcal{B}_\om(L^2_{e}) \to \mathcal{B}(\ell^2_{e})$) is an isometric isomorphism whose inverse is given by $\frac{1}{\sqrt{|\om|}} \bgam^\dagger : \ell^2_{e} \to L^2_{e,\om}$ (respectively $\bGam^\dagger :   \mathcal{B}(\ell^2_{e}) \to \mathcal{B}_\om(L^2_{e})$). In particular,
    \begin{align}\label{eq : parseval's identity}
        \|\mathbf{u}\|_2 = \sqrt{\om}\|\mathbf{U}\|_2 \text{ and } \|\mathbb{K}\|_2 = \|K\|_2
    \end{align}
    for all $\mathbf{u} \in L^2_{e,\om}$ and $\mathbb{K} \in \mathcal{B}_\om(L^2_{e})$ where $\mathbf{U} \bydef \bgam(\mathbf{u})$ and $K \bydef \bGam(\mathbb{K})$.
\end{lemma}
\begin{proof}
The proof is obtained following similar steps as the ones of Lemma 3.2 in \cite{unbounded_domain_cadiot}.
\end{proof}
The above lemma not only provides a one-to-one correspondence between the elements in $L^2_{e,\om}$ (respectively $\mathcal{B}_\om(L^2_{e})$) and the ones in $\ell^2_{e}$ (respectively $\mathcal{B}(\ell^2_{e})$) but it also provides an identity on norms. This property is essential in our construction of an approximate inverse  in Section \ref{sec : A}.

Now, we define the Hilbert space $\mathscr{h}$ as 
\begin{align*}
    \mathscr{h} \bydef \left\{ \mathbf{U}  \in \ell^2_{e} \text{ such that } \|\mathbf{U}\|_{\mathscr{h}} < \infty \right\}
\end{align*} associated to its inner product $(\cdot,\cdot)_{\mathscr{h}}$ and norm $\|\cdot\|_{\mathscr{h}}$ defined as 
\begin{align*}
    (\mathbf{U},\mathbf{V})_{\mathscr{h}} \bydef \sum_{n \in \mathbb{N}_0}  \alpha_n\left(L(\tilde{n})\mathbf{u}_n\right)\cdot \left(L(\tilde{n})\mathbf{v}_n^*\right),~
    \|\mathbf{U}\|_{\mathscr{h}} \bydef \sqrt{(\mathbf{U},\mathbf{U})_{\mathscr{h}}}
\end{align*}
for all $\mathbf{U}=(\mathbf{u}_n)_{n \in \mathbb{N}_0}, \mathbf{V}=(\mathbf{v}_n)_{n \in \mathbb{N}_0} \in \mathscr{h}$ and $\alpha_n$ is defined as in \eqref{def : alpha_n}.
Denote by $L : \mathscr{h} \to \ell^2_{e}$ and $G : \mathscr{h} \to \ell^2_{e}$ the Fourier coefficients representation of $\mathbb{L}$ and $\mathbb{G}$ respectively. More specifically,
\begin{align*}
    &L \bydef \begin{bmatrix}
        L_{11} & L_{12}\\
        L_{21} & L_{22}
    \end{bmatrix},~G(\bU) \bydef \begin{bmatrix}
       g(\bU)\\
        0
    \end{bmatrix}
\end{align*}
for all $\bU \in \mathscr{h}.$ For all $(i,j) \in \{(1,1),(1,2),(2,1),(2,2)\}$, 
$L_{ij}$ is an infinite diagonal matrix with coefficients $\left(l_{ij}(\tilde{n})\right)_{n\in \mathbb{N}_0}$ on its diagonal.Then, we define $U*V \bydef \gamma(\gamma^\dagger(U)\gamma^\dagger(V))$ is defined as the discrete convolution (under even symmetry). In particular, notice that Young's inequality for convolution is applicable 
\begin{align}
    \|U*V\|_{2} \leq \|U\|_{2} \|V\|_{1}\label{young_inequality}
\end{align}
for all $U \in \ell^2(\mathbb{N}_0), V \in \ell^1(\mathbb{N}_0)$. 
We now define $g(\mathbf{U})$ as
{\small\begin{align}
    g(\mathbf{U}) &\bydef -\gamma\left(\frac{\gamma^\dagger(\lambda_3 U_1*U_1 + \lambda_4 U_1 - \nu_1 U_1 *U_2 - \nu_1 \lambda_1 U_2)}{\gamma^\dagger(e_0 + U_1 + \lambda_1e_0 + \nu_2(U_1 + \lambda_1)*(U_1 + \lambda_1))}+\gamma^{\dagger}(\lambda_6 U_1) + \gamma^{\dagger}(\lambda_7 U_2)\right) \\
    &= -\frac{\lambda_3 U_1^2 + \lambda_4 U_1 - \nu_1 U_1*U_2 - \nu_1 \lambda_1 U_2}{e_0 + U_1 + \lambda_1e_0 + \nu_2(U_1 + \lambda_1)^2} - \lambda_6 U_1 - \lambda_7 U_2
\end{align}}
where $(e_0)_n = 1$ if $n = 0$ and $(e_0)_n = 0$ for all $n \in \mathbb{N}$. Furthermore, using the definition of $L$, notice that \[\|\mathbf{U}\|_{\mathscr{h}} = \|L\mathbf{U}\|_{2}.\] Finally, we define $F(\mathbf{U}) \bydef L\mathbf{U} + G(\mathbf{U})$ and introduce  
\begin{equation}\label{eq : F(U)=0 in X^l_e}
    F(\mathbf{U}) =0 ~~ \text{ and } ~~ \mathbf{U} \in \mathscr{h}
\end{equation}
as the periodic equivalent on $\om$ of \eqref{eq : f(u)=0 on H^l_e}.
\subsection{Setting up the Computer-Assisted Approach for Localized Patterns}\label{sec : patterns}\
We now recall the primary theorem we apply for our computer assisted approach. This will be a Newton-Kantorovich type theorem for proving the existence of solutions. Its proof will be thanks to a fixed point argument. Then, we will need to construct the related objects to apply it and treat the nonpolynomial nonlinearity.
\subsubsection{Radii-Polynomial Theorem}\label{sec : radii poly theorem}
 Given $\overline{\mathbf{u}} \in \mathcal{H}_{e}$, an approximate solution to \eqref{eq : f(u)=0 on H^l_e}, and $\mathbb{A} : L^2_{e} \to \mathcal{H}_{e}$, an approximate inverse to $D\mathbb{F}(\overline{\mathbf{u}})$, we want to prove that there exists $r>0$ such that $\mathbb{T} : \overline{B_r(\overline{\mathbf{u}})} \to \overline{B_r(\overline{\mathbf{u}})}$ given by
\[
\mathbb{T}(\mathbf{u}) \bydef \mathbf{u} - \mathbb{A}\mathbb{F}(\mathbf{u})
\]
is well-defined and a contraction. In order to determine a possible value for $r>0$ that would provide the contraction, we wish to use a Radii-Polynomial theorem.  In particular, we  build $\mathbb{A} : L_{e}^2 \to \mathcal{H}_{e}$, $\mathcal{Y}_0, \mathcal{Z}_1 >0$ and  $\mathcal{Z}_2 : (0, \infty) \to [0,\infty)$ in such a way that the hypotheses of the following theorem are satisfied.
\begin{theorem}\label{th: radii polynomial}
Let $\mathbb{A} : L_{e}^2 \to \mathcal{H}_{e}$ be a bounded linear operator. Moreover, let $\mathcal{Y}_0, \mathcal{Z}_1$ be non-negative constants and let  $\mathcal{Z}_2 : (0, \infty) \to [0,\infty)$ be a non-negative function  such that for all $r>0$
  \begin{align}\label{eq: definition Y0 Z1 Z2}
    &\|\mathbb{A}\mathbb{F}(\overline{\mathbf{u}})\|_{\mathcal{H}} \leq \mathcal{Y}_0\\
    &\|I_d - \mathbb{A}D\mathbb{F}(\overline{\mathbf{u}})\|_{\mathcal{H}} \leq \mathcal{Z}_1\\
    &\|\mathbb{A}\left({D}\mathbb{F}(\mathbf{u}) - D\mathbb{F}(\overline{\mathbf{u}})\right)\|_{\mathcal{H}} \leq \mathcal{Z}_2(r)r, ~~ \text{for all } \mathbf{u} \in B_r(\overline{\mathbf{u}})
\end{align}  
If there exists $r_0>0$ such that
\begin{equation}\label{condition radii polynomial}
    \frac{1}{2}\mathcal{Z}_2(r_0)r_0^2 - (1-\mathcal{Z}_1)r_0 + \mathcal{Y}_0 <0, \ and \ \mathcal{Z}_1 + \mathcal{Z}_2(r_0)r_0 < 1 
 \end{equation}
then there exists a unique $\tilde{\mathbf{u}} \in \overline{B_{r_0}(\overline{\mathbf{u}})} \subset \mathcal{H}_{e}$ such that $\mathbb{F}(\tilde{\mathbf{u}})=0$, where $B_{r_0}(\overline{\mathbf{u}})$ is the open ball of radius $r_0$ in $\mathcal{H}_{e}$ and centered at $\overline{\mathbf{u}}$. 
\end{theorem}
\begin{proof}
The proof can be found in \cite{gs_cadiot_blanco}.
\end{proof}
In order to apply Theorem \ref{th: radii polynomial}, we need to construct explicitly $\overline{\mathbf{u}} \in \mathcal{H}_e$ and $\mathbb{A} : \mathcal{H}_e \to L^2_e$. Following this, we need to discuss the nonpolynomial nonlinearity. These are the topics of the next sections.
\subsubsection{Construction of \texorpdfstring{$\overline{\mathbf{u}}$}{mbfu0}}\label{sec : u0}
In this section, we discuss the  construction of $\overline{\mathbf{u}}$, which is an approximate solution to \eqref{eq : f(u)=0 on H^l_e}. This is generally a challenging problem. In the specific case of \eqref{eq:thomas}, we rely on the approximate solutions outlined in \cite{thomas_1d}. Our approach relies on $\overline{\mathbf{u}}$ being constructed numerically on $\om = (-d,d)$ thanks to its Fourier coefficients representation. Fix $N \in \mathbb{N}$ to be the size of our numerical approximation for linear operators (i.e. matrices) and $N_0 \in \mathbb{N}$ to be the one of our Fourier coefficients approximations (i.e. vectors). Now, given $\mathcal{N} \in \mathbb{N}$, let us introduce the following projection operators 
 \begin{align}
 \nonumber
    (\Pi^{\leq\mathcal{N}}V)_n  =  \begin{cases}
          v_n,  & n \in I^{\mathcal{N}} \\
              0, &n \notin I^{\mathcal{N}}
    \end{cases} ~~ \text{ and } ~~
     (\Pi^{>\mathcal{N}}V)_n  =  \begin{cases}
          0,  & n \in I^{\mathcal{N}} \\
              v_n, &n \notin I^{\mathcal{N}}
    \end{cases}\label{def : piN and pisubN}
 \end{align}
where $I^{\mathcal{N}} \bydef \{n \in \mathbb{N}_0, ~ n \leq \mathcal{N}\}$ for all $V = (v_n)_{n \in  \mathbb{N}_0} \in \ell^2_e.$ Then, we define
\begin{align}
    (\bpi^{\leq\mathcal{N}}\mathbf{U})_n \bydef ((\Pi^{\leq\mathcal{N}}U_1)_n, (\Pi^{\leq\mathcal{N}}U_2)_n) ~~\text{and}~~(\bpi^{>\mathcal{N}}\mathbf{U})_n \bydef ((\Pi^{>\mathcal{N}}U_1)_n, (\Pi^{>\mathcal{N}}U_2)_n)
\end{align}
for all $\mathbf{U} = \mathbf{u}_{n \in \mathbb{N}_0} \in \ell^2_{e}$.
 To obtain a numerical approximation of a pattern, we looked at the numerical solutions found in \cite{thomas_1d} and tried to recover them as a Fourier series representation using guesses of the form:
 \begin{equation*}
    \begin{cases}
        u_{1,0}(x) =  \beta_1\mathrm{sech}^2(\zeta_1 x) + \lambda_1 \\
        u_{2,0}(x) = \beta_2\mathrm{sech}(\zeta_2 x)^2 + \lambda_2
    \end{cases}
\end{equation*}
Then we tuned each parameter to replicate the localized solutions. $\zeta_i, i = 1,2$ controls the stiffness of each spike, and $\beta_i$ their amplitude. Once we found a promising candidate, we applied Newton's method to obtain a better approximate solution of \eqref{eq : F(U)=0 in X^l_e}. Then we applied the change of variable introduced in Section~\ref{sec : introduction} to get a candidate solution of \eqref{eq : gen_model}. From here, we compute a cosine Fourier sequence approximation of $\mathbf{u}_0 = (\bar{u}_{1,0}, \bar{u}_{2,0})$, $\mbf{U}_0 = \gamma(\bar{\mbf{u}}_0) \in \ell^2_e$. We note that $\mbf{U}_0$ satisfies the identity $\mbf{U}_0 = \bpi^{\leq N_0}\mbf{U}_0$ as it is a Fourier sequence with finitely many coefficients. At this point, we now have a vector representation on $\Omega_0$, $\mbf{u}_0 = \mathbf{\gamma}^\dagger(\mbf{U}_0)$ that we extend by $0$ to have $\mbf{u}_0 \in L^2_e$; however, such a function is not necessarily in $\mathcal{H}_e$. To ensure $\mbf{u}_0 \in \mathcal{H}_e$, we need to impose that $\mbf{u}_0$ and some of its derivatives to vanish at $\pm d$. Note that $\mathcal{H}_e$ contains the even functions of $H^2(\R) \times H^2(\R)$; hence, in order to have $\mbf{u}_0 \in \mathcal{H}_e$, we need to enforce 
$\mbf{u}_0(\pm d) = 0$ only since $\partial_x \mbf{u}_0 = 0$ is automatically satisfied thanks to the symmetry. We do so using the approach of Section 3.2 of \cite{blanco_cadiot_fassler_saddle}.
That is, we define the operator $\mathcal{P}: \ell^2_e \rightarrow \ell_e^2$,
\begin{equation*}
    (\mathcal{P}V)_n = \begin{cases}
        \sum_{n=1}^{N_0} 2V_n(-1)^{n+1} & \text{if } n = 0 \\
        V_n & \text{if } n \neq 0
    \end{cases}, ~~\mathbfcal{P}\mathbf{V} = (\mathcal{P}V_1, \mathcal{P}V_2)\label{def : operator P}
\end{equation*}
This allows us to define 
$$\overline{\mathbf{U}} = \mathbfcal{P}\mathbf{U}_0 \text{ and } \overline{\mathbf{u}} = \gamma^\dagger(\overline{\mathbf{U}}).$$ By construction, we have $\overline{\mathbf{u}}(\pm d)=0$ and  $\overline{\mathbf{u}} \in \mathcal{H}_e$. Also note that we built this in such a way that $\overline{\mathbf{U}} = \Pi^{\leq N_0}\overline{\mathbf{U}}$ by construction. As a result, we define our approximate solution as
\begin{equation}
    \overline{\mathbf{u}} = \gamma^\dagger(\overline{\mathbf{U}}) \in \mathcal{H}_e, ~  \text{ and } \overline{\mathbf{U}} = \Pi^{\leq N_0}\overline{\mathbf{U}}.
\end{equation}
\subsubsection{The Operator \texorpdfstring{$\mathbb{A}$}{A}}\label{sec : A}
In this section, we focus our attention to the operator $\mbb{A}: L^2_e \rightarrow \mathcal{H}_e$ required by Theorem \ref{th: radii polynomial}. $\mbb{A}$ is an approximate inverse of $D\mbb{F}(\overline{\mathbf{u}})$. We now present its construction following a similar argument as in \cite{gs_cadiot_blanco}. First we observe that $\mbb{L}: \mathcal{H}_e \rightarrow L^2_e$ is an isometric isomorphism. Therefore building $\mbb{A}: L^2_e \rightarrow \mathcal{H}_e$ is equivalent to building $\mbb{B}: L^2_e \rightarrow L^2_e$ approximating the inverse of $D\mbb{F}(\overline{\mathbf{u}})\mbb{L}^{-1}$, and setting $\mbb{A} \bydef \mbb{L}^{-1}\mbb{B}$. We will then construct $\mbb{B}$ using Lemma \ref{lem : gamma and Gamma properties} and a bounded linear operator on Fourier coefficients. 
First we construct $B^N$ a numerical approximate inverse to $\bpi^{\leq N}DF(\overline{\mathbf{U}})L^{-1}\bpi^{\leq N}$. We choose $B^N$ satisfying $B^N = \bpi^{\leq N} B^N \bpi^{\leq N}$, hence it admits a matrix representation. We now describe further its construction. Due to the form of $G$, we obtain
{\footnotesize\begin{align}
    &DF(\overline{\mathbf{U}})L^{-1} = 
     \begin{bmatrix}
        I_d + DG_{11}(\overline{\mathbf{U}})L_{\text{den}}^{-1}L_{22} -DG_{12}(\overline{\mathbf{U}})L_{\text{den}}^{-1}L_{21} & -DG_{11}(\overline{\mathbf{U}})L_{\text{den}}^{-1}L_{12} + DG_{12}(\overline{\mathbf{U}})L_{\text{den}}^{-1}L_{11} \\ 0 & I_d
        \end{bmatrix}\label{eq: B^N matrix}
\end{align}}
where 
\begin{align} 
L_{\text{den}} \bydef  L_{11}L_{22} - L_{12}L_{21}.\label{def : l_den}
\end{align}  
Since \eqref{eq: B^N matrix} is upper triangular by block, $B^N$ is chosen to be upper triangular by block as well. As in \cite{gs_cadiot_blanco}, we choose
\begin{align} 
B^N = \begin{bmatrix}
    B_{11}^N & B_{12}^N \\  0 & \Pi^{\leq N}
\end{bmatrix},~
B = \bpi^{>N} + B^N \bydef \begin{bmatrix}
    B_{11} & B_{12} \\ 0 & I_d
\end{bmatrix}\label{def : B_finite_infinite}
\end{align}
Finally, we get $\mbb{B}: L_e^2 \rightarrow L_e^2$ and $\mbb{A}: L_e^2 \rightarrow \mathcal{H}_e$ as 
\begin{equation}\label{eq:def_A}
    \mbb{B}\bydef \begin{bmatrix}
        \mbb{1}_{\mathbb{R} \backslash \Omega_0} & 0 \\ 0 & \mbb{1}_{\mathbb{R} \backslash \Omega_0}
    \end{bmatrix} + \mbf{\Gamma}^{\dagger}(B)= \begin{bmatrix}
        \mbb{B}_{11} & \mbb{B}_{12} \\ 0 & I_d
    \end{bmatrix},~ \mbb{A}\bydef \mbb{L}^{-1}\mbb{B}.
\end{equation}
where $\mbb{B}_{1j} = \Gamma^{\dagger}(B_{1j})$ for $j =1,2$. We will justify later on such a choice for $\mathbb{A}$ by computing explicitly the defect $\|I - \mathbb{A}D\mathbb{f}(\bar{\mathbf{u}})\|_{\mathcal{H}}$. We now recall a result from \cite{unbounded_domain_cadiot}.
\begin{lemma}\label{lem:bound_B}
    Let $\mbb{A}: L^2_e \rightarrow \mathcal{H}_e$ be as in \eqref{eq:def_A}. Then,
    \begin{equation}\label{eq:B_bound}
        \|\mbb{A}\|_{2, \mathcal{H}} = \|\mbb{B}\|_2 = \max\{1, \|B^N\|_2\}.
    \end{equation}
\end{lemma}
\begin{proof}
The proof can be found in \cite{unbounded_domain_cadiot}.
\end{proof}
With $\mathbb{A}$ now built, we must now discuss the nonpolynomial nonlinearity, $\mathbb{G}(\mathbf{u})$ and how we can compute it. This will be the topic of the next section.
\subsubsection{Computing the nonlinearity}\label{sec : nonlinearity}
In this section, we handle the nonpolynomial nonlinearity and its derivatives. In particular, our existence proof will require us to compute $G(\overline{\mathbf{U}})$ and $DG(\overline{\mathbf{U}})$. The problematic is that the Fourier coefficients corresponding to $g(\overline{\mathbf{U}})$ cannot be determined exactly as there are infinitely many nonzero Fourier coefficients. Furthermore, even finitely many coefficients cannot be computed exactly due to the nonpolynomial nature of the nonlinearity. To compute the Fourier coefficients of $g(\overline{\mathbf{U}})$, we first define 
{\footnotesize\begin{align}
    \overline{\psi} \bydef \lambda_3 \overline{u}_1^2 + \lambda_4 \overline{u}_1 - \nu_1 \overline{u}_1 \overline{u}_2 - \nu_1 \lambda_1 \overline{u}_2, \overline{\Psi} \bydef \gamma(\overline{\psi}),~\overline{\phi} \bydef \mathbb{1}_{\om} + \overline{u}_1 + \lambda_1 \mathbb{1}_{\om} + \nu_2(\overline{u}_1 + \lambda_1 \mathbb{1}_{\om})^2, ~ \overline{\Phi} \bydef \gamma(\phi).\label{def : phi}
\end{align}}
so that
\begin{align}
    \mathbb{g}(\overline{\mathbf{u}}) \bydef -\frac{\overline{\psi}}{\overline{\phi}} - \lambda_6 \overline{u}_1 - \lambda_7 \overline{u}_2, ~ g(\overline{\mathbf{U}}) \bydef \gamma(\mathbb{g}(\overline{\mathbf{u}})- \lambda_6 \overline{u}_1 - \lambda_7 \overline{u}_2) = -\overline{\Psi} * \overline{\Phi}^{-1} - \lambda_6\overline{U}_1 - \lambda_7 \overline{U}_2.
\end{align}
Note that $\overline{\Psi} = \Pi^{\leq 2N_0} \overline{\Psi}$ and $\overline{\Phi} = \Pi^{\leq 2N_0} \overline{\Phi}$. Let us discuss $D_{u_1} \mathbb{g}(\overline{\mathbf{u}})$ and $D_{u_2} \mathbb{g}(\overline{\mathbf{u}})$. First, let 
\begin{align}
     \overline{\psi}_{1} \bydef \mathbb{1}_{\om} + 2\nu_2 (\overline{u}_1+\lambda_1), ~ \overline{\psi}_{2} \bydef -\nu_1 \overline{u}_1 - \nu_1\lambda_1 \mathbb{1}_{\om},~\overline{\psi}_{3} \bydef 2\lambda_3 \overline{u}_1 + \lambda_4 \mathbb{1}_{\om} - \nu_1\lambda_1 \overline{u}_2
\end{align}
We then introduce $v_1, v_2 \in L^\infty(\mbb{R}) \cap L^2(\mbb{R})$
\begin{align}\label{def : v1v2}
    v_1 \bydef -(\overline{\phi}\overline{\psi}_{3} - \overline{\psi} \overline{\psi}_1)\overline{\phi}^{-2} -\lambda_6\mathbb{1}_{\om} \text{ and }v_2 \bydef - \overline{\psi}_2\overline{\phi}^{-1} - \lambda_7\mathbb{1}_{\om}.
\end{align}
Therefore, we have $D_{u_j} \mathbb{g}(\overline{\mathbf{u}}) = v_j$ for $j = 1,2$. We then define
{\small\begin{align} 
\overline{\Psi}_j \bydef \gamma(\overline{\psi}_j),~V_1\bydef \gamma(v_1) = -(\overline{\Phi} * \overline{\Psi}_1 - \overline{\Psi}*\overline{\Psi}_3)*\overline{\Phi}^{-2} - \lambda_6 e_0, V_2\bydef \gamma(v_2) = -\overline{\Psi}_2 *\overline{\Phi}^{-1} - \lambda_7 e_0,\label{def : V1V2}
\end{align}}
where $\overline{\Phi}^{-2} = \overline{\Phi}^{-1} * \overline{\Phi}^{-1}$. Now, let $\overline{\Phi}_{\mathrm{inv}}$ be a numerical approximation of $\overline{\Phi}^{-1}$ where
\begin{align}
    \overline{\Phi}_{\mathrm{inv}} = \Pi^{\leq 2N_0} \overline{\Phi}_{\mathrm{inv}}.\label{def : phiinv}
\end{align}
Such an approximation can be computed using the Fast Fourier Transform (FFT), and we refer the interested reader to \cite{marco_thesis} for the details. Our goal will be to control the error between $\overline{\Phi}^{-1}$ and $\overline{\Phi}_{\mathrm{inv}}$. We do so using the following lemma from \cite{maxime_paper_continuation,olivier_kevin_paper}.
\begin{lemma}\label{lem : phi bound}
Let $\overline{\Phi},\overline{\Phi}_{\mathrm{inv}} $  be defined as in \eqref{def : phi} and \eqref{def : phiinv} respectively. Let $\ell^1_{\tau}(\mathbb{N}_0)$ for some $\tau \geq 1$ be the Banach space defined in \eqref{ell1Gnu}. If $\|1 - \overline{\Phi}*\overline{\Phi}_{\mathrm{inv}}\|_{1,\tau} < 1$, then 
\begin{align}
    \|\overline{\Phi}_{\mathrm{inv}} - \overline{\Phi}^{-1}\|_{1,\tau} \leq \frac{\|\overline{\Phi}_{\mathrm{inv}}*(1-\overline{\Phi}*\overline{\Phi}_{\mathrm{inv}})\|_{1,\tau}}{1-\|1-\overline{\Phi}*\overline{\Phi}_{\mathrm{inv}}\|_{1,\tau}}.
\end{align}
\end{lemma}
\begin{proof}
The proof can be found in \cite{olivier_kevin_paper}.
\end{proof}
Lemma \ref{lem : phi bound} provides us with an estimate on the difference of the finite truncation $\overline{\Phi}_{\mathrm{inv}}$ and the full inverse element $\overline{\Phi}^{-1}$. Note that to control $\overline{\Phi}^{-2}$, we apply Lemma \ref{lem : phi bound} directly on the squared term.
\subsection{Computing the Bounds for Localized Patterns}\label{sec : Bounds of patterns}
In this section, we provide explicit formulas for the bounds $\mathcal{Y}_0,\mathcal{Z}_1,$ and $\mathcal{Z}_2$ in Theorem \ref{th: radii polynomial}.
Before we begin the computation of the bounds, we recall some preliminary notations from \cite{unbounded_domain_cadiot}. First, given $u \in L^\infty(\mathbb{R})$,  denote by 
\begin{align}\label{def : multiplication operator}
    \mathbb{u}  \colon L^2(\mathbb{R}) &\to L^2(\mathbb{R})\\
    v &\mapsto uv
\end{align}
 the multiplication operator associated to $u$. Similarly, given $U= (u_n)_{n \in \mathbb{N}_0} \in \ell^1(\mathbb{N}_0)$,  denote by
\begin{align}\label{def : discrete conv operator}
    \mathbb{U} : \ell^2(\mathbb{N}_0) &\to \ell^2(\mathbb{N}_0) \\
    V &\mapsto  U * V
\end{align}
 the discrete convolution operator associated to $U$. Now, we provide explicit formulas for the bounds of Theorem \ref{th: radii polynomial} in the following sections. In particular, each formula relies on finite dimensional computations (that is vector or matrix norms), which can be rigorously evaluated thanks to computer-assisted techniques (cf. \cite{julia_blanco_fassler}). This justifies how specific choices for the approximate objects $\bar{\mathbf{u}}$ and $\mathbb{A}$. We also introduce three useful notations
{\footnotesize\begin{align}\label{def : scr g V1 and V2}
    \mathscr{g}(\overline{\mathbf{U}}) \bydef -\overline{\Psi} * \overline{\Phi}_{\mathrm{inv}} - \lambda_6 \overline{U}_1 - \lambda_7 \overline{U}_2, ~ \mathscr{V}_1 \bydef -(\overline{\Phi}*\overline{\Psi}_3 - \overline{\Psi}*\overline{\Psi}_1)*\overline{\Phi}_{\mathrm{inv}}^2 - \lambda_6 e_0, ~ \mathscr{V}_2 \bydef -\overline{\Psi}_2 * \overline{\Phi}_{\mathrm{inv}} - \lambda_7 e_0.
\end{align}}
Essentially, \eqref{def : scr g V1 and V2} provides notation for the approximations of $g(\overline{\mathbf{U}}), V_1,$ and $V_2$ using $\overline{\Phi}_{\mathrm{inv}}$.
\subsubsection{The Bound \texorpdfstring{$\mathcal{Y}_0$}{Y0}}
In this section, we compute the bound $\mathcal{Y}_0$. Note that this computation is possibly thanks to the construction done in Section \ref{sec : nonlinearity} to handle the nonlinearity.
\begin{lemma}\label{lem : bound Y_0}
Let $\mathcal{Y}_0 > 0$ be defined as 
\begin{align}
    \mathcal{Y}_0 \bydef |\om|^{\frac{1}{2}} \left(\mathcal{Y}_{0,1} + \mathcal{Y}_{0,2}\right)
\end{align}
where 
\begin{align}
    &\mathcal{Y}_{0,1} \bydef \sqrt{\left\|B^N\left(L \overline{\mathbf{U}} + \begin{bmatrix}
        \mathscr{g}(\overline{\mathbf{U}}) \\ 0
    \end{bmatrix}\right)\right\|_{2}^2 + \left\|(\bpi^{\leq N_0}-\bpi^{\leq N})L \overline{\mathbf{U}} + \begin{bmatrix}
        (\Pi^{\leq 4N_0} - \Pi^{\leq N})(\overline{\Psi}*\overline{\Phi}_{\mathrm{inv}}) \\ 0
    \end{bmatrix}\right\|_{2}^2} \\
    &\mathcal{Y}_{0,2} \bydef  \max(1,\|B^N\|_{2}) \|
        \overline{\Psi}\|_{2}\frac{\|\overline{\Phi}_{\mathrm{inv}}*(1-\overline{\Phi}*\overline{\Phi}_{\mathrm{inv}})\|_{1}}{1-\|1-\overline{\Phi}*\overline{\Phi}_{\mathrm{inv}}\|_{1}}.
\end{align}
Then $\|\mathbb{A}\mathbb{F}(\overline{\mathbf{u}})\|_{\mathcal{H}} \leq \mathcal{Y}_0.$
\end{lemma}

\begin{proof}
We follow the same steps as those of Lemma 4.3 in \cite{gs_cadiot_blanco}. In particular, we combine \eqref{def : definition of the norm and inner product Hl}, the operator $\mathbb{A}$, and the properties of $\bgam$. Additionally, since $\mathrm{supp}(\overline{\mathbf{u}}) \subset \overline{\om}$, it follows that $\mathrm{supp}(\mathbb{g}(\overline{\mathbf{u}})) \subset \overline{\om}$. Hence, we can apply Parseval's identity and get
\begin{align}
    \|\mathbb{A}\mathbb{F}(\overline{\mathbf{u}})\|_{\mathcal{H}} = \|\mathbb{B}\mathbb{F}(\overline{\mathbf{u}})\|_{2} &= |\om|^{\frac{1}{2}}\|BF(\overline{\mathbf{U}})\|_2.
    \end{align}
Now, we introduce $\overline{\Phi}_{\mathrm{inv}}$ as
{\small\begin{align}
    \|BF(\overline{\mathbf{U}})\|_{2} =  \|B(L\overline{\mathbf{U}} + G(\overline{\mathbf{U}}))\|_{2} 
    &=\left\|B\left(L\overline{\mathbf{U}} + \begin{bmatrix}
        g(\overline{\mathbf{U}}) \\
        0
    \end{bmatrix}\right)\right\|_{2} \\
    &\leq \left\|B\left(L \overline{\mathbf{U}} + \begin{bmatrix}
        \mathscr{g}(\overline{\mathbf{U}}) \\ 0
    \end{bmatrix}\right)\right\|_{2} + \|B_{11}\|_{2} \|
        \overline{\Psi}*(\overline{\Phi}^{-1} - \overline{\Phi}_{\mathrm{inv}})\|_{2} \\
        &\leq \left\|B\left(L \overline{\mathbf{U}} + \begin{bmatrix}
        \mathscr{g}(\overline{\mathbf{U}}) \\ 0
    \end{bmatrix}\right)\right\|_{2} + \max(1,\|B_{11}^N\|_{2}) \|
        \overline{\Psi}\|_{2}\|\overline{\Phi}^{-1} - \overline{\Phi}_{\mathrm{inv}}\|_{1}\label{return to for Y01 and Y02}
\end{align}}
where the last step followed from Young's inequality for convolution (see \eqref{young_inequality}). We now apply Lemma \ref{lem : phi bound} to obtain
\begin{align}
    \max(1,\|B_{11}^N\|_{2}) \|
        \overline{\Psi}\|_{2}\|\overline{\Phi}^{-1} - \overline{\Phi}_{\mathrm{inv}}\|_{1} \leq  \max(1,\|B^N\|_{2}) \|
        \overline{\Psi}\|_{2}\frac{\|\overline{\Phi}_{\mathrm{inv}}*(1-\overline{\Phi}*\overline{\Phi}_{\mathrm{inv}})\|_{1}}{1-\|1-\overline{\Phi}*\overline{\Phi}_{\mathrm{inv}}\|_{1}} \bydef \mathcal{Y}_{0,2}\label{phi estimate in Y0}
\end{align}
For the remaining term, we follow the steps of Lemma 4.3 of \cite{blanco_cadiot_fassler_saddle} on \eqref{return to for Y01 and Y02} to get
    {\scriptsize\begin{align}
      \left\|B\left(L \overline{\mathbf{U}} + \begin{bmatrix}
        \mathscr{g}(\overline{\mathbf{U}}) \\ 0
    \end{bmatrix}\right)\right\|_{2}^2
    &= \left\|\bpi^{\leq N}B\left(L \overline{\mathbf{U}} + \begin{bmatrix}
        \mathscr{g}(\overline{\mathbf{U}}) \\ 0
    \end{bmatrix}\right)\right\|_{2}^2 + \left\|\bpi^{> N}\left(L \overline{\mathbf{U}} + \begin{bmatrix}
        \mathscr{g}(\overline{\mathbf{U}}) \\ 0
    \end{bmatrix}\right)\right\|_{2}^2 \\
    &= \left\|B^N\left(L \overline{\mathbf{U}} + \begin{bmatrix}
        \mathscr{g}(\overline{\mathbf{U}}) \\ 0
    \end{bmatrix}\right)\right\|_{2}^2 + \left\|(\bpi^{\leq N_0}-\bpi^{\leq N})L \overline{\mathbf{U}} + \begin{bmatrix}
        (\Pi^{\leq 4N_0} - \Pi^{\leq N})(\overline{\Psi}*\overline{\Phi}_{\mathrm{inv}}) \\ 0
    \end{bmatrix}\right\|_{2}^2 \\
    &\bydef \mathcal{Y}_{0,1}^2.\label{parseval in Y0}
    \end{align}}
Combining \eqref{phi estimate in Y0} and \eqref{parseval in Y0} yields the result.
\end{proof}
\subsubsection{The Bound \texorpdfstring{$\mathcal{Z}_2$}{Z2}}
Let us now compute the $\mathcal{Z}_2$ bound. Before doing so, we provide a useful lemma.
\begin{lemma}\label{lem : extra kappa estimates}
Let $\mathbf{u} \bydef (u_1,u_2),\mathbf{v} = (v_1,v_2),\mathbf{s} \bydef (s_1,s_2),\mathbf{y} = (y_1,y_2),\mathbf{z} = (z_1,z_2) \in \mathcal{H}$. Let $\kappa_1$ be defined as in Lemma \ref{corr : banach algebra}. Then, for all $i,j,k,o,p \in \{1,2\}$, it follows that
\begin{align}
    \|u_i v_j\|_{2} \leq \kappa_1 \|\mathbf{u}\|_{\mathcal{H}}\|\mathbf{v}\|_{\mathcal{H}}, ~ \|u_i v_j s_k\|_{2} \leq \kappa_1 \|l^{-1}\|_{\mathcal{M}_2} \|\mathbf{u}\|_{\mathcal{H}}\|\mathbf{v}\|_{\mathcal{H}}\|\mathbf{s}\|_{\mathcal{H}},\label{lem kappa first two}
   \end{align}
   {\scriptsize\begin{align}
    &\|u_i v_j s_k y_o\|_{2} \leq \kappa_1 \|l^{-1}\|_{\mathcal{M}_2}^2 \|\mathbf{u}\|_{\mathcal{H}}\|\mathbf{v}\|_{\mathcal{H}}\|\mathbf{y}\|_{\mathcal{H}} \|\mathbf{z}\|_{\mathcal{H}},~\|u_i v_j s_k y_o z_p\|_{2} \leq \kappa_1 \|l^{-1}\|_{\mathcal{M}_2}^3 \|\mathbf{u}\|_{\mathcal{H}}\|\mathbf{v}\|_{\mathcal{H}}\|\mathbf{s}\|_{\mathcal{H}}\|\mathbf{y}\|_{\mathcal{H}} \|\mathbf{z}\|_{\mathcal{H}}.\label{lem kappa third fourth}
\end{align}}
\end{lemma}
\begin{proof}
The proof of \eqref{lem kappa first two} follows from that of Lemma 2.2 of \cite{blanco_cadiot_fassler_saddle}. For \eqref{lem kappa third fourth}, observe that
\begin{align}
    \nonumber \|u_i v_j s_k y_o\|_{2}  \leq  \|u_i \|_{\infty} \|s_k\|_{\infty} \|y_o\|_{\infty} \|v_j\|_{2}, ~ \|u_i v_j s_k y_o z_p\|_{2}  \leq  \|u_i \|_{\infty} \|s_k\|_{\infty} \|y_o\|_{\infty} \|z_p\|_{\infty}\|v_j\|_{2}.
\end{align}
Then, we can again use the steps of Lemma 2.2 from \cite{blanco_cadiot_fassler_saddle} to obtain the results. 
\end{proof}
We now state the lemma for the $\mathcal{Z}_2$ bound.
\begin{lemma}\label{lem : Z2 patterns}
    Let $\kappa_0,\kappa_1 > 0$ be defined as in Lemma \ref{corr : banach algebra} . Moreover, define $q_1, q_2 \in L^\infty(\mbb{R})\cap L^2(\mbb{R})$  as
    \begin{align}\label{def : qjs}
        &q_1 \bydef \lambda_1 \nu_1 \nu_2 \mathbb{1}_{\om} + \nu_1 \nu_2 \overline{u}_1,~
        q_2 \bydef (-\nu_1 + \lambda_1^2 \nu_1 \nu_2) \mathbb{1}_{\om} + 2\lambda_1 \nu_1 \nu_2 \overline{u}_1 + \nu_1 \nu_2 \overline{u}_1^2
    \end{align}
    and their associated Fourier Coefficients $Q_1 = \gamma(q_1), Q_2 = \gamma(q_2)$. Next, let 
    \begin{align}
        \mathcal{Q} \bydef \{(0,1),(1,0),(1,1),(2,0),(2,1),(3,0),(4,0)\}\label{def : set Q}.
    \end{align}
Then define $q_{i,j} \in L^\infty(\mbb{R})\cap L^2(\mbb{R})$ and $\upsilon_{0,0}^{i,j} \in \mathbb{R}$ for $(i,j) \in \mathcal{Q}$ as in \eqref{def : qijs} and \eqref{def : rho and upsilon}. Let $\mathcal{Z}_2 : [0, \infty) \rightarrow (0, \infty)$ be defined as
    $$\mathcal{Z}_2(r) \bydef \kappa_0^2 \kappa_1 \left(\mathcal{Z}_{2,1} + \|l^{-1}\|_{\mathcal{M}_2} \mathcal{Z}_{2,2} r + \kappa_0^2 \sum_{(i,j) \in \mathcal{Q}} \mathcal{Z}_{2,i,j} \|l^{-1}\|_{\mathcal{M}_2}^{i + j - 1} r^{i + j - 1}\right)$$
    for all $r\geq 0$ and where
    \begin{align}
        &\mathcal{Z}_{2,1} \bydef \max\left\{|\lambda_1| \nu_1 \nu_2,\left(\|B_{11}^N Q_1^2 (B_{11}^N)^*\|_{2} + \|Q_1\|_{1}^2\right)^{\frac{1}{2}}\right\}, \\
        &\mathcal{Z}_{2,2} \bydef \max\left\{|-\nu_1 + \lambda_1^2 \nu_1 \nu_2|,\left(\|B_{11}^N Q_2^2 (B_{11}^N)^*\|_{2} + \|Q_2\|_{1}^2\right)^{\frac{1}{2}}\right\}, \\
        &\mathcal{Z}_{2,i,j}\bydef \max\left\{|\upsilon^{i,j}_{0,0}|, \left(\|B_{11}^N Q_{i,j}^2 (B_{11}^N)^*\|_{2} + \|Q_{i,j}\|_{1}^2\right)^{\frac{1}{2}}\right\} ~ \text{for all} ~ (i,j) \in \mathcal{Q}.
    \end{align} 
Then $\|\mathbb{A}(D\mathbb{F}(\mathbf{u}) - D\mathbb{F}(\bar{\mathbf{u}})) \|_{\mathcal{H}} \leq \mathcal{Z}_2(r)r$ for all $r>0$ and $\mathbf{u} \in \overline{B_r(\overline{\mathbf{u}})} \subset \mathcal{H}_e$.
\end{lemma}

\begin{proof}
The proof can be found in Appendix \ref{apen Z2}.
\end{proof}
 \subsubsection{The Bound \texorpdfstring{$\mathcal{Z}_1$}{Z1}}
For now, we provided explicit formulas for the bounds $\mathcal{Y}_0$ and $\mathcal{Z}_2$, and it remains to  estimate $\|I - \mathbb{A}D\mathbb{F}(\bar{\mathbf{u}})\|_{\mathcal{H}} \leq \mathcal{Z}_1$. This bound requires a bit more analysis, and we present its treatment in this section. For $j \in \{1,2\}$, we first define
\begin{align}
     V_j^N\bydef \Pi^{\leq N} \mathscr{V}_j,~DG^N(\overline{\mathbf{U}})\bydef \begin{bmatrix}
    \mathbb{V}_1^N & \mathbb{V}_2^N \\
    0 & 0
\end{bmatrix},~D\mbb{G}^N(\overline{\mathbf{u}})\bydef \Gamma^\dag(DG^N(\overline{\mathbf{U}})),~ v_j^N\bydef \gamma^\dag(V_j^N) 
\end{align}
where $\mathscr{V}_1$ and $\mathscr{V}_2$ are defined as in \eqref{def : scr g V1 and V2}, and $ \mbb{V}_1^N, \mbb{V}_2^N$ are the discrete convolution operators associated to $V_1^N, V_2^N$. This notation allows us to decompose both $D\mbb{G}(\overline{\mathbf{u}})$ and $DG(\overline{\mathbf{U}})$ using a truncation of size $N$. Additionally, observe that $v_1,v_2 \to 0$ as $|x|\to\infty$ based on how we defined $\mathbb{G}(\mathbf{u})$. We now state a lemma for the $\mathcal{Z}_1$ bound.
\begin{lemma}\label{lem : Z_full_1}
Let $\left(\mathcal{Z}_{u,k,j}\right)_{k \in \{1,2\}, j \in \{1,2,3,4\}}$ be bounds satisfying
\begin{align}
&\mathcal{Z}_{u,1,1} \geq \|\mathbb{1}_{\mathbb{R}\setminus\om} \mathbb{L}_{22}\mathbb{L}_{den}^{-1} \mathbb{v}_1^N\|_{2},~\mathcal{Z}_{u,2,1} \geq \|\mathbb{1}_{\om} (\mathbb{L}_{22}\mathbb{L}_{den}^{-1} - \Gamma^\dagger(L_{22}L_{\mathrm{den}}^{-1}))\mathbb{v}_1^N\|_{2}\\
&\mathcal{Z}_{u,1,2} \geq \|\mathbb{1}_{\mathbb{R}\setminus\om}\mathbb{L}_{21}\mathbb{L}_{den}^{-1}\mathbb{v}_2^N\|_{2},~\mathcal{Z}_{u,2,2} \geq \|\mathbb{1}_{\om} (\mathbb{L}_{21}\mathbb{L}_{den}^{-1} - \Gamma^\dagger(L_{21}L_{\mathrm{den}}^{-1}))\mathbb{v}_2^N\|_{2}
    \\
    &\mathcal{Z}_{u,1,3} \geq \|\mathbb{1}_{\mathbb{R}\setminus\om}\mathbb{L}_{12}\mathbb{L}_{den}^{-1}\mathbb{v}_1^N\|_{2},~\mathcal{Z}_{u,2,3} \geq \|\mathbb{1}_{\om}(\mathbb{L}_{12}\mathbb{L}_{den}^{-1} - \Gamma^\dagger(L_{12}L_{\mathrm{den}}^{-1}))\mathbb{v}_1^N\|_{2}\\
    &\mathcal{Z}_{u,1,4} \geq \|\mathbb{1}_{\mathbb{R}\setminus\om}\mathbb{L}_{22}\mathbb{L}_{den}^{-1}\mathbb{v}_2^N\|_{2},~\mathcal{Z}_{u,2,4} \geq \|\mathbb{1}_{\om}(\mathbb{L}_{11}\mathbb{L}_{den}^{-1} - \Gamma^\dagger(L_{11}L_{\mathrm{den}}^{-1}))\mathbb{v}_2^N\|_{2}
\end{align} 
Moreover, given $k \in \{1,2\}$, define  $\mathcal{Z}_{u,k} \bydef \sqrt{(\mathcal{Z}_{u,k,1} + \mathcal{Z}_{u,k,2})^2 + (\mathcal{Z}_{u,k,3}+\mathcal{Z}_{u,k,4})^2}$.
Then, it follows that $\mathcal{Z}_{u,1}$ and $\mathcal{Z}_{u,2}$ satisfy
\begin{align} 
\mathcal{Z}_{u,1} \geq \|\mathbb{1}_{\mathbb{R} \setminus \om} D\mathbb{G}^N(\overline{\mathbf{u}}) \mathbb{L}^{-1}\|_{2},~~
    \mathcal{Z}_{u,2} \geq \|\mathbb{1}_{\om}D\mathbb{G}^N(\overline{\mathbf{u}})(\bGam^\dagger(L^{-1}) - \mathbb{L}^{-1})\|_{2}.\label{def : Zu1 and Zu2}
\end{align}
Also define  $\mathcal{Z}_u  \bydef \sqrt{\mathcal{Z}_{u,1}^2 + \mathcal{Z}_{u,2}^2} $. Furthermore, let $\mathcal{Z}_b$ be a non-negative constant satisfying
\begin{align}
\mathcal{Z}_b &\geq \|I_d - B(I_d + DG^N(\overline{\mathbf{U}})L^{-1})\|_{2}.
\end{align}
We also introduce 
\begin{align}
    &\mathcal{Z}_{\infty,1} \bydef \|V_1^N - \mathscr{V}_1\|_{1} + \|\overline{\Phi}* \overline{\Psi}_3 - \overline{\Psi}* \overline{\Psi}_1\|_{1}\frac{\|\overline{\Phi}_{\mathrm{inv}}^2*(1-\overline{\Phi}^2*\overline{\Phi}_{\mathrm{inv}}^2)\|_{1}}{1-\|1-\overline{\Phi}^2*\overline{\Phi}_{\mathrm{inv}}^2\|_{1}}, \\
    &\mathcal{Z}_{\infty,2} \bydef \|V_2^N - \mathscr{V}_2\|_{1} + \|\overline{\Psi}_2\|_{1}\frac{\|\overline{\Phi}_{\mathrm{inv}}*(1-\overline{\Phi}*\overline{\Phi}_{\mathrm{inv}})\|_{1}}{1-\|1-\overline{\Phi}*\overline{\Phi}_{\mathrm{inv}}\|_{1}}.
\end{align}
Then, we define $\mathcal{Z}_{\infty} \bydef \sqrt{\mathcal{Z}_{\infty,1}^2 + \mathcal{Z}_{\infty,2}^2}$.
Finally, defining $\mathcal{Z}_1>0$ as
\begin{equation}\label{eq : first definition Z1}
    \mathcal{Z}_1 \bydef \mathcal{Z}_b + \max\left\{1, \|B_{11}^N\|_{2}\right\} \left(\mathcal{Z}_{u} +  \|\mathbb{L}^{-1}\|_{2} \mathcal{Z}_{\infty}\right),
\end{equation}
it follows that $ \|I_d -{\mathbb{A}}D\mathbb{F}(\overline{\mathbf{u}})\|_{\mathcal{H}} \leq \mathcal{Z}_1.$
\end{lemma}
\begin{proof}
The proof can be found in \cite{blanco_cadiot_fassler_saddle}. The difference between the formulas is that we do not have a $\sqrt{2}$ present in the definition of $\mathcal{Z}_{u,k}$ since we only have one nonlinear equation. Our main task will be to control the quantities $\|V_j^N - V_j\|_{1}$ for $j = 1,2$. For simplicity, let
    $\zeta \bydef -(\overline{\Phi}* \overline{\Psi}_3 - \overline{\Psi}* \overline{\Psi}_1)$.
Now, observe that
\begin{align}
    \|V_1^N - V_1\|_{1} &\leq \|V_1^N - \mathscr{V}_1\|_{1} + \|\mathscr{V}_1 - V_1\|_{1}\\
    &= \|V_1^N - \mathscr{V}_1\|_{1} + \|\zeta *\overline{\Phi}_{\mathrm{inv}}^2 - \zeta *\overline{\Phi}^{-2}\|_{1} \\
    &\leq \|V_1^N - \mathscr{V}_1\|_{1} + \|\zeta\|_{1}\| \overline{\Phi}_{\mathrm{inv}}^2 -  \overline{\Phi}^{-2}\|_{1} \\
    &\leq \|V_1^N - \mathscr{V}_1\|_{1} + \|\overline{\Phi}* \overline{\Psi}_3 - \overline{\Psi}* \overline{\Psi}_1\|_{1}\frac{\|\overline{\Phi}_{\mathrm{inv}}^2*(1-\overline{\Phi}^2*\overline{\Phi}_{\mathrm{inv}}^2)\|_{1}}{1-\|1-\overline{\Phi}^2*\overline{\Phi}_{\mathrm{inv}}^2\|_{1}}
\end{align}
which is the $\mathcal{Z}_{\infty,1}$ bound. Similarly, observe that
\begin{align}
    \|V_2^N - V_2\|_{1} &\leq \|V_2^N - \mathscr{V}_2\|_{1} + \|\mathscr{V}_2 - V_2\|_{1} \\
    &\leq \|V_2^N - \mathscr{V}_2\|_{1} + \|-\overline{\Psi}_2*\overline{\Phi}_{\mathrm{inv}} + \overline{\Psi}_2*\overline{\Phi}^{-1}\|_{1} \\
    &\leq \|V_2^N - \mathscr{V}_2\|_{1} + \|\overline{\Psi}_2\|_{1}\|\overline{\Phi}_{\mathrm{inv}} -\overline{\Phi}^{-1}\|_{1} \\
    &\leq \|V_2^N - \mathscr{V}_2\|_{1} + \|\overline{\Psi}_2\|_{1}\frac{\|\overline{\Phi}_{\mathrm{inv}}*(1-\overline{\Phi}*\overline{\Phi}_{\mathrm{inv}})\|_{1}}{1-\|1-\overline{\Phi}*\overline{\Phi}_{\mathrm{inv}}\|_{1}}
\end{align}
which is the $\mathcal{Z}_{\infty,2}$ bound. This concludes the proof. 
\end{proof}
Now, we must compute the bounds $\mathcal{Z}_b$ and $\mathcal{Z}_u$. We will begin with the bound $\mathcal{Z}_b$. This bound is usually referred to as the $Z_1$ bound for periodic solutions when performing CAPs (see Theorem \ref{th : radii polynomial theorem periodic} for instance). We have denoted it as $\mathcal{Z}_b$ here to avoid confusion with the $Z_1$ bound for periodic solutions in Section \ref{sec : periodic solutions}. The estimate for $\mathcal{Z}_b$ will be similar to that of Lemma 4.5 from \cite{blanco_cadiot_fassler_saddle}. We present it here.
\begin{lemma}\label{lem : Z1 periodic patterns}
    Let
    \begin{align}
        \mathcal{Z}_b \bydef \sqrt{\mathcal{Z}_{b,0}^2 + \mathcal{Z}_{b,1}^2 + \mathcal{Z}_{b,2}^2}
    \end{align}
    where
    \begin{align*}
        &\mathcal{Z}_{b,0} \bydef \|\bpi^{\leq2N} - \bpi^{\leq 3N}B(I_d + DG^N(\overline{\mbf{U}}) L^{-1})\bpi^{\leq2N}\|_{2}\\
        &\mathcal{Z}_{b,1} \bydef \left(\frac{l_{22}}{l_{\mathrm{den}}}\right)_{2N} \| V_1^N\|_{1} + \left(\frac{l_{21}}{l_{\mathrm{den}}}\right)_{2N} \|V_2^N\|_{1},~~ 
        \mathcal{Z}_{b,2} \bydef \left(\frac{l_{12}}{l_{\mathrm{den}}}\right)_{2N} \|V_1^N\|_{1} + \left(\frac{l_{11}}{l_{\mathrm{den}}}\right)_{2N} \|V_2^N\|_{1} 
    \end{align*}
    and $\left(\frac{l_{ij}}{l_{\mathrm{den}}}\right)_{2N} = \max_{n\in \mathbb{Z}\setminus I^{2N}}\frac{l_{ij} (\tilde n)}{l_{den}(\tilde{n})}$.
Then, it follows that $\|I_d - B(I_d + DG^N(\overline{\mathbf{U}}) L^{-1})\|_{2} \leq \mathcal{Z}_b$.
\end{lemma}

\begin{proof}
To begin, we introduce a truncation.
\begin{align}
    &\|I_d - B(I_d + DG^N(\overline{\mbf{U}}) L^{-1})\|_{2}^2 \\ &\leq \|\bpi^{\leq2N} - B(I_d + DG^N(\overline{\mbf{U}}) L^{-1})\bpi^{\leq2N}\|_{2}^2 +\|\bpi^{>2N} - B(I_d + DG^N(\overline{\mbf{U}}) L^{-1})\bpi^{>2N}\|_{2}^2 \\
    &=\|\bpi^{\leq2N} - \bpi^{\leq 3N}B(I_d + DG^N(\overline{\mbf{U}}) L^{-1})\bpi^{\leq2N}\|_{2}^2 + \|\bpi^{>2N} - B(I_d + DG^N(\overline{\mathbf{U}}) L^{-1})\bpi^{>2N}\|_{2}^2
\end{align}
where the last step followed from the definition of $B$. Also using the definition of $B$, we have $B \bpi^{>2N} = \bpi^{>2N}$ and $B\bpi^{>N} = \bpi^{>N}$. Additionally, since $V_1^N$ and $V_2^N$ are of size $N$, it follows that $DG^N(\overline{\mathbf{U}})\bpi^{>2N} = \bpi^{>N}DG^N(\overline{\mathbf{U}})\bpi^{>2N}$. Hence, we have
\begin{align}
    &\|\bpi^{>2N} - B \bpi^{>2N} - BDG^N(\overline{\mathbf{U}}) \bpi^{<2N} l^{-1}\|_{2} \\
    &= \|\bpi^{>2N} -  \bpi^{>2N} - B\bpi^{>N} DG^N(\overline{\mathbf{U}}) \bpi^{>2N} L^{-1}\|_{2} \\
    &= \|\bpi^{>N} DG^N(\overline{\mathbf{U}}) L^{-1}\bpi^{>2N}\|_{2} \\
    &= \left\|\begin{bmatrix}
        \Pi^{> N}(\mathbb{V}_1^N L_{22} L_{\mathrm{den}}^{-1} - \mathbb{V}_2^N L_{21} L_{\mathrm{den}}^{-1})\Pi^{> 2N}  & \Pi^{> N}(-\mathbb{V}_1^N  L_{12} L_{\mathrm{den}}^{-1} + \mathbb{V}_2^N L_{11} L_{\mathrm{den}}^{-1})\Pi^{>2N} \\0 & 0
    \end{bmatrix}\right\|_{2} \\
    &\leq\sqrt{\|\Pi^{> N}(\mathbb{V}_1^N L_{22} L_{\mathrm{den}}^{-1} - \mathbb{V}_2^N L_{21} L_{\mathrm{den}}^{-1})\Pi^{> 2N}\|_{2}^2 + \|\Pi^{> N}(-\mathbb{V}_1^N  L_{12} L_{\mathrm{den}}^{-1} + \mathbb{V}_2^N L_{11} L_{\mathrm{den}}^{-1})\Pi^{>2N}\|_{2}^2}
\end{align}
where we used the Cauchy-Schwarz inequality. We now examine each term
\begin{align}
    &\|\|\Pi^{> N}(\mathbb{V}_1^N L_{22} L_{\mathrm{den}}^{-1} - \mathbb{V}_2^N L_{21} L_{\mathrm{den}}^{-1})\Pi^{> 2N}\|_{2} \leq \left(\frac{l_{22}}{l_{\mathrm{den}}}\right)_{2N} \| V_1^N\|_{1} + \left(\frac{l_{21}}{l_{\mathrm{den}}}\right)_{2N} \|V_2^N\|_{1} \bydef \mathcal{Z}_{b,1}  \\
    &\|\Pi^{> N}(-\mathbb{V}_1^N  L_{12} L_{\mathrm{den}}^{-1} + \mathbb{V}_2^N L_{11} L_{\mathrm{den}}^{-1})\Pi^{>2N}\|_{2} \leq \left(\frac{l_{12}}{l_{\mathrm{den}}}\right)_{2N} \|V_1^N\|_{1} + \left(\frac{l_{11}}{l_{\mathrm{den}}}\right)_{2N} \|V_2^N\|_{1} \bydef \mathcal{Z}_{b,2}.
\end{align}
Therefore, we obtain the desired result.
\end{proof}
With $\mathcal{Z}_b$ now computed, we have the first part of the $\mathcal{Z}_1$ bound complete. We now turn our attention to the $\mathcal{Z}_u$ bound. Much of our results will be based on those of \cite{blanco_cadiot_fassler_saddle}. We recall some of the fundamental results.
\begin{lemma}\label{lem : ift_L_inverse}
    Assume that Assumption \ref{assumption : fourier} holds, then we have the following disjunction : 
    \begin{enumerate}
        \item If $(\nu\nu_3 - \lambda_6 + 1)^2 + 4\nu(\nu_3\lambda_6 - \nu_3 + \lambda_5 \lambda_7) < 0$, then  define $z_1, z_2$ as 
       
        \begin{align}\label{eq: def z1 z2 and y}
        z_1 \bydef 2\pi i y \text{ and } z_2 \bydef -2\pi i \bar{y}, \end{align}
    where
    \begin{align}
            y \bydef \frac{1}{2\pi} \left(\frac{\lambda_6 - 1 - \nu\nu_3+ i \sqrt{-(\nu\nu_3 - \lambda_6 + 1)^2 - 4\nu(\nu_3\lambda_6 - \nu_3 + \lambda_5 \lambda_7)}}{2\nu}\right)^{\frac{1}{2}}.
        \end{align}
        
        \normalsize
    \item If $(\nu\nu_3 - \lambda_6 + 1)^2 + 4\nu(\nu_3\lambda_6 - \nu_3 + \lambda_5 \lambda_7) \ge 0$ and \\ $\nu\nu_3 - \lambda_6 + 1 > \sqrt{(\nu\nu_3 - \lambda_6 + 1)^2 + 4\nu(\nu_3\lambda_6 - \nu_3 + \lambda_5 \lambda_7)}$, then define $z_1, z_2$ as 
  \begin{align}\label{eq : def z1 z2 second case}
    z_1 &\bydef  \left(\frac{ \nu\nu_3 - \lambda_6 + 1 + \sqrt{(\nu\nu_3 - \lambda_6 + 1)^2 + 4\nu(\nu_3\lambda_6 - \nu_3 + \lambda_5 \lambda_7)}}{2\nu}\right)^{\frac{1}{2}}\\
    z_2 &\bydef  \left(\frac{ \nu\nu_3 - \lambda_6 + 1 -  \sqrt{(\nu\nu_3 - \lambda_6 + 1)^2 + 4\nu(\nu_3\lambda_6 - \nu_3 + \lambda_5 \lambda_7)}}{2\nu}\right)^{\frac{1}{2}}.
\end{align}
\end{enumerate}
In both cases, for $d_1,d_2 \in \mathbb{R}$, we obtain that
    \begin{equation*}
\left|\mathcal{F}^{-1}\left(\frac{d_1|2\pi\xi|^2 + d_2}{l_{den}(\xi)}\right)(x)\right|\leq C_0e^{-a|x|},
    \end{equation*}
    with 
    \begin{align}
    a = \min\{\mathrm{Re}(z_1), \mathrm{Re}(z_2)\}\text{ and } C_0(d_1,d_2) = \frac{1}{|2\nu(z_1^2 - z_2^2)|} \left( |d_1 z_2|  +  \frac{|d_2|}{|z_2|} \right)\label{def : a and C0}
\end{align}
\end{lemma}

\begin{proof}
The proof follows from Lemma 4.6 of \cite{blanco_cadiot_fassler_saddle} with the corresponding parameters replaced.
\end{proof}
With Lemma \ref{lem : ift_L_inverse} available, we can compute the $\mathcal{Z}_u$ bound. We state the result in the next lemma.
\begin{lemma}\label{lem : Zu old exact}
Let $a$ be defined as in \eqref{def : a and C0}. Moreover, let $E \in \ell^2$ and $C(d),C_1,C_2,C_3,C_4 > 0$ be defined as
\begin{align}
    &E \bydef \gamma(\mathbb{1}_{\om} \cosh(2ax)),~ C(d) \bydef 4d + \frac{4e^{-ad}}{a(1-e^{-\frac{3ad}{2}})} + \frac{2}{a(1-e^{-2ad})},\\
    &C_1 \bydef C_0(-1,-\nu_3),~ C_2 \bydef C_0(0,\lambda_7),~C_3 \bydef C_0(0,\lambda_5),~\text{and}~ C_4 \bydef C_0(-\nu,\lambda_6 - 1).\label{def : E Cd Cj}
\end{align}
Then, let $(\mathcal{Z}_{u,k,j})_{k \in \{1,2\},j \in \{1,2,3,4\}} > 0$ be defined as
\begin{align}
&\mathcal{Z}_{u,1,1}^2 \bydef |\om| \frac{C_{1}^2e^{-2ad}}{a} (V_1^N, V_1^N * E)_2,~\mathcal{Z}_{u,2,1}^2 \bydef \mathcal{Z}_{u,1,1}^2 + e^{-4ad} C(d) C_1^2 |\om|(V_1^N,V_1^N * E)_2 \\
    &\mathcal{Z}_{u,1,2}^2 \bydef |\om| \frac{C_2^2 e^{-2ad}}{a} (V_2^N, V_2^N * E)_2,~\mathcal{Z}_{u,2,2}^2 \bydef \mathcal{Z}_{u,1,2}^2 + e^{-4ad} C(d) C_2^2 |\om|(V_2^N,V_2^N * E)_2 \\
    &\mathcal{Z}_{u,1,3}^2 \bydef |\om| \frac{C_3^2 e^{-2ad}}{a} (V_1^N, V_1^N * E)_2,~\mathcal{Z}_{u,2,3}^2 \bydef \mathcal{Z}_{u,1,3}^2 + e^{-4ad} C(d) C_3^2 |\om|(V_1^N,V_1^N * E)_2 \\
    &\mathcal{Z}_{u,1,4}^2 \bydef |\om| \frac{C_4^2e^{-2ad}}{a} (V_2^N,V_2^N*E)_2,~\mathcal{Z}_{u,2,4}^2 \bydef \mathcal{Z}_{u,1,4}^2 + e^{-4ad} C(d) C_4^2 |\om|(V_2^N,V_2^N * E)_2.
\end{align}
If $\mathcal{Z}_{u,1},\mathcal{Z}_{u,2},$ and $\mathcal{Z}_u$ are defined as in Lemma \ref{lem : Z_full_1}, then $\|\mathbb{B}D\mathbb{G}^N(\overline{\mathbf{u}})(\bGam^\dagger(L^{-1}) - \mathbb{L}^{-1})\|_{2} \leq  \mathcal{Z}_u$.
\end{lemma}
\begin{proof}
The proof can be found in \cite{blanco_cadiot_fassler_saddle}. In particular, it follows from Lemma 6.5 of \cite{unbounded_domain_cadiot}. Indeed, each $\mathcal{Z}_{u,k,j}$ can be computed using the results of the aforementioned lemma.
\end{proof}
\subsubsection{Constructive existence proofs of localized solutions}\label{sec : Proofs of patterns}
In this section, we apply the Newton-Kantorovich approach presented in Section \ref{sec : patterns} thanks to the explicit formulas derived in Section \ref{sec : Bounds of patterns}. In fact, the bounds are evaluated rigorously thanks to the \textit{Julia} packages \textit{IntervalArithmetic.jl} \cite{julia_interval} and \textit{RadiiPolynomial.jl} \cite{julia_olivier}. More specifically, the computer-assisted proofs details can be found on \cite{julia_blanco_fassler}. Note that existence of a unique solution is obtained in a ball around the approximate solution we use in the application of Theorem \ref{th: radii polynomial}. Hence, no conclusions regarding multiplicity are made. Now, we present our obtained constructive proofs of existence for localized patterns in \eqref{eq:thomas}. 

\begin{theorem}[\bf First Localized Solution in Thomas]\label{th : spike in thomas}
Let $\nu = 1.1664, \nu_1 = 8, \nu_2 = 1, \nu_3 = 0.28, \nu_4 = 39.1, \nu_5 = 150$. Moreover, let $r_0 \bydef 3 \times 10^{-10}$. Then there exists a unique solution $\tilde{\mbf{u}}$ to \eqref{eq : gen_model} in $\overline{B_{r_0}(\overline{\mbf{u}})} \subset \mathcal{H}_{e}$ and we have that $\|\tilde{\mbf{u}}-\overline{\mbf{u}}\|_{\mathcal{H}} \leq r_0$. 
\end{theorem}
\begin{proof}
Choose $N_0 = 500, N = 300, d = 30$. Then, we perform the full construction to build $\overline{\mathbf{u}} = \bgam^\dagger(\overline{\mathbf{U}})$. Then, we compute $ \overline{\mathbf{u}}$ as in Section \ref{sec : u0}. Next, we construct $B^N$, and use Lemma \ref{corr : banach algebra} to find
\begin{align}
    \|B_{11}^N\|_{2} \leq 20.885,~\kappa_0 \bydef 1.34, \kappa_1 \bydef 1.0234.
\end{align}
This allows us to compute the $\mathcal{Z}_2(r)$ bound defined in Section \ref{sec : Bounds of patterns}. 
Finally, using \cite{julia_blanco_fassler}, we choose $r_0 \bydef 3 \times 10^{-10}$ and define
\begin{align}
    \mathcal{Y}_0 \bydef 9.98 \times 10^{-11} \text{,}~\mathcal{Z}_{2}(r_0) \bydef 8.207\times 10^7    \text{,}~\mathcal{Z}_1 \bydef 0.02673 
    \end{align}
and prove that these values satisfy Theorem \ref{th: radii polynomial}. 
\end{proof}
\begin{figure}[H]
\centering
 \begin{minipage}{.5\linewidth}
  \centering\epsfig{figure=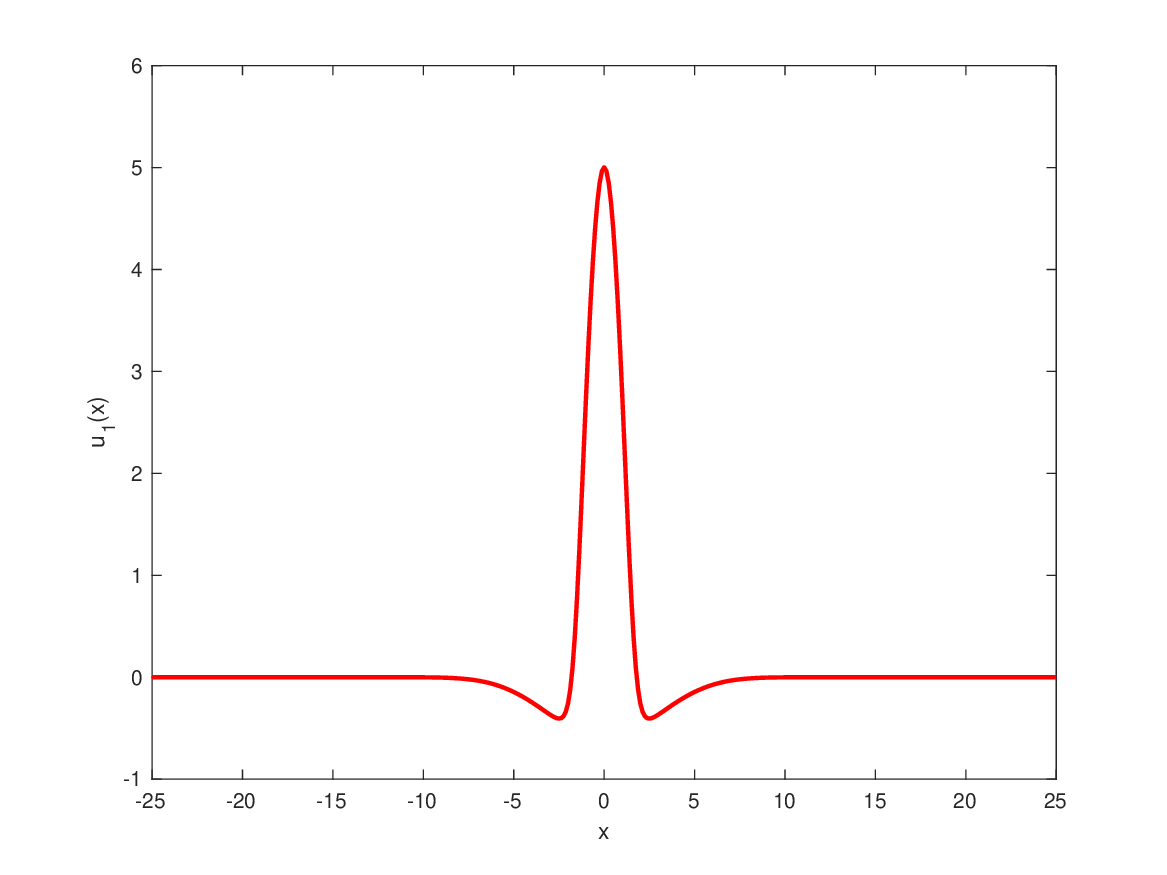,width=\linewidth}
  \end{minipage}%
 \begin{minipage}{.5\linewidth}
  \centering\epsfig{figure=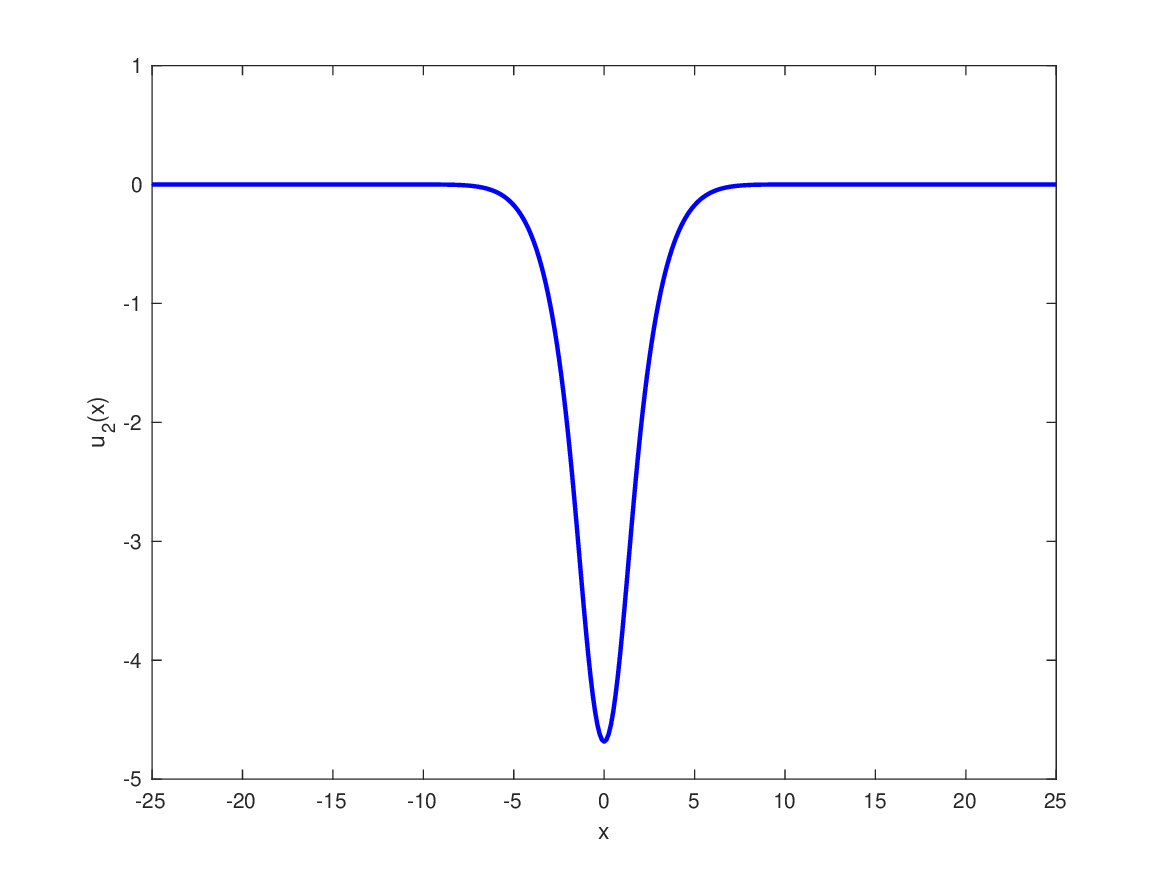,width=\linewidth}
 \end{minipage} 
 \caption{Plot of $ \overline{U}_1$ (L) and $\overline{U}_2$ (R) on $(-25,25)$ used in the proof of Theorem \ref{th : spike in thomas}.}\label{fig : th1}
 \end{figure}%
 \begin{theorem}[\bf Second Localized Solution in Thomas]\label{th : spike in thomas 2}
Let $\nu = 1.1664, \nu_1 = 8, \nu_2 = 1, \nu_3 = 0.28, \nu_4 = 39.10658, \nu_5 = 150$. Moreover, let $r_0 \bydef 2 \times 10^{-9}$. Then there exists a unique solution $\tilde{\mbf{u}}$ to \eqref{eq : gen_model} in $\overline{B_{r_0}(\overline{\mbf{u}})} \subset \mathcal{H}_{e}$ and we have that $\|\tilde{\mbf{u}}-\overline{\mbf{u}}\|_{\mathcal{H}} \leq r_0$. 
\end{theorem}
\begin{proof}
Choose $N_0 = 750, N = 300, d = 50$. The proof is obtained similarly to that of Theorem \ref{th : spike in thomas}.
In particular, we find
{\small\begin{align}
    \|B_{11}^N\|_{2} \leq 38.78,~\kappa_0 \bydef 0.9388,~ \kappa_1 \bydef 1.0163, ~\mathcal{Y}_0 \bydef  6.94 \times 10^{-10}\text{,}~\mathcal{Z}_{2}(r_0) \bydef 9.0442 \times 10^7     \text{,}~\mathcal{Z}_1 \bydef  0.192
    \end{align}}
and prove that these values satisfy Theorem \ref{th: radii polynomial}. 
\end{proof}
\begin{figure}[H]
\centering
 \begin{minipage}{.5\linewidth}
  \centering\epsfig{figure=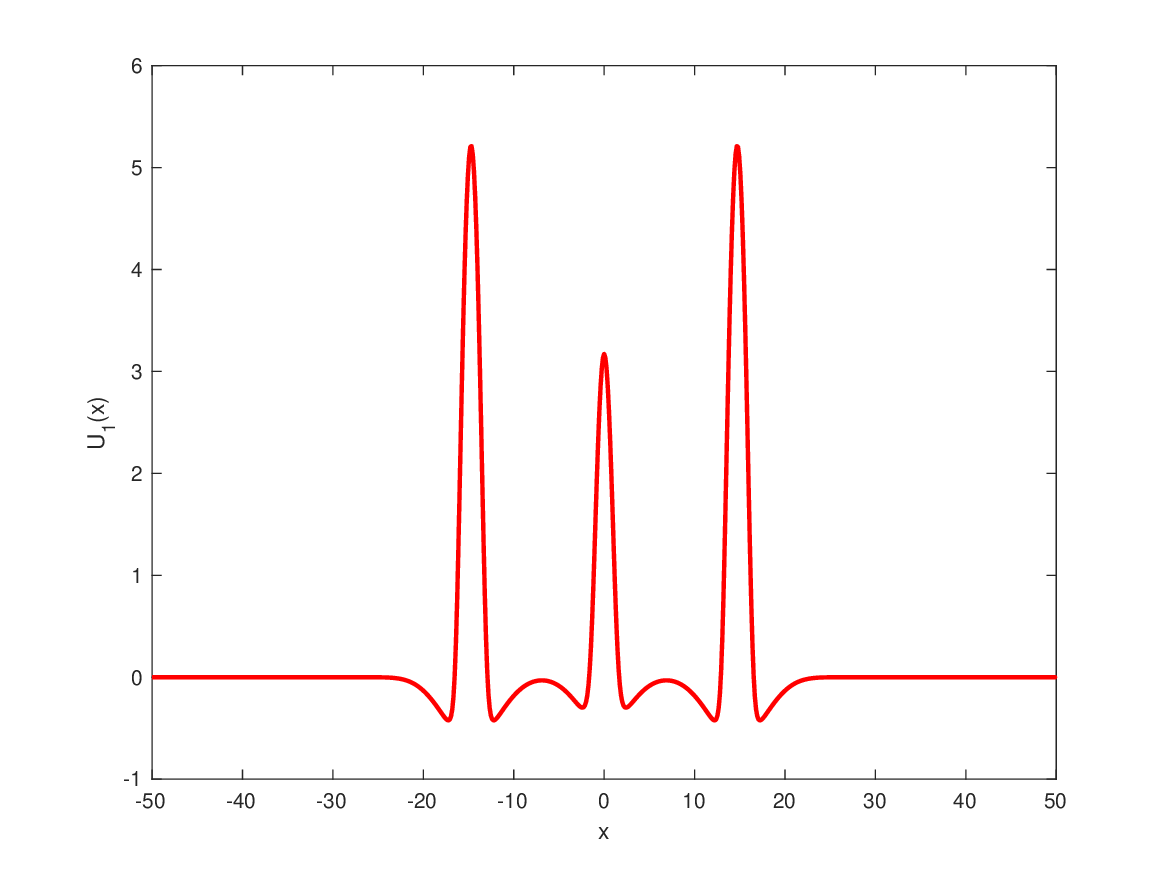,width=\linewidth}
  \end{minipage}%
 \begin{minipage}{.5\linewidth}
  \centering\epsfig{figure=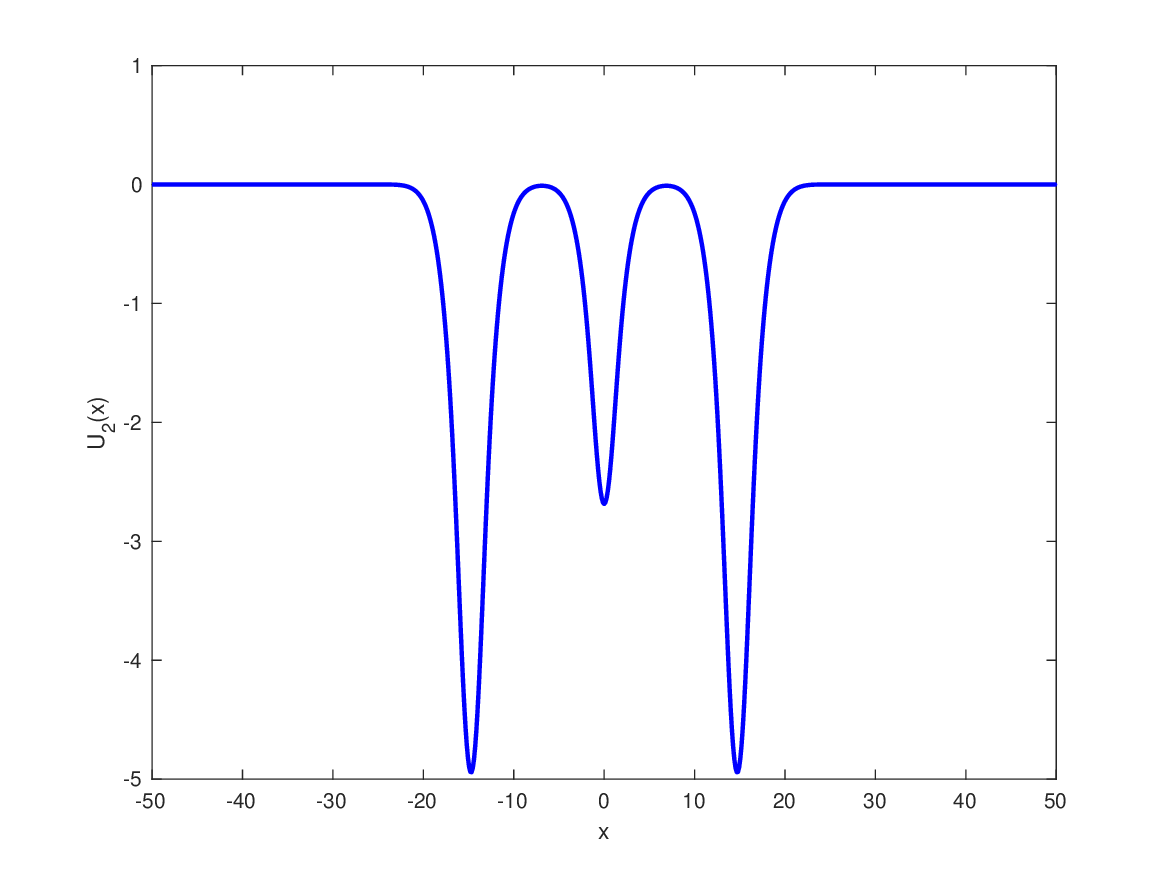,width=\linewidth}
 \end{minipage} 
 \caption{Plot of $ \overline{U}_1$ (L) and $\overline{U}_2$ (R) on $(-50,50)$ used in the proof of Theorem \ref{th : spike in thomas 2}.}\label{fig : th11}
 \end{figure}%
  \begin{theorem}[\bf Third Localized Solution in Thomas]\label{th : spike in thomas 3}
Let $\nu = 1.1664, \nu_1 = 8, \nu_2 = 1, \nu_3 = 0.28, \nu_4 = 39.10658, \nu_5 = 149.83$. Moreover, let $r_0 \bydef 4 \times 10^{-9}$. Then there exists a unique solution $\tilde{\mbf{u}}$ to \eqref{eq : gen_model} in $\overline{B_{r_0}(\overline{\mbf{u}})} \subset \mathcal{H}_{e}$ and we have that $\|\tilde{\mbf{u}}-\overline{\mbf{u}}\|_{\mathcal{H}} \leq r_0$. 
\end{theorem}
\begin{proof}
Choose $N_0 = 750, N = 300, d = 50$. The proof is obtained similarly to that of Theorem \ref{th : spike in thomas}.
In particular, we find
{\small\begin{align}
    \|B_{11}^N\|_{2} \leq 13.821,~\kappa_0 \bydef 1.34,~ \kappa_1 \bydef 1.087, ~\mathcal{Y}_0 \bydef  2.063 \times 10^{-9}\text{,}~\mathcal{Z}_{2}(r_0) \bydef 1.0143 \times 10^8     \text{,}~\mathcal{Z}_1 \bydef  0.17364
    \end{align}}
and prove that these values satisfy Theorem \ref{th: radii polynomial}. 
\end{proof}
\begin{figure}[H]
\centering
 \begin{minipage}{.5\linewidth}
  \centering\epsfig{figure=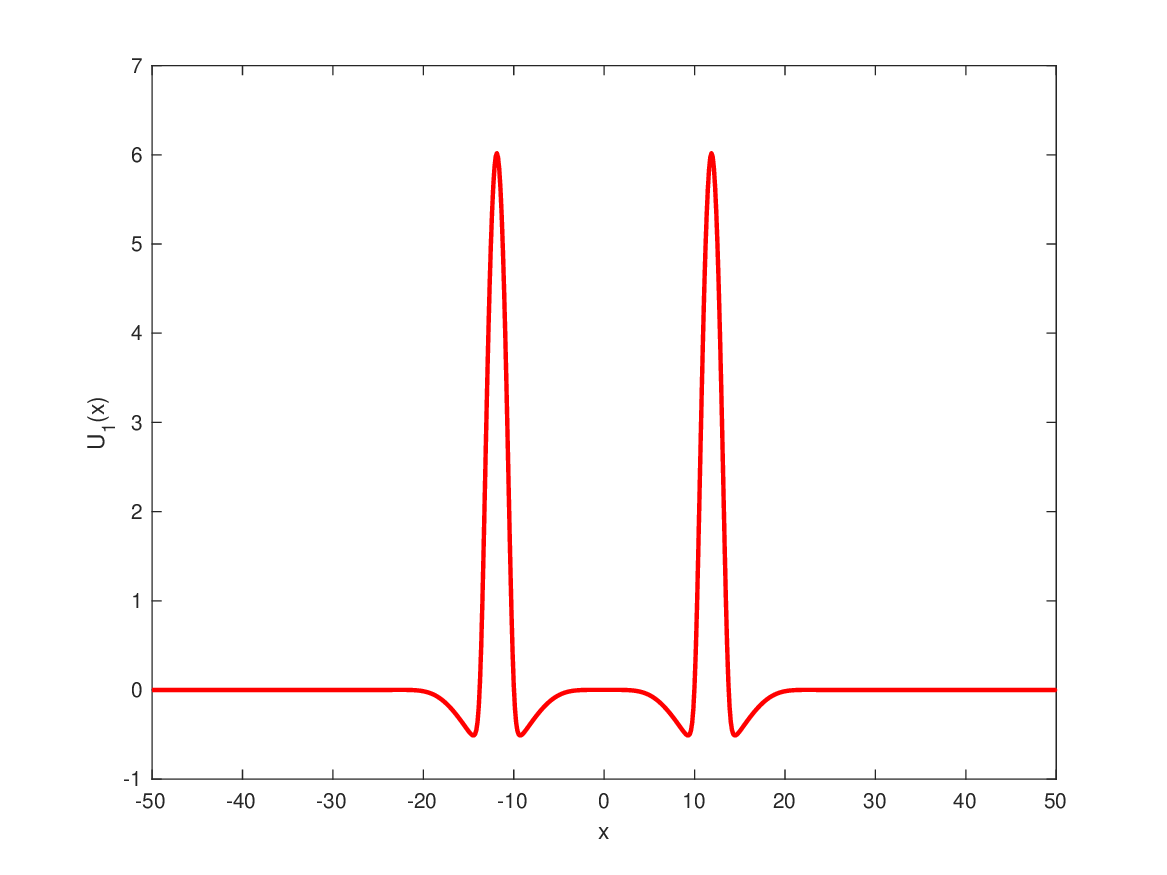,width=\linewidth}
  \end{minipage}%
 \begin{minipage}{.5\linewidth}
  \centering\epsfig{figure=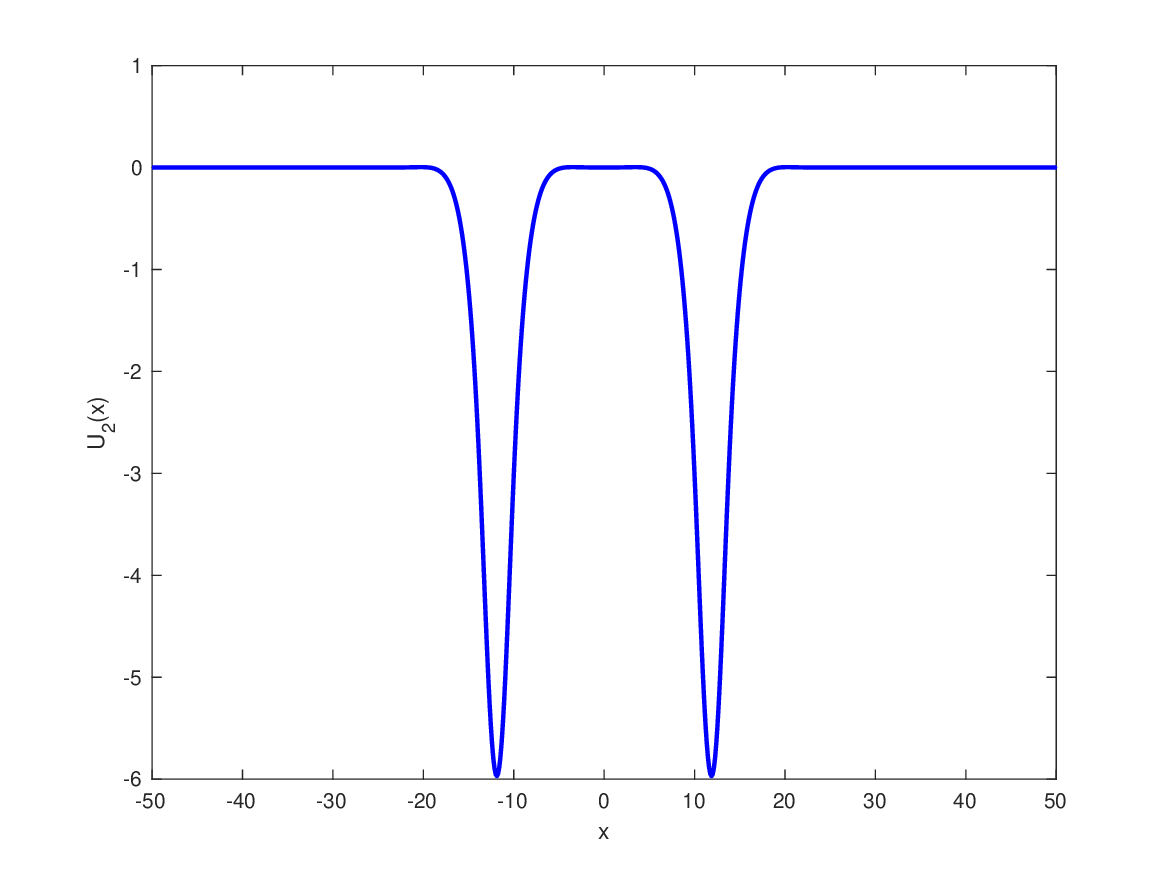,width=\linewidth}
 \end{minipage} 
 \caption{Plot of $ \overline{U}_1$ (L) and $\overline{U}_2$ (R) on $(-50,50)$ used in the proof of Theorem \ref{th : spike in thomas 3}.}\label{fig : th12}
 \end{figure}%
   \begin{theorem}[\bf Fourth Localized Solution in Thomas]\label{th : spike in thomas 4}
Let $\nu = 0.2116, \nu_1 = 8, \nu_2 = 1, \nu_3 = 0.28, \nu_4 = 21.3, \nu_5 = 64.5$. Moreover, let $r_0 \bydef 4 \times 10^{-11}$. Then there exists a unique solution $\tilde{\mbf{u}}$ to \eqref{eq : gen_model} in $\overline{B_{r_0}(\overline{\mbf{u}})} \subset \mathcal{H}_{e}$ and we have that $\|\tilde{\mbf{u}}-\overline{\mbf{u}}\|_{\mathcal{H}} \leq r_0$. 
\end{theorem}
\begin{proof}
Choose $N_0 = 1800, N = 650, d = 57$. The proof is obtained similarly to that of Theorem \ref{th : spike in thomas}.
In particular, we find
{\small\begin{align}
    \|B_{11}^N\|_{2} \leq 15.354,~\kappa_0 \bydef 1.34,~ \kappa_1 \bydef 5.8682, ~\mathcal{Y}_0 \bydef 1.1182 \times 10^{-11} \text{,}~\mathcal{Z}_{2}(r_0) \bydef 1.341 \times 10^{9}     \text{,}~\mathcal{Z}_1 \bydef  0.268
    \end{align}}
and prove that these values satisfy Theorem \ref{th: radii polynomial}. 
\end{proof}
\begin{figure}[H]
\centering
 \begin{minipage}{.5\linewidth}
  \centering\epsfig{figure=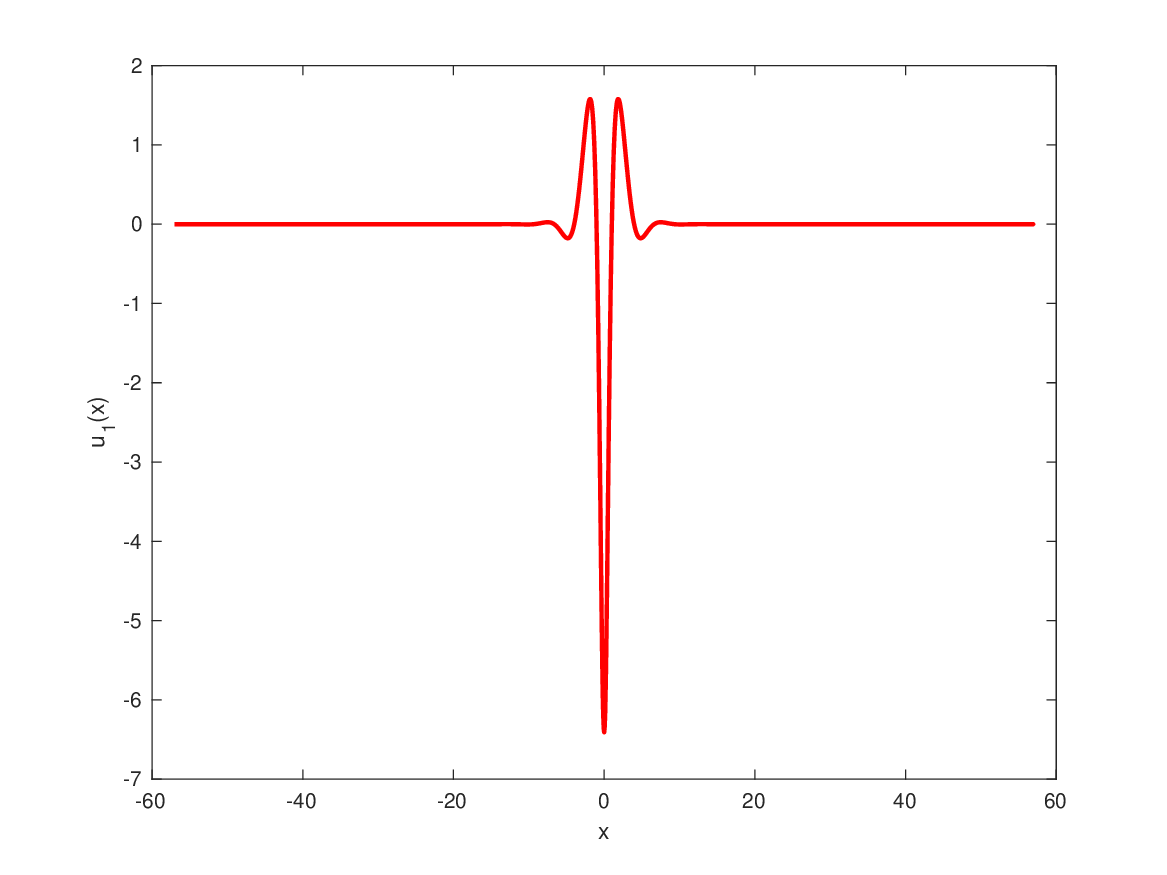,width=\linewidth}
  \end{minipage}%
 \begin{minipage}{.5\linewidth}
  \centering\epsfig{figure=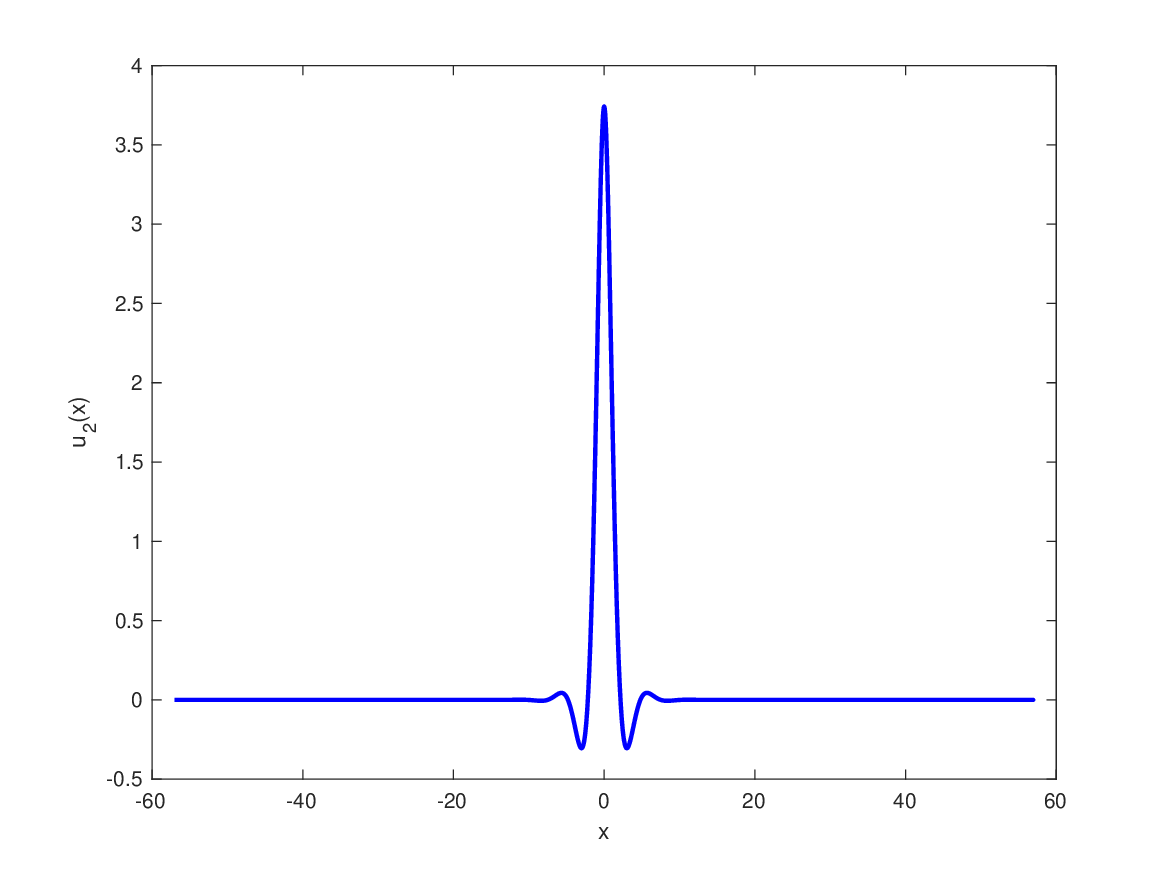,width=\linewidth}
 \end{minipage} 
 \caption{Plot of $ \overline{U}_1$ (L) and $\overline{U}_2$ (R) on $(-50,50)$ used in the proof of Theorem \ref{th : spike in thomas 4}.}\label{fig : th13}
 \end{figure}%
   \begin{theorem}[\bf Fifth Localized Solution in Thomas]\label{th : spike in thomas 5}
Let $\nu = 0.228, \nu_1 = 8, \nu_2 = 0.9498,$\\$ \nu_3 = 0.2799999, \nu_4 = 21, \nu_5 = 64.04$. Moreover, let $r_0 \bydef 4 \times 10^{-11}$. Then there exists a unique solution $\tilde{\mbf{u}}$ to \eqref{eq : gen_model} in $\overline{B_{r_0}(\overline{\mbf{u}})} \subset \mathcal{H}_{e}$ and we have that $\|\tilde{\mbf{u}}-\overline{\mbf{u}}\|_{\mathcal{H}} \leq r_0$. 
\end{theorem}
\begin{proof}
Choose $N_0 = 2100, N = 2100, d = 62$. The proof is obtained similarly to that of Theorem \ref{th : spike in thomas}.
In particular, we find
{\small\begin{align}
    \|B_{11}^N\|_{2} \leq 63.602,~\kappa_0 \bydef 1.3573,~ \kappa_1 \bydef 8.2362, ~\mathcal{Y}_0 \bydef 3.03 \times 10^{-11} \text{,}~\mathcal{Z}_{2}(r_0) \bydef 6.08 \times 10^{9}     \text{,}~\mathcal{Z}_1 \bydef  0.042
    \end{align}}
and prove that these values satisfy Theorem \ref{th: radii polynomial}. 
\end{proof}
\begin{figure}[H]
\centering
 \begin{minipage}{.5\linewidth}
  \centering\epsfig{figure=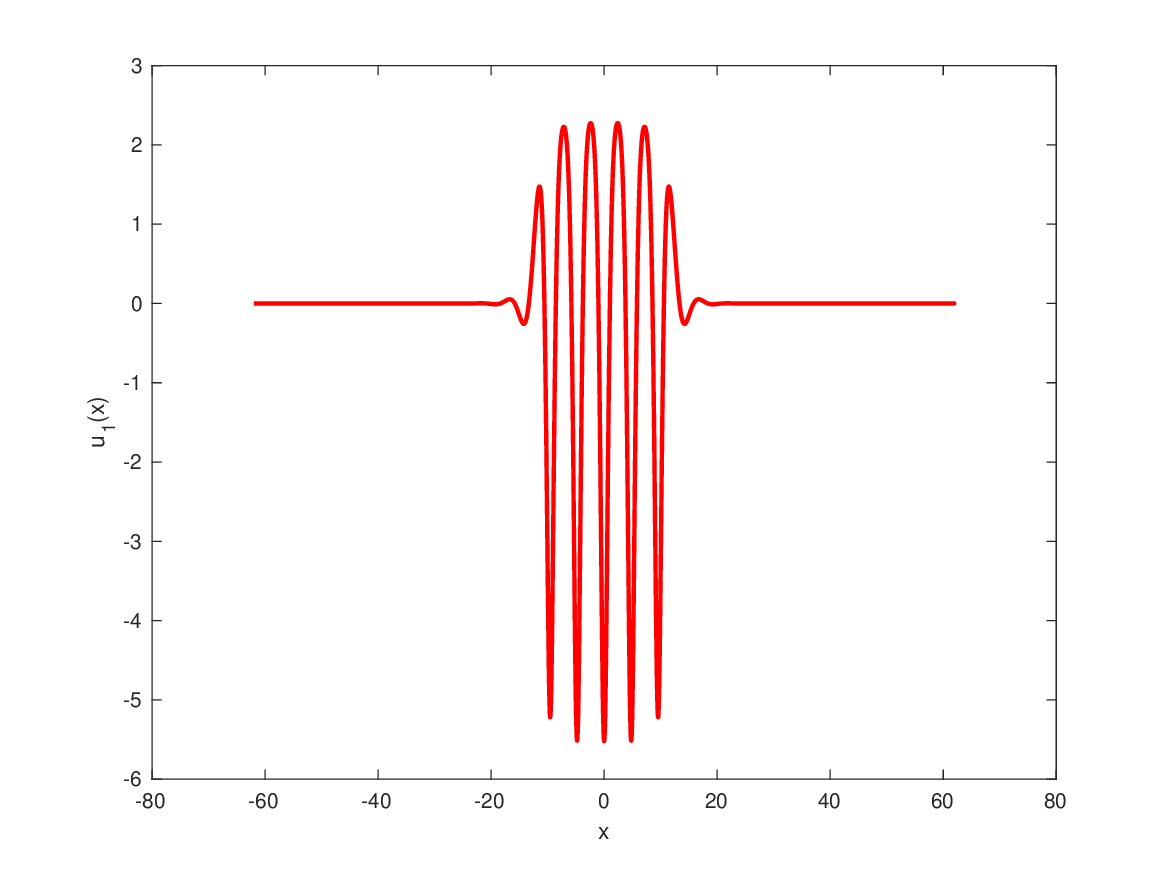,width=\linewidth}
  \end{minipage}%
 \begin{minipage}{.5\linewidth}
  \centering\epsfig{figure=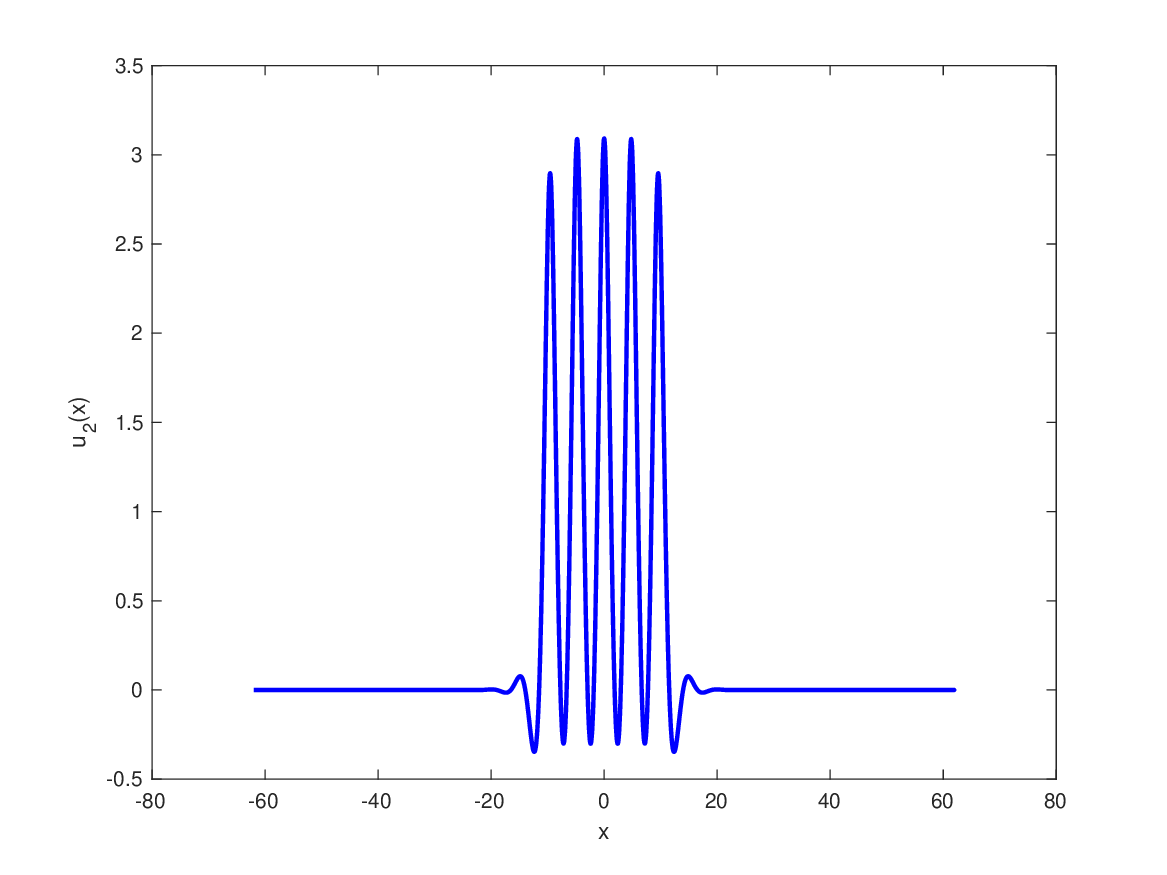,width=\linewidth}
 \end{minipage} 
 \caption{Plot of $ \overline{U}_1$ (L) and $\overline{U}_2$ (R) on $(-62,62)$ used in the proof of Theorem \ref{th : spike in thomas 5}.}\label{fig : th14}
 \end{figure}%
We have now completed our rigorous study of localized patterns in \eqref{eq:thomas}. We provided explicit estimates which allow for one to prove localized solutions in this model, which partially answers the question asked by the authors of \cite{thomas_1d}. Note that the approximate solutions we used in these proofs are chosen merely to demonstrate our approach. We did not choose them in an attempt to provide any additional insight on the dynamics or bifurcation structure to the Thomas model itself. The methodology to rigorously validate solutions by itself does not provide insight into the dynamics nor the bifurcation structure of the Thomas model. Finding an approximate solution for use in a rigorous proof whose result would have dynamical significance is outside the scope of this paper. To provide more rigorous results, we now move to periodic solutions. 
\section{Periodic Solutions}\label{sec : periodic solutions}
In \cite{thomas_1d}, the authors not only identified localized patterns. More specifically, in Figures 2.B and 2.E from the aforementioned paper, a branch of periodic solutions is presented. We wish to provide an approach for rigorously verifying such solutions in \eqref{eq:thomas}. As we are now proving periodic solutions, we no longer require the change of variable used to obtain \eqref{eq : gen_model}. Hence, we will introduce a more convenient zero finding problem with new notation. We let $U_1 = U$ and $U_2 \bydef \nu U - V$ to obtain
\begin{equation}
    \begin{split}
        &0 = \nu \Delta U_1 - U_1 + \nu_4 - \frac{\nu_1 \nu U_1^2 - \nu_1 U_1 *U_2}{1 + U_1 + \nu_2 U_1^2} \\
        &0 = \Delta U_2 - \nu_3 U_2 + (\nu_3 \nu-1) U_1 - \nu_3 \nu_5 + \nu_4.
    \end{split}
\end{equation}
We then define $\mathbf{U} \bydef (\mathbf{U}_1,\mathbf{U}_2)$ and introduce the \emph{periodic} zero finding problem as
\begin{align}\label{def : Fp}
    F_p(\mathbf{U}) \bydef L_p\mathbf{U} + G_p(\mathbf{U}), ~ L_p \bydef \begin{bmatrix}
        L_{p,11} & 0 \\
        L_{p,21} & L_{p,22}
    \end{bmatrix}, ~ G_p(\mathbf{U}) \bydef \begin{bmatrix}
        g_p(\mathbf{U}) \\
        -\nu_3 \nu_5 + \nu_4
    \end{bmatrix}
\end{align}
where
\begin{align}
    &L_{p,11} \bydef \nu \Delta - I_d, L_{p,21} \bydef (\nu_3 \nu-1) I_d ~L_{p,22} \bydef \Delta - \nu_3 I_d,~g_{p}(\mathbf{U}) \bydef \nu_4 - \frac{\nu_1 \nu U_1^2 - \nu_1 U_1 *U_2}{1 + U_1 + \nu_2 U_1^2}.
\end{align}
Recall that $\ell^1_{e,\tau} \bydef \ell^1_{\tau}(\mathbb{N}_0) \times \ell^1_{\tau}(\mathbb{N}_0)$ as in \eqref{def : norm on product space}. Given Fourier coefficients $U = (U_n)_{n \in \mathbb{Z}},  V= (V_n)_{n \in \mathbb{Z}}$ corresponding to the usual exponential Fourier series expansion, we define the discrete convolution 
\[
(U*V)_n = \sum_{k \in \mathbb{Z}} U_{n-k}V^*_k.
\]
 In the case of symmetric sequences $\ell^1_{\tau}(\mathbb{N}_0)$, we still denote $U*V$ the discrete convolution representing the product of two sequences for consistency. Now, let $S$ denote the Banach space that is the image of $F_p$ under $\ell^1_{e,\tau}$. That is, $F_p : \ell^1_{e,\tau} \to S$. Suppose we have $\overline{\mathbf{U}}$ such that $F_p(\overline{\mathbf{U}}) \approx 0$. Then, we state the following theorem whose proof can be found, for instance, in \cite{maxime_general}.
 \begin{theorem}\label{th : radii polynomial theorem periodic}
Let $R > 0 $. Let $A_p \in \mathcal{B}(S,\ell^1_{e,\tau})$ be injective. Moreover, let $Y_0, Z_1$ and  $Z_2 = Z_2(R)$ be non-negative constants such that
\begin{align}
    &\|A_pF_p(\overline{\mathbf{U}})\|_{1,\tau} \leq Y_0 \\
    &\|I_d - A_pDF_p(\overline{\mathbf{U}})\|_{\mathcal{B}(\ell^1_{e,\tau})} \leq Z_1  \\
    &\underset{\mathbf{U} \in B_R(\overline{\mathbf{U}})}{\sup}\|A_pDF^2_p(\mathbf{U})\|_{\mathcal{B}(\ell^1_{e,\tau},\mathcal{B}(\ell^1_{e,\tau}))} \leq Z_2.
\end{align}
If 
\begin{equation}\label{condition radii polynomial periodic}
    2Y_0 Z_2 \leq (1- Z_1)^2, ~\text{and}~ Z_1 < 1,
 \end{equation}
then, for any $r > 0$ such that
\begin{align}
    \frac{1- Z_1 - \sqrt{(1 - Z_1)^2 - 2Y_0 Z_2}}{Z_2} \leq r < \min\left(\frac{1-Z_1}{Z_2},R\right)
\end{align}
there exists a unique $\widetilde{\mathbf{U}} \in \overline{B_r(\overline{\mathbf{U}})} \subset \ell^1_{e,\tau}$ such that $F(\widetilde{\mathbf{U}})=0$, where $B_r(\overline{\mathbf{U}})$ is the open ball of $\ell^1_{e,\tau}$ centered at $\overline{\mathbf{U}}$ with radius $r$.
 \end{theorem}

 With Theorem \ref{th : radii polynomial theorem periodic} available, let us now discuss the numerical aspects of our approach.
\subsection{Numerical Aspects of Periodic Solutions}\label{sec : numerical aspects of periodic solutions}
In this section, we will compute $\overline{\mathbf{U}}, A_p,$ and the nonlinear terms. As far as the numerical aspects of this approach are concerned, we build $\overline{\mathbf{U}}$ in the same way as discussed in Section \ref{sec : u0}, the only difference is that we do not apply the operator $\mathcal{P}$ as it is not necessary for proving periodic solutions. We then perform numerical continuation on the solution obtained using the approach described in Section \ref{sec : u0} to obtain our candidate solutions. The construction of the approximate inverse $A_p$ is what changes the most. We will choose $A_p$ to have a finite part, which can be stored on the computer as a matrix, and an infinite tail which we will control theoretically. In particular, 
\begin{align}
    A_p^N \approx (\bpi^{\leq N} DF_p(\overline{\mathbf{U}})\bpi^{\leq N})^{-1}.
\end{align}
As $A_p^N$ is finite, we can store it on the computer as a matrix. Next, let $\mathbf{H} \in \ell^1_{e,\tau}$. For the tail, rather than considering the full $DF_p(\overline{\mathbf{U}})^{-1}$, we only choose the inverse of the linear part, $L_p$. Hence,
\begin{align}
    (A_p\mathbf{H})_n \bydef \begin{cases}
        (A_p^N \bpi^{\leq N} \mathbf{H})_{n} & n \in I^N \\
        \begin{bmatrix}
            \frac{(H_1)_n}{\nu \tilde{n}^2 + 1} \\
            \frac{(\nu_3 \nu-1)(H_1)_n}{(\nu\tilde{n}^2 + 1)(\tilde{n}^2+\nu_3)}+\frac{(H_2)_n}{\tilde{n}^2+\nu_3}
        \end{bmatrix} & n \in \mathbb{N}_0 \setminus I^N
    \end{cases}.
\end{align}
With this definition, observe that
\begin{align}
    \|\bpi^{\leq N} A_p\|_{\mathcal{B}(\ell^1_{e,\tau})} &\leq \max_{n \in \mathbb{Z}^2 \setminus I^N} \left(\frac{1}{|\nu \tilde{n}^2 + 1|}+\left|\frac{\nu_3 \nu-1}{(\nu \tilde{n}^2 + 1)(\tilde{n}^2 + \nu_3)}\right| + \frac{1}{|\tilde{n}^2+\nu_3|}\right)  \bydef \mathcal{L}_{\infty}.
\end{align}
We now discuss the nonlinearities. We follow a similar strategy to Section \ref{sec : nonlinearity} but defined directly on sequences as
\begin{align}
    \overline{\Psi}_p \bydef -(\nu_1 \nu \overline{U}_1^2 - \nu_1 \overline{U}_1* \overline{U}_2), ~ \overline{\Phi}_p \bydef 1 + \overline{U}_1 + \nu_2 \overline{U}_1^2
\end{align}
so that $g_p(\overline{\mathbf{U}}) = \nu_4 + \overline{\Psi}_p * \overline{\Phi}_p^{-1}$.
We then define
\begin{align}
    \overline{\Psi}_{p,1} \bydef -\nu_1(-\overline{U}_2 + 2\nu \overline{U}_1 + \nu \overline{U}_1^2 + \nu_2 \overline{U}_1^2*\overline{U}_2), ~ \overline{\Psi}_{p,2} \bydef \nu_1 \overline{U}_1
\end{align}
so that we can define $V_{p,1},V_{p,2} \in \ell^1_{e,\tau}$ as
\begin{align}
    V_{p,1} \bydef \overline{\Psi}_{p,1} * \overline{\Phi}_p^{-2}, ~ V_{p,2} \bydef  \overline{\Psi}_{p,2} * \overline{\Phi}_p^{-1}.
\end{align}
As before, $\overline{\Phi}_{p,\mathrm{inv}}$ is an approximation of $\overline{\Phi}_p^{-1}$ using the FFT (see \cite{marco_thesis} for the computational details). Additionally, we let
\begin{align}\label{def : scr gp V1p and V2p}
    \mathscr{g}_p(\overline{\mathbf{U}}) \bydef \nu_4 + \overline{\Psi}_p * \overline{\Phi}_{p,\mathrm{inv}}, ~ \mathscr{V}_{p,1} \bydef \overline{\Psi}_{p,1} * \overline{\Phi}_{p,\mathrm{inv}}^2, ~ \mathscr{V}_{p,2} \bydef \overline{\Psi}_{p,2} * \overline{\Phi}_{p,\mathrm{inv}}
\end{align}We now recall an additional lemma from \cite{maxime_paper_continuation,olivier_kevin_paper}, which we will now require.
\begin{lemma}\label{lem : ball phi bound}
Let $r > 0$ and $\overline{\Phi}_p, \overline{\Phi}_{p,\mathrm{inv}} \in \ell^1_{\tau}(\mathbb{N}_0)$. If $\|1 - \overline{\Phi}_p * \overline{\Phi}_{p,\mathrm{inv}}\|_{1,\tau} + r \|\overline{\Phi}_{p,\mathrm{inv}}\|_{1,\tau} < 1$, then
\begin{align}
    \sup_{\Phi \in B_r(\overline{\Phi}_p)} \|\Phi^{-1}\|_{1,\tau} \leq \frac{\|\overline{\Phi}_{p,\mathrm{inv}}\|_{1,\tau}}{1 - \|1-\overline{\Phi}_p*\overline{\Phi}_{p,\mathrm{inv}}\|_{1,\tau} - r \|\overline{\Phi}_{p,\mathrm{inv}}\|_{1,\tau}}.
\end{align}
\end{lemma}
\begin{proof}
The proof can be found in \cite{olivier_kevin_paper}.
\end{proof}
The previous lemma allows us to obtain an upper bound for the inverse of an element, which will be useful when computing the $Z_2$ bound. Finally, we modify Lemma \ref{lem : phi bound} with a stronger variant. 
\begin{lemma}\label{lem : phi bound stronger}
Let $\overline{\Phi}_p,\overline{\Phi}_{p,\mathrm{inv}} \in \ell^1_{\tau}(\mathbb{N}_0) $. Let $\ell^1_{\tau}(\mathbb{N}_0)$ for some $\tau \geq 1$ be the Banach space defined in \eqref{ell1Gnu}.
Let $\Xi \in \ell^1_{\tau}(\mathbb{N}_0)$ and $\mathcal{A} \in \mathcal{B}(S,\ell^1_{e,\tau})$.
If $\|1 - \overline{\Phi}_p\overline{\Phi}_{p,\mathrm{inv}}\|_{1,\tau} < 1$, then 
\begin{align}
    \left\|\mathcal{A}\begin{bmatrix} \Xi*(\overline{\Phi}_{p,\mathrm{inv}} - \overline{\Phi}_p^{-1}) \\ 0\end{bmatrix}\right\|_{1,\tau} \leq \frac{\left\|\mathcal{A}\begin{bmatrix}\Xi*\overline{\Phi}_{p,\mathrm{inv}}*(1-\overline{\Phi}_p*\overline{\Phi}_{p,\mathrm{inv}}) \\ 0 \end{bmatrix}\right\|_{1,\tau}}{1-\|1-\overline{\Phi}_p*\overline{\Phi}_{p,\mathrm{inv}}\|_{1,\tau}}.
\end{align}
\end{lemma}
\begin{proof}
The proof is similar to that of Lemma \ref{lem : phi bound}, we simply keep extra terms inside the estimate.
\end{proof}
With $\overline{\mathbf{U}}, g_p(\overline{\mathbf{U}}), V_{p,1}, V_{p,2},$ and $A_p$ now constructed, let us compute the needed bounds to prove periodic solutions.
 \subsection{Computing the Bounds for Periodic Solutions}\label{sec : Bounds of periodic solutions}
 In this section, we compute the bounds for periodic solutions. That is, we compute $Y_0, Z_1,$ and $Z_2(r)$ satisfying Theorem \ref{th : radii polynomial theorem periodic}. We now begin with $Y_0$.
 \begin{lemma}\label{lem : bound Y_0 periodic}
Let $Y_0 > 0$ be defined as 
\begin{align}
    Y_0 \bydef Y_{0,1} + Y_{0,2}
\end{align}
where
{\footnotesize\begin{align}
    &Y_{0,1} \bydef \left\| A_p^N \left(L_p \overline{\mathbf{U}} + \begin{bmatrix}
        \mathscr{g}_p(\overline{\mathbf{U}}) \\ -\nu_3 \nu_5 + \nu_4
    \end{bmatrix}\right)\right\|_{1,\tau} + \mathcal{L}_{\infty}\left\|(\bpi^{\leq N_0} - \bpi^{\leq N}) L_p\overline{\mathbf{U}} + \begin{bmatrix}
        (\Pi^{\leq 4N_0} - \Pi^{\leq N}) (\overline{\Psi}_p * \overline{\Phi}_{p,\mathrm{inv}}) \\ 0
    \end{bmatrix}\right\|_{1,\tau} \end{align}}
    {\footnotesize\begin{align}
    &Y_{0,2} \bydef   \frac{\left\|A_p^N \begin{bmatrix} \overline{\Psi}_p*\overline{\Phi}_{p,\mathrm{inv}}*(1-\overline{\Phi}_p*\overline{\Phi}_{p,\mathrm{inv}}) \\ 0 \end{bmatrix}\right\|_{1,\tau} + \mathcal{L}_{\infty} \|(\Pi^{\leq 8N_0} - \Pi^{ \leq N})(\overline{\Psi}_p*\overline{\Phi}_{p,\mathrm{inv}}*(1 - \overline{\Phi}_p*\overline{\Phi}_{p,\mathrm{inv}}))\|_{1,\tau}}{1-\|1-\overline{\Phi}_p*\overline{\Phi}_{p,\mathrm{inv}}\|_{1,\tau}}.
\end{align}}
Then, it follows that $\|A_p F_p(\overline{\mathbf{U}})\|_{1,\tau} \leq Y_0$.
 \end{lemma}
 \begin{proof}
The proof follows the steps of Lemma \ref{lem : bound Y_0} after applying Parseval's identity. More specifically, since we are directly working with sequences, we can immediately introduce $\overline{\Phi}_{p,\mathrm{inv}}$. We also use Lemma \ref{lem : phi bound stronger} to get
{\small\begin{align}
    \|A_p F_p(\overline{\mathbf{U}})\|_{1,\tau} &= \left\| A_p \left(L_p \overline{\mathbf{U}} + \begin{bmatrix}
        g_p(\overline{\mathbf{U}}) \\ -\nu_3 \nu_5 + \nu_4
    \end{bmatrix}\right)\right\|_{1,\tau} \\
    &\leq \left\| A_p \left(L_p \overline{\mathbf{U}} + \begin{bmatrix}
        \mathscr{g}_p(\overline{\mathbf{U}}) \\ -\nu_3 \nu_5 + \nu_4
    \end{bmatrix}\right)\right\|_{1,\tau} +  \left\|A_p\begin{bmatrix} \overline{\Psi}_p*(\overline{\Phi}_p^{-1} -\overline{\Phi}_{p,\mathrm{inv}}) \\ 0 \end{bmatrix}\right\|_{1,\tau} \\
    &\leq \left\| A_p \left(L_p \overline{\mathbf{U}} + \begin{bmatrix}
        \mathscr{g}_p(\overline{\mathbf{U}}) \\ -\nu_3 \nu_5 + \nu_4
    \end{bmatrix}\right)\right\|_{1,\tau} +   \frac{\left\|A_p \begin{bmatrix} \overline{\Psi}_p*\overline{\Phi}_{p,\mathrm{inv}}*(1-\overline{\Phi}_p*\overline{\Phi}_{p,\mathrm{inv}}) \\ 0 \end{bmatrix}\right\|_{1,\tau}}{1-\|1-\overline{\Phi}_p*\overline{\Phi}_{p,\mathrm{inv}}\|_{1,\tau}}.
\end{align}}
We then perform truncations to obtain
{\footnotesize\begin{align}
    &\left\| A_p \left(L_p \overline{\mathbf{U}} + \begin{bmatrix}
        \mathscr{g}_p(\overline{\mathbf{U}}) \\ -\nu_3 \nu_5 + \nu_4
    \end{bmatrix}\right)\right\|_{1,\tau} \\
    &= \left\| A_p^N \left(L_p \overline{\mathbf{U}} + \begin{bmatrix}
        \mathscr{g}_p(\overline{\mathbf{U}}) \\ -\nu_3 \nu_5 + \nu_4
    \end{bmatrix}\right)\right\|_{1,\tau} + \left\| \bpi^{> N}A_p \left(L_p \overline{\mathbf{U}} + \begin{bmatrix}
        \mathscr{g}_p(\overline{\mathbf{U}}) \\ -\nu_3 \nu_5 + \nu_4
    \end{bmatrix}\right)\right\|_{1,\tau} \\
    &\leq \left\| A_p^N \left(L_p \overline{\mathbf{U}} + \begin{bmatrix}
        \mathscr{g}_p(\overline{\mathbf{U}}) \\ -\nu_3 \nu_5 + \nu_4
    \end{bmatrix}\right)\right\|_{1,\tau} + \mathcal{L}_{\infty}\left\|(\bpi^{\leq N_0} - \bpi^{\leq N}) L_p\overline{\mathbf{U}} + \begin{bmatrix}
        (\Pi^{\leq 4N_0} - \Pi^{\leq N}) (\overline{\Psi}_p * \overline{\Phi}_{p,\mathrm{inv}}) \\ 0
    \end{bmatrix}\right\|_{1,\tau}
\end{align}}
which is the $Y_{0,1}$ bound. Note that since $\overline{\mathbf{U}}$ has finitely many coefficients, it follows that $L_p \overline{\mathbf{U}} \in \ell^1_{e,\tau}$. Hence, we could perform the final estimate. We also truncate to get
{\footnotesize\begin{align}
    &\frac{\left\|A_p \begin{bmatrix} \overline{\Psi}_p*\overline{\Phi}_{p,\mathrm{inv}}*(1-\overline{\Phi}_p*\overline{\Phi}_{p,\mathrm{inv}}) \\ 0 \end{bmatrix}\right\|_{1,\tau}}{1-\|1-\overline{\Phi}_p*\overline{\Phi}_{p,\mathrm{inv}}\|_{1,\tau}} \\
    &\leq \frac{\left\|A_p^N \begin{bmatrix} \overline{\Psi}_p*\overline{\Phi}_{p,\mathrm{inv}}*(1-\overline{\Phi}_p*\overline{\Phi}_{p,\mathrm{inv}}) \\ 0 \end{bmatrix}\right\|_{1,\tau}}{1-\|1-\overline{\Phi}_p*\overline{\Phi}_{p,\mathrm{inv}}\|_{1,\tau}} + \frac{\left\|\bpi^{> N}A_p^N \begin{bmatrix} \overline{\Psi}_p*\overline{\Phi}_{p,\mathrm{inv}}*(1-\overline{\Phi}_p*\overline{\Phi}_{p,\mathrm{inv}}) \\ 0 \end{bmatrix}\right\|_{1,\tau}}{1-\|1-\overline{\Phi}_p*\overline{\Phi}_{p,\mathrm{inv}}\|_{1,\tau}} \\
    &\leq \frac{\left\|A_p^N \begin{bmatrix} \overline{\Psi}_p*\overline{\Phi}_{p,\mathrm{inv}}*(1-\overline{\Phi}_p*\overline{\Phi}_{p,\mathrm{inv}}) \\ 0 \end{bmatrix}\right\|_{1,\tau}}{1-\|1-\overline{\Phi}_p*\overline{\Phi}_{p,\mathrm{inv}}\|_{1,\tau}} + \frac{\|\bpi^{> N} A_p\|_{\mathcal{B}(\ell^1_{e,\tau})}\left\|\Pi^{> N} \overline{\Psi}_p*\overline{\Phi}_{p,\mathrm{inv}}*(1-\overline{\Phi}_p*\overline{\Phi}_{p,\mathrm{inv}})\right\|_{1,\tau}}{1-\|1-\overline{\Phi}_p*\overline{\Phi}_{p,\mathrm{inv}}\|_{1,\tau}} \\
    &\leq \frac{\left\|A_p^N \begin{bmatrix} \overline{\Psi}_p*\overline{\Phi}_{p,\mathrm{inv}}*(1-\overline{\Phi}_p*\overline{\Phi}_{p,\mathrm{inv}}) \\ 0 \end{bmatrix}\right\|_{1,\tau}}{1-\|1-\overline{\Phi}_p*\overline{\Phi}_{p,\mathrm{inv}}\|_{1,\tau}} + \frac{\mathcal{L}_{\infty}\left\|(\Pi^{\leq 8N_0} - \Pi^{\leq N}) (\overline{\Psi}_p*\overline{\Phi}_{p,\mathrm{inv}}*(1-\overline{\Phi}_p*\overline{\Phi}_{p,\mathrm{inv}}))\right\|_{1,\tau}}{1-\|1-\overline{\Phi}_p*\overline{\Phi}_{p,\mathrm{inv}}\|_{1,\tau}}
\end{align}}
which leads to the $Y_{0,2}$ bound.
 \end{proof}
 Next, we compute the $Z_2$ bound. Its computation will differ from Lemma \ref{lem : Z2 patterns} as we are estimating it in a fundamentally different way. 
 \begin{lemma}\label{lem : Z2 periodic}
Let $R > 0$. Define $Z_2 > 0$ as 
\begin{align}
    Z_2 \bydef 2\nu_1\left(\|A_p^N\|_{\mathcal{B}(\ell^1_{e,\tau})} + \mathcal{L}_{\infty}\right)(Z_{2,1}Z_{2,2} + Z_{2,3})
\end{align}
where
\begin{align}
    &Z_{2,1} \bydef \|\nu_2^2 \overline{U}_1^3 *\overline{U}_2 + \nu \nu_2 \overline{U}_1^3 + 3\nu \nu_2 \overline{U}_1^2 - 3\nu_2 \overline{U}_1 *\overline{U}_2 - \nu - \overline{U}_2\|_{1,\tau} \\
    &+(\|-1-3\nu_2 \overline{U}_1+\nu_2^2 \overline{U}_1^3\|_{1,\tau} + \|6\nu\nu_2 \overline{U}_1 + 3 \nu \nu_2 \overline{U}_1^2 - 3\nu_2 \overline{U}_2 + 3\nu_2^2 \overline{U}_1^2* \overline{U}_2\|_{1,\tau})R \\
    &+ (\|-3\nu_2 + 3\nu_2^2 \overline{U}_1^2\|_{1,\tau} + \|3\nu \nu_2 + 3\nu \nu_2 \overline{U}_1 + 3\nu_2^2 \overline{U}_1^2\|_{1,\tau}) R^2 \\
    &+ (\|3\nu_2^2 \overline{U}_1 \|_{1,\tau} + \|\nu \nu_2 + \nu_2^2 \overline{U}_2\|_{1,\tau})R^3 + \nu_2^2 R^4, \\
    &Z_{2,2} \bydef \frac{\|\overline{\Phi}_{p,\mathrm{inv}}\|_{1,\tau}^3}{\left(1- \|1-\overline{\Phi}_p*\overline{\Phi}_{p,\mathrm{inv}}\|_{1,\tau} - R\|\overline{\Phi}_{p,\mathrm{inv}}\|_{1,\tau}\right)^3}, \\
    &Z_{2,3} \bydef \frac{(\|1 - \nu_2 \overline{U}_1^2\|_{1,\tau} + 2\nu_2 \|\overline{U}_1\|_{1,\tau} R + \nu_2 R^2)\|\overline{\Phi}_{p,\mathrm{inv}}\|_{1,\tau}^2}{\left(1 - \|1-\overline{\Phi}_p*\overline{\Phi}_{p,\mathrm{inv}}\|_{1,\tau} - R \|\overline{\Phi}_{p,\mathrm{inv}}\|_{1,\tau}\right)^2}.
\end{align}
Then, it follows that $\underset{\mathbf{U} \in B_R(\overline{\mathbf{U}})}{\sup}\|A_pDF^2_p(\mathbf{U}) \|_{\mathcal{B}(\ell^1_{e,\tau},\mathcal{B}(\ell^1_{e,\tau}))} \leq Z_2.$
 \end{lemma}
 \begin{proof}
The proof can be found in Appendix \ref{apen : Z2 periodic}
 \end{proof}
 Let us now compute the $Z_1$ bound. We first introduce 
 {\footnotesize\begin{align}
     V_{p,1}^N \bydef \Pi^{\leq N} (\overline{\Psi}_{p,1} * \overline{\Phi}_{p,\mathrm{inv}}^{2}), ~ V_{p,2}^N \bydef \Pi^{\leq N} (\overline{\Psi}_{p,2} * \overline{\Phi}_{p,\mathrm{inv}}), ~ DG_p^N(\overline{\mathbf{U}}) \bydef \begin{bmatrix}
         \mathbb{V}_{p,1}^N & \mathbb{V}_{p,2}^N \\
         0 & 0
     \end{bmatrix}, ~ DG_p(\overline{\mathbf{U}}) \bydef \begin{bmatrix}
         \mathbb{V}_{p,1} & \mathbb{V}_{p,2} \\
         0 & 0
     \end{bmatrix}.
 \end{align}}
 We now state the following lemma.
 \begin{lemma}\label{lem : Z1 periodic}
Let $Z_{\infty} > 0$ be defined as 
\begin{align}
&Z_{\infty} \bydef (\|A_p^N\|_{\mathcal{B}(\ell^1_{e,\tau})} + \mathcal{L}_{\infty})(Z_{\infty,1} + Z_{\infty,2}) 
\end{align}
where 
\begin{align}
&Z_{\infty,1} \bydef \|V_{p,1}^N - \mathscr{V}_{p,1}\|_{1,\tau} + \frac{\|\overline{\Psi}_{p,1}*\overline{\Phi}_{p,\mathrm{inv}}^2*(1-\overline{\Phi}_p^2*\overline{\Phi}_{p,\mathrm{inv}}^2)\|_{1,\tau}}{1-\|1-\overline{\Phi}_p^2*\overline{\Phi}_{p,\mathrm{inv}}^2\|_{1,\tau}} \\
&Z_{\infty,2} \bydef \|V_{p,2}^N - \mathscr{V}_{p,2}\|_{1,\tau} + \frac{\|\overline{\Psi}_{p,2}*\overline{\Phi}_{p,\mathrm{inv}}*(1-\overline{\Phi}_p*\overline{\Phi}_{p,\mathrm{inv}})\|_{1,\tau}}{1-\|1-\overline{\Phi}_p*\overline{\Phi}_{p,\mathrm{inv}}\|_{1,\tau}}
\end{align}
Then, define $Z_1 > 0$ as 
\begin{align}
    &Z_1 \bydef \|\bpi^{\leq N}(I_d - A_p(L_p + DG_p^N(\overline{\mathbf{U}})))\bpi^{\leq 2N}\|_{\mathcal{B}(\ell^1_{e,\tau})} + 2\mathcal{L}_{\infty}(\|V_{p,1}^N\|_{1,\tau} + \|V_{p,2}^N\|_{1,\tau}) + Z_{\infty}.
\end{align}
Then, it follows that $\|I_d - A_p DF_p(\overline{\mathbf{U}})\|_{\mathcal{B}(\ell^1_{e,\tau})} \leq Z_1$.
 \end{lemma}
 \begin{proof}
To begin, we introduce $DG_p^N(\overline{\mathbf{U}})$.
{\small\begin{align}
    \|I_d - A_pDF_p(\overline{\mathbf{U}})\|_{\mathcal{B}(\ell^1_{e,\tau})} &= \|I_d - A_p(L_p + DG_p(\overline{\mathbf{U}}))\|_{\mathcal{B}(\ell^1_{e,\tau})}\\
    &\leq \|I_d - A_p(L_p + DG_p^N(\overline{\mathbf{U}}))\|_{\mathcal{B}(\ell^1_{e,\tau})} + \|A_p(DG_p^N(\overline{\mathbf{U}}) - DG_p(\overline{\mathbf{U}}))\|_{\mathcal{B}(\ell^1_{e,\tau})}.\label{before splitting Z_infinity}
    \end{align}}
We now examine the second term of \eqref{before splitting Z_infinity}.
{\footnotesize\begin{align}
    \|A_p(DG_p^N(\overline{\mathbf{U}}) - DG_p(\overline{\mathbf{U}}))\|_{\mathcal{B}(\ell^1_{e,\tau})}&\leq  \|A_p\|_{\mathcal{B}(\ell^1_{e,\tau})}\|DG_p^N(\overline{\mathbf{U}}) - DG_p(\overline{\mathbf{U}})\|_{\mathcal{B}(\ell^1_{e,\tau})} \\
    &\leq \left(\|A_p^N\|_{\mathcal{B}(\ell^1_{e,\tau})} + \mathcal{L}_{\infty}\right) (\|V_{p,1}^N - V_{p,1}\|_{1,\tau} + \|V_{p,2}^N - V_{p,2}\|_{1,\tau})\\
    &\leq \left(\|A_p^N\|_{\mathcal{B}(\ell^1_{e,\tau})} + \mathcal{L}_{\infty}\right) (Z_{\infty,1} + Z_{\infty,2}) \bydef Z_{\infty}\label{def : Zinfty}
\end{align}}
where we used similar steps to those done in \ref{lem : Z_full_1} to estimate $Z_{\infty}$ only differing in our use of \ref{lem : phi bound stronger}.
Let us now examine $I_d - A_p(L_p + DG_p^N(\overline{\mathbf{U}}))$. Let $M = L_p + DG_p^N(\overline{\mathbf{U}})$ We follow the approach used by the authors of \cite{miguel_soliton} and introduce projections
{\scriptsize\begin{align}
    I_d - A_pM&= \bpi^{\leq N} (I_d - A_pM)\bpi^{\leq 2N} + \bpi^{\leq N}(I_d - A_pM)\bpi^{>2N} + \bpi^{>N} (I_d - A_pM)\bpi^{\leq 2N} + \bpi^{> N}(I_d - A_pM)\bpi^{> 2N} \\
    &= \bpi^{\leq N}(I_d - A_pM)\bpi^{\leq 2N} - A_p^N M \bpi^{> 2N} + \bpi^{>N} \bpi^{\leq 2N} - \bpi^{>N} A_p M \bpi^{ \leq 2N} + \bpi^{>N} - \bpi^{>N} A_pM \bpi^{> 2N} \\
    &= \bpi^{\leq N}(I_d - A_pM)\bpi^{\leq 2N} - A_p^N M \bpi^{> 2N} + \bpi^{\leq 2N} - \bpi^{\leq N} - \bpi^{>N} A_p M \bpi^{\leq 2N} + \bpi^{>N} - \bpi^{>N} A_pM \bpi^{>2N}.
\end{align}}
The first term above is completed, and present in our definition of $Z_1$.
For the second term, since $\bpi^{\leq N} L_p \bpi^{>2N} = 0$, observe that
{\footnotesize\begin{align}
    \|-A_p^N M \bpi^{>2N}\|_{\mathcal{B}(\ell^1_{e,\tau})} = \|A_p^N (L_p + DG_p^N(\overline{\mathbf{U}}))\bpi^{>2N}\|_{\mathcal{B}(\ell^1_{e,\tau})}
    &= \|A_p^N DG_p^N(\overline{\mathbf{U}})\bpi^{>2N}\|_{\mathcal{B}(\ell^1_{e,\tau})} \\
    &\leq \|A_p^N\|_{\mathcal{B}(\ell^1_{e,\tau})} \|\bpi^{\leq N}DG_p^N(\overline{\mathbf{U}})\bpi^{>2N}\|_{\mathcal{B}(\ell^1_{e,\tau})}.
\end{align}}
Now, let $\mathbf{W} \bydef (W_1,W_2) \in \ell^1_{e,\tau}, \|\mathbf{W}\|_{1,\tau} = 1$. Then,
\begin{align}
    \|\bpi^{\leq N} DG_p^N(\overline{\mathbf{U}}) \bpi^{> 2N}\|_{\mathcal{B}(\ell^1_{e,\tau})} &= \|\Pi^{\leq N} \mathbb{V}_{p,1}^N \Pi^{>2N} W_1 \|_{1,\tau} + \|\Pi^{\leq N} \mathbb{V}_{p,2}^N \Pi^{>2N} W_2 \|_{1,\tau} \\
    &= \|\Pi^{\leq N} (V_{p,1}^N * \Pi^{>2N} W_1)\|_{1,\tau} + \|\Pi^{\leq N} (V_{p,2}^N * \Pi^{>2N} W_2)\|_{1,\tau} \\
    &= \|\Pi^{\leq N} \Pi^{>N} (V_{p,1}^N * W_1)\|_{1,\tau} + \|\Pi^{\leq N} \Pi^{>N} (V_{p,2}^N * W_2)\|_{1,\tau}= 0.\label{piN front pi2N back equals 0}
\end{align}
Hence, the term $\|-A_p^N M\bpi^{>2N}\|_{\mathcal{B}(\ell^1_{e,\tau})} = 0$. We now move to the third term where
{\scriptsize\begin{align}
    \|\bpi^{\leq 2N} - \bpi^{\leq N} - \bpi^{> N} A_p M\bpi^{\leq 2N}\|_{\mathcal{B}(\ell^1_{e,\tau})}
    &= \|\bpi^{\leq 2N} - \bpi^{\leq N} - \bpi^{>N} A_p(L_p + DG_p^N(\overline{\mathbf{U}}))\bpi^{\leq 2N}\|_{\mathcal{B}(\ell^1_{e,\tau})} \\
    &= \|\bpi^{\leq 2N} - \bpi^{\leq N} - \bpi^{> N}\bpi^{\leq 2N} - \bpi^{>N} A_pDG_p^N(\overline{\mathbf{U}})\bpi^{\leq 2N}\|_{\mathcal{B}(\ell^1_{e,\tau})} \\
    &= \|\bpi^{\leq 2N} - \bpi^{\leq N} - (\bpi^{\leq 2N} - \bpi^{\leq N}) - \bpi^{>N} A_pDG_p^N(\overline{\mathbf{U}})\bpi^{\leq 2N}\|_{\mathcal{B}(\ell^1_{e,\tau})}\\
    &= \|\bpi^{>N} A_pDG_p^N(\overline{\mathbf{U}})\bpi^{\leq 2N}\|_{\mathcal{B}(\ell^1_{e,\tau})} \end{align}}
We now perform further estimates to remove the $\bpi^{> N} A_p$. 
{\footnotesize\begin{align}
    \|\bpi^{>N} A_pDG_p^N(\overline{\mathbf{U}})\bpi^{\leq 2N}\|_{\mathcal{B}(\ell^1_{e,\tau})}&\leq \|\bpi^{> N} A_p\|_{\mathcal{B}(\ell^1_{e,\tau})} \|\bpi^{>N}DG_p^N(\overline{\mathbf{U}}) \bpi^{\leq 2N}\|_{\mathcal{B}(\ell^1_{e,\tau})}
    \\
    &\leq \mathcal{L}_{\infty} \|DG_p^N(\overline{\mathbf{U}})\|_{\mathcal{B}(\ell^1_{e,\tau})} \\
    &\leq \mathcal{L}_{\infty} (\|V_{p,1}^N\|_{1,\tau} + \|V_{p,2}^N\|_{1,\tau}).\label{first Z12N}
\end{align}}
We now move to the final term and estimate.
\begin{align}
    \|\bpi^{>N} - \bpi^{>N} A_p M \bpi^{> 2N} \|_{\mathcal{B}(\ell^1_{e,\tau})} &= \|\bpi^{>N} - \bpi^{>N} A_p (L_p + DG_p^N(\overline{\mathbf{U}})) \bpi^{>2N} \|_{\mathcal{B}(\ell^1_{e,\tau})} \\
    &=\|\bpi^{>N} - \bpi^{>N}\bpi^{>2N} - \bpi^{>N}A_p DG_p^N(\overline{\mathbf{U}}) \bpi^{>2N} \|_{\mathcal{B}(\ell^1_{e,\tau})} \\
    &= \|\bpi^{>N} A_p DG_p^N(\overline{\mathbf{U}})\bpi^{>2N}\|_{\mathcal{B}(\ell^1_{e,\tau})}.
\end{align}
We now perform similar steps as those used in \eqref{first Z12N}. That is,
{\footnotesize\begin{align}
    \|\bpi^{>N} A_p DG_p^N(\overline{\mathbf{U}})\bpi^{>2N}\|_{\mathcal{B}(\ell^1_{e,\tau})} &\leq \|\bpi^{>N} A_p \|_{\mathcal{B}(\ell^1_{e,\tau})} \|\bpi^{>N} DG_p^N(\overline{\mathbf{U}}) \bpi^{>2N}\|_{\mathcal{B}(\ell^1_{e,\tau})}
    \\
    &\leq \mathcal{L}_{\infty} \|DG_p^N(\overline{\mathbf{U}})\|_{\mathcal{B}(\ell^1_{e,\tau})} \\
    &\leq \mathcal{L}_{\infty} (\|V_{p,1}^N\|_{1,\tau} + \|V_{p,2}^N\|_{1,\tau}).\label{second Z12N}
\end{align}}
We now combine \eqref{def : Zinfty}, \eqref{first Z12N}, and \eqref{second Z12N} to obtain
the desired result.
\end{proof}
With $Y_0, Z_1,$ and $Z_2$ now computed, let us discuss how we can obtain our rigorous proofs.
\subsection{Constructive existence proofs of periodic solutions}\label{sec : Proofs of periodic solutions}
In this section, we present our proofs of existence (and local uniqueness) of periodic solutions to the 1D Thomas model. 
\begin{theorem}[\bf First Periodic Solution in Thomas]\label{th : periodic solution in thomas}
Let $\nu = 0.1764, \nu_1 = 8, \nu_2 = 1, \nu_3 = 0.28, \nu_4 = 21, \nu_5 = 67.46981860371494$. Moreover, let $R \bydef 5 \times 10^{-9}, \tau \bydef 1.01$. Then there exists a unique solution $\widetilde{\mbf{U}}$ to \eqref{def : Fp} in $\overline{B_{R}(\overline{\mbf{U}})} \subset \ell^1_{e,1.01}$ and we have that $\|\widetilde{\mbf{U}}-\overline{\mbf{U}}\|_{1,1.01} \leq R$. 
\end{theorem}
\begin{proof}
Choose $N_0 = 27,N = 12, d = 5$. Then, find $\overline{\mathbf{U}}$ as discussed in Section \ref{sec : u0}. Next, we construct $A_p^N$ and find
$\|A_p^N\|_{\mathcal{B}(\ell^1_{e,\tau})} \leq 105.41$. Finally, using \cite{ThomasProofs.jl}, we choose $R \bydef 5 \times 10^{-9} $ and define
\begin{align}
    Y_0 \bydef 1.374 \times 10^{-9} \text{,}~Z_{2} \bydef 164.5     \text{,}~Z_1 \bydef 0.67, 
    \end{align}
and prove that these values satisfy Theorem \ref{th : radii polynomial theorem periodic}. 
\end{proof}
\begin{figure}[H]
\centering
 \begin{minipage}{.5\linewidth}
  \centering\epsfig{figure=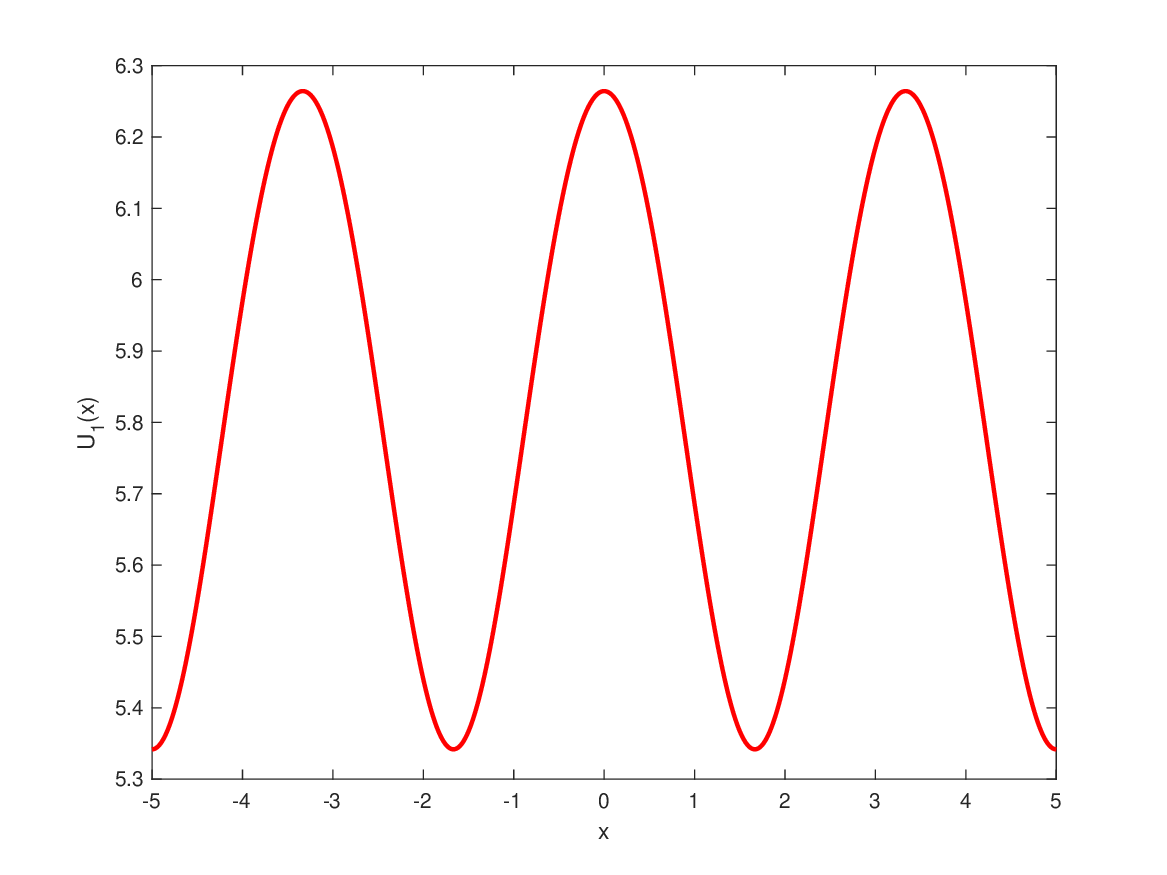,width=\linewidth}
  \end{minipage}%
 \begin{minipage}{.5\linewidth}
  \centering\epsfig{figure=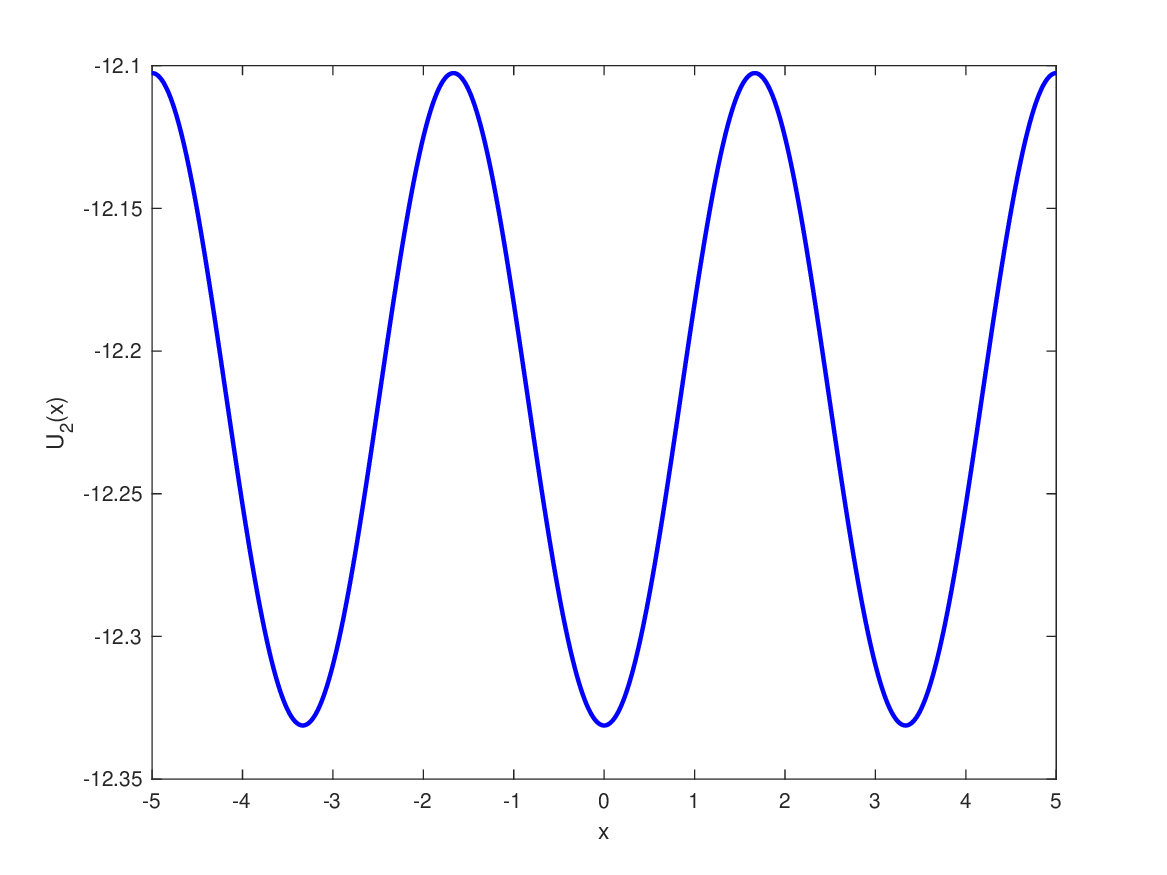,width=\linewidth}
 \end{minipage} 
 \caption{Plot of $ \overline{U}_1$ (L) and $\overline{U}_2$ (R) on $(-5,5)$ used in the proof of Theorem \ref{th : periodic solution in thomas}.}\label{fig : th2}
 \end{figure}%
 \begin{theorem}[\bf Second Periodic Solution in Thomas]\label{th : periodic solution in thomas 2}
Let $\nu = 1.1664, \nu_1 = 8, \nu_2 = 1, \nu_3 = 0.28, \nu_4 = 39.1, \nu_5 = 149.7672$. Moreover, let $R \bydef 5 \times 10^{-8}, \tau \bydef 1.02$. Then there exists a unique solution $\widetilde{\mbf{U}}$ to \eqref{def : Fp} in $\overline{B_{R}(\overline{\mbf{U}})} \subset \ell^1_{e,1.02}$ and we have that $\|\widetilde{\mbf{U}}-\overline{\mbf{U}}\|_{1,1.02} \leq R$. 
\end{theorem}
\begin{proof}
Choose $N_0 = 30,N = 17, d = \frac{10}{3}$. The proof follows similarly to that of Theorem \ref{th : periodic solution in thomas}. In particular, we find
\begin{align}
    \|A_p^N\|_{\mathcal{B}(\ell^1_{e,1.02})} \leq 372.94,~Y_0 \bydef 9.051 \times 10^{-9} \text{,}~Z_{2} \bydef 1.583 \times 10^{6}     \text{,}~Z_1 \bydef 0.7544, 
    \end{align}
and prove that these values satisfy Theorem \ref{th : radii polynomial theorem periodic}. 
\end{proof}
\begin{figure}[H]
\centering
 \begin{minipage}{.5\linewidth}
  \centering\epsfig{figure=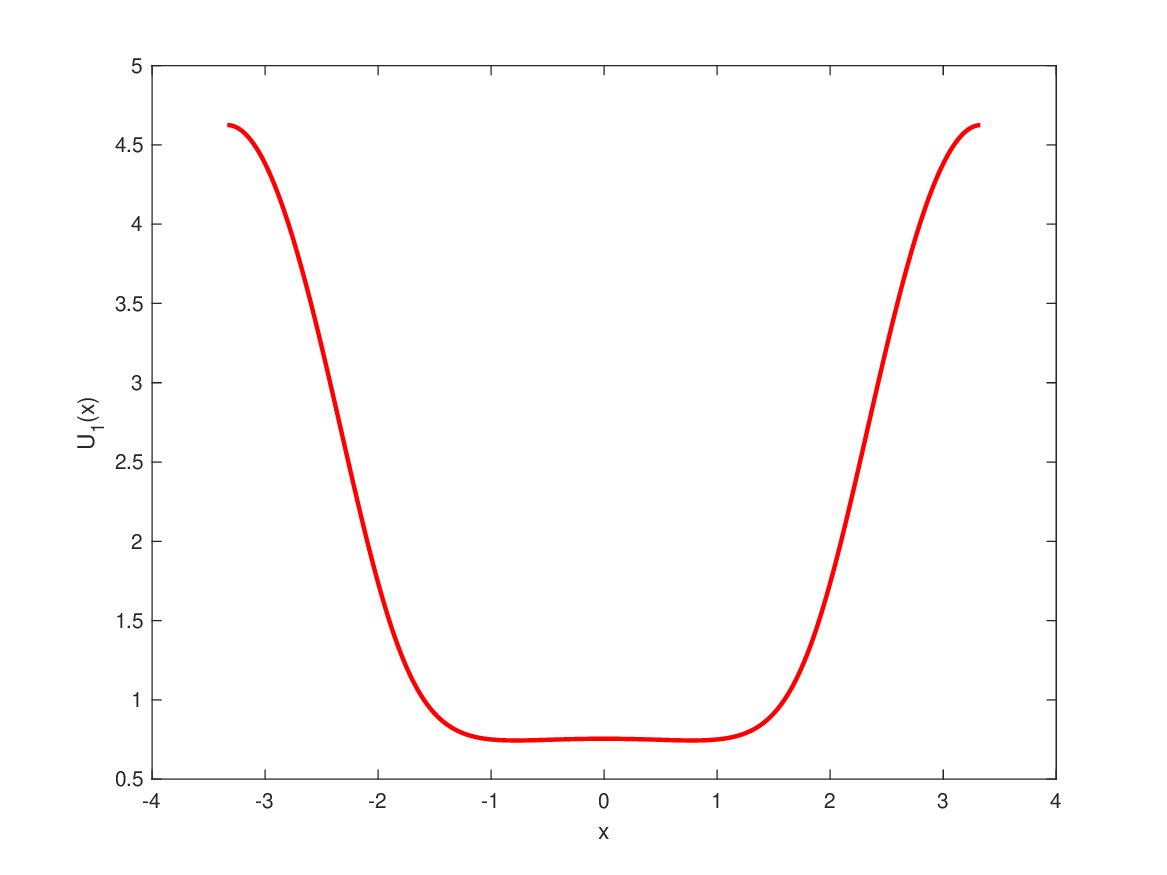,width=\linewidth}
  \end{minipage}%
 \begin{minipage}{.5\linewidth}
  \centering\epsfig{figure=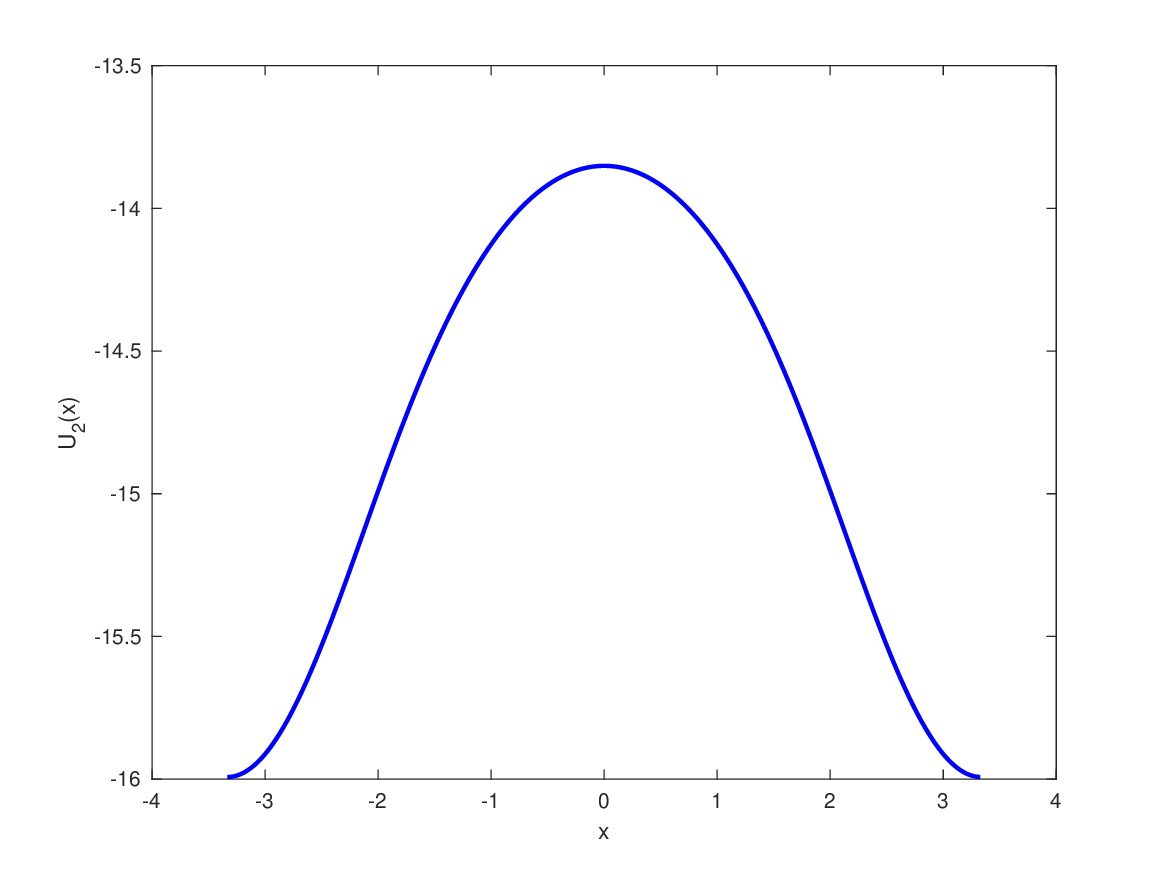,width=\linewidth}
 \end{minipage} 
 \caption{Plot of $ \overline{U}_1$ (L) and $\overline{U}_2$ (R) on $\left(-\frac{10}{3},\frac{10}{3}\right)$ used in the proof of Theorem \ref{th : periodic solution in thomas 2}.}\label{fig : th22}
 \end{figure}%
 \begin{theorem}[\bf Third Periodic Solution in Thomas]\label{th : periodic solution in thomas 3}
Let $\nu = 0.1764, \nu_1 = 8, \nu_2 = 1, \nu_3 = 0.28, \nu_4 = 21, \nu_5 = 65$. Moreover, let $R \bydef 3 \times 10^{-7}, \tau \bydef 1.02$. Then there exists a unique solution $\widetilde{\mbf{U}}$ to \eqref{def : Fp} in $\overline{B_{R}(\overline{\mbf{U}})} \subset \ell^1_{e,1.02}$ and we have that $\|\widetilde{\mbf{U}}-\overline{\mbf{U}}\|_{1,1.02} \leq R$. 
\end{theorem}
\begin{proof}
Choose $N_0 = 100, N = 100, d = 10$. The proof follows similarly to that of Theorem \ref{th : periodic solution in thomas}. In particular, we find
\begin{align}
    \|A_p^N\|_{\mathcal{B}(\ell^1_{e,1.02})} \leq 3197.45,~Y_0 \bydef 2.08 \times 10^{-7} \text{,}~Z_{2} \bydef 509283.1     \text{,}~Z_1 \bydef 0.123, 
    \end{align}
and prove that these values satisfy Theorem \ref{th : radii polynomial theorem periodic}. 
\end{proof}
\begin{figure}[H]
\centering
 \begin{minipage}{.5\linewidth}
  \centering\epsfig{figure=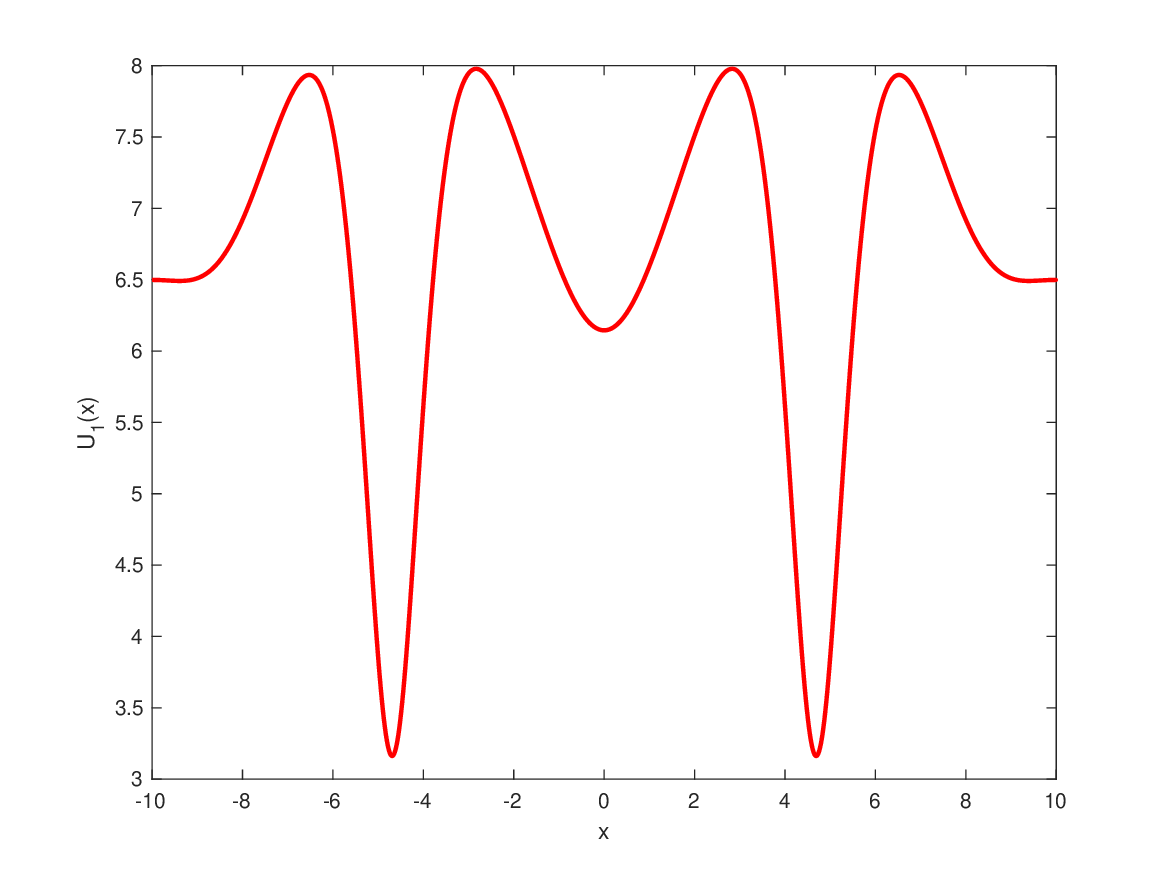,width=\linewidth}
  \end{minipage}%
 \begin{minipage}{.5\linewidth}
  \centering\epsfig{figure=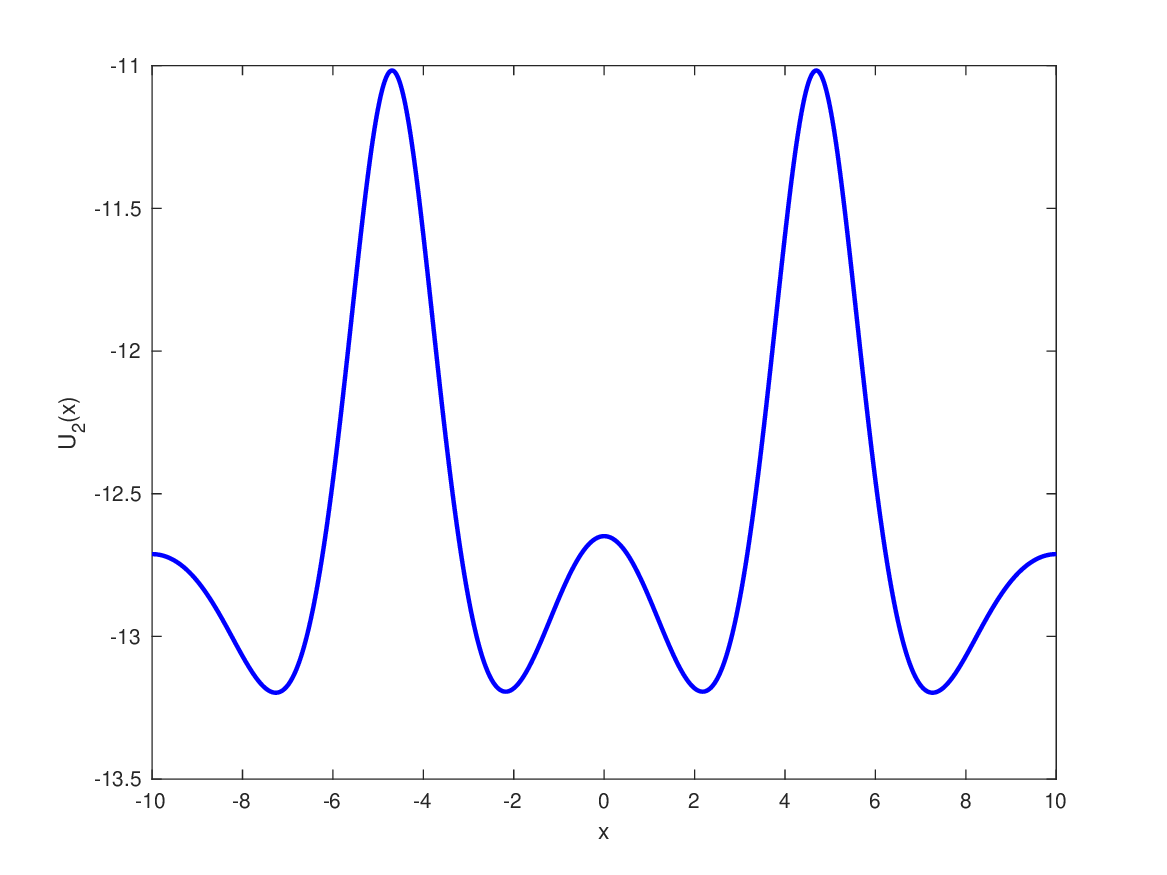,width=\linewidth}
 \end{minipage} 
 \caption{Plot of $ \overline{U}_1$ (L) and $\overline{U}_2$ (R) on $(-10,10)$ used in the proof of Theorem \ref{th : periodic solution in thomas 3}.}\label{fig : th23}
 \end{figure}%
Now that we have provided an approach to prove a specific periodic solution, we are able to rigorously verify such solutions in \eqref{eq:thomas}. This partially answers more of the questions asked by the authors of \cite{thomas_1d}. Now, let us provide a stronger approach to prove a branch of periodic solutions. This will allow us to verify not only the periodic solutions obtainable via this section's approach, but an entire branch as well.
 \section{Branches of Periodic Solutions}\label{sec : branches}
As mentioned at the start of Section \ref{sec : periodic solutions}, the authors of \cite{thomas_1d} not only identified approximate periodic solutions in \eqref{eq:thomas}, but also an approximate branch of periodic solutions. In this section, we present a methodology for rigorously verifying such branches. We will use a similar approach to that presented in Section 4 of \cite{dominic_sh_periodic}.
\par For our purposes, we will be performing continuation in $\nu_5$, which is what the authors of \cite{thomas_1d} did. We will use pseudo-arclength continuation to perform our rigorous proof. The benefit of pseudo-arclength continuation is that we can pass through saddle node (or fold) bifurcations. We will require this as the authors of \cite{thomas_1d} conjectured a fold bifurcation. To begin, we define $X_{e,\tau} \bydef \mathbb{R} \times \ell^1_{e,\tau}$ and write $\mathbf{W} \bydef (\nu_5,\mathbf{U}) \in X_{e,\tau}$. Now, we will expand $\nu_5$ and $\mathbf{U}$ as Chebyshev series dependent on the pseudo-arclength, $s$. That is, we write
\begin{align}
    \nu_5(s) \bydef (\nu_5)_0 + 2\sum_{n = 1}^{N_c} (\nu_5)_n T_n(s), ~~ \text{and} ~~ U_j(s) \bydef (U_j)_0 + 2\sum_{n = 1}^{N_c} (U_j)_0 T_n(s), ~ s \in [-1,1]
\end{align}
for $j = 1,2$ and where $T_n [-1,1] \to \mathbb{R}$ are the Chebyshev polynomials of the first kind. Additionally, $(\nu_5)_n \in \mathbb{R}, (U_j)_n \in \ell^1_{e,\tau}$ for all $n \in \{0,\dots,N_c\}$. We then write $\mathbf{U}(s) \bydef (U_1(s),U_2(s))$ and $\mathbf{W}(s) \bydef (\nu_5(s),\mathbf{U}(s)) \in \mathbb{R}_{\mathrm{con}} \times \ell^1_{e,\tau,\mathrm{con}} \bydef X_{e,\tau,\mathrm{con}}$ where $\mathbb{R}_{\mathrm{con}},\ell^1_{e,\tau,\mathrm{con}} \bydef \ell^1_{\tau,\mathrm{con}}(\mathbb{N}_0)) \times \ell^1_{\tau,\mathrm{con}}(\mathbb{N}_0)),$ and $X_{e,\tau,\mathrm{con}}$ are Banach spaces with the norms
\begin{align}
    &\|\nu_5(s)\|_{\mathbb{R}_{\mathrm{con}}} \bydef |(\nu_5)_0| + 2\sum_{n = 1}^{N_c} |(\nu_5)_n|, ~ \|U_j(s)\|_{1,\tau,\mathrm{con}} \bydef \|(U_j)_0\|_{1,\tau} + 2\sum_{n = 1}^{N_c} \|(U_j)_n\|_{1,\tau} \\
    &\|\mathbf{U}(s)\|_{1,\tau,\mathrm{con}} \bydef \|U_1(s)\|_{1,\tau,\mathrm{con}} + \|U_2(s)\|_{1,\tau,\mathrm{con}}, ~ \|\mathbf{W}(s)\|_{X_{\tau,\mathrm{con}}} \bydef \|\nu_5(s)\|_{\mathbb{R}_{\mathrm{con}}} + \|\mathbf{U}(s)\|_{1,\tau,\mathrm{con}}.
\end{align}
We then define the map $F_{c} : X_{e,\tau,\mathrm{con}} \to \mathcal{S}$ for $\mathcal{S}$ another Banach space as
\begin{align}
    F_{c}(\mathbf{W}(s)) \bydef \begin{bmatrix}
        (\mathbf{U}(s) - \overline{\mathbf{U}}(s),\dot{\mathbf{U}}(s))_2 \\
        F_p(\nu_5(s),\mathbf{U}(s))
    \end{bmatrix}\label{eq : psuedo arclength system}
\end{align}
where $\dot{\mathbf{U}}(s)$ is the second component of the \emph{tangent vector}, $\dot{\mathbf{W}}(s)$.
We do the same for the approximate inverse, which we will denote by $A_{c}(s) : \mathcal{S} \to X_{e,\tau,\mathrm{con}}$ 
\begin{align}
    A_{c}(s) = (A_{c})_0 + 2 \sum_{n = 1}^{N_{c}} (A_{c})_n T_n(s).
\end{align}
Note that in this case $A_{c}(s)$ is an approximate inverse to $DF_{c}(\mathbf{W}(s))$. 
Then, the following theorem makes use of a uniform contraction argument to allow us to rigorously prove a branch of solutions. 
\begin{theorem}[Newton-Kantorovich Theorem for Branches]
\label{th: radii polynomial continuation}
Let $R > 0$. Also let $A_{c}(s) \in \mathcal{B}(\mathcal{S},X_{e,\tau,\mathrm{con}})$ be injective. Moreover, let $Y_0^{[s]} Z_1^{[s]}, Z_2^{[s]} \bydef Z_2^{[s]}(R)$ be non-negative constants such that
  \begin{align}\label{eq: definition Y0 Z1 Z2 continuation}
    &\sup_{s \in [-1,1]}\|A_{c}(s)F_{c}(\overline{\mathbf{W}}(s))\|_{X_{\tau,\mathrm{con}}}  \le Y_0^{[s]}\\
    &\sup_{s \in [-1,1]}\|I_d - A_{c}(s)DF_{c}(\overline{\mathbf{W}}(s))\|_{\mathcal{B}(X_{e,\tau,\mathrm{con}})} \le Z_1^{[s]}\\
    &\sup_{s \in [-1,1]} \sup_{\mathbf{W}(s) \in B_R(\overline{\mathbf{W}}(s))}\|A_{c}(s)D^2F_{c}(\mathbf{W}(s))\|_{\mathcal{B}(X_{e,\tau,\mathrm{con}},\mathcal{B}(X_{\tau,\mathrm{con}}))}  \le Z_2^{[s]}.
\end{align} 
If 
\begin{align}
    2Y_0^{[s]} Z_2^{[s]} \leq (1- Z_1^{[s]})^2, ~\text{and} ~ Z_1^{[s]} < 1,
\end{align}
is satisfied, then for every $s \in [-1,1]$ and any $r > 0 $ such that 
\begin{equation}\label{condition radii polynomial continuation}
    \frac{1 - Z_1^{[s]} - \sqrt{(1-Z_1^{[s]})^2 - 2Y_0^{[s]}Z_2^{[s]}}}{Z_2} \leq r < \min\left(\frac{1-Z_1^{[s]}}{Z_2^{[s]}},R\right),
 \end{equation}
 there exists a unique $\widetilde{\mathbf{W}}(s) \in \overline{B_{r}(\overline{\mathbf{W}}(s))} \subset X_{e,\tau,\mathrm{con}}$ such that $F_{c}(\widetilde{\mathbf{W}}(s))=0$. Moreover, the function $s \to \widetilde{\mathbf{W}}(s)$ is of class $C^{\infty}.$
\end{theorem}
\begin{proof}
The proof can be found in \cite{continuation_1}, \cite{continuation_2}, and \cite{continuation_3} for instance.
\end{proof}
Let us now apply the aforementioned theorem to prove branches.
\subsection{Computing the Bounds for Branches of Periodic Solutions}\label{sec : bounds branches}
In order to apply Theorem \ref{th: radii polynomial continuation}, we need to construct $\overline{\mathbf{W}}(s)$ and $A_{c}(s)$. We first introduce some notation. We define $\mathscr{P}^{\leq N}$ and $\mathscr{P}^{>N},$ as
\begin{align}
    \mathscr{P}^{\leq N} \bydef \begin{bmatrix}
        1 & 0 \\
        0 & \bpi^{\leq N}
    \end{bmatrix}, ~ \mathscr{P}^{ > N} \bydef \begin{bmatrix}
        0 & 0 \\
        0 & \bpi^{> N}
    \end{bmatrix}.
\end{align}
We then will choose $A_{c}(s) = (A_{c}(s))_{n \in N_c}$ where each $(A_{c})_n$ is constructed as in Section \ref{sec : A}. We also define 
\begin{align}
    A_{c}^N(s) \approx (\mathscr{P}^{\leq N} DF_{c}(\overline{\mathbf{W}}(s))\mathscr{P}^{\leq N})^{-1}.
\end{align}
Unlike in \cite{dominic_sh_periodic}, the linear part of $F_{c}$ does not depend on $\nu_5$. As a result, we can use uniform estimate $\|\mathscr{P}^{> N}(A_{c})_n\|_{\mathcal{B}(X_{e,\tau})} \leq \mathcal{L}_{\infty}$ for each $n = 0,\dots,N_c$. We let $\mathbf{V}(s) \in X_{e,\tau,\mathrm{con}}, \|\mathbf{V}(s)\|_{X_{\tau,\mathrm{con}}} = 1$. Then,
\begin{align}
    \|\mathscr{P}^{> N} A_{c}(s)\|_{\mathcal{B}(X_{e,\tau,\mathrm{con}})} &= \|\mathscr{P}^{> N} A_{c}(s) \mathbf{V}\|_{X_{\tau,\mathrm{con}}} \\
    &= \|\mathscr{P}^{> N} (A_{c})_0\mathbf{V}_0\|_{X_{\tau}} + 2\sum_{n = 1}^{N_c} \|\mathscr{P}^{> N} (A_{c})_n\mathbf{V}_n\|_{X_{\tau}} \\ 
    &\leq \|\mathscr{P}^{> N}(A_{c})_0\|_{\mathcal{B}(X_{e,\tau})}\|\mathbf{V}_0\|_{X_{\tau}} + 2\sum_{n = 1}^{N_c} \|\mathscr{P}^{> N}(A_{c})_0\|_{\mathcal{B}(X_{e,\tau})} \|\mathbf{V}_n\|_{X_{\tau}}
    \\
    &\leq \mathcal{L}_{\infty} \|\mathbf{V}_0\|_{X_{\tau}} + 2\mathcal{L}_{\infty}\sum_{n = 1}^{N_c} \| \mathbf{V}_n\|_{X_{\tau}} \\
    &= \mathcal{L}_{\infty} \|\mathbf{V}(s)\|_{X_{\tau,\mathrm{con}}} = \mathcal{L}_{\infty}.
\end{align}
The constructions of both $\overline{\mathbf{W}}(s)$ and $A_{c}(s)$ objects can be found in Section 4.1 of \cite{dominic_sh_periodic} with only minor changes as we have a system. Due to the similarity, we omit the details and refer the interested reader to the aforementioned paper. We also define $\mathcal{F}_{\mathcal{N}}\{\cdot\}$ and $\mathcal{F}_{\mathcal{N}}^{-1}\{\cdot\}$ to be the Fourier transform and inverse Fourier transform respectively of a Chebyshev sequence with $\mathcal{N}$ coefficients. Finally, we denote the product of two Chebyshev sequences by $\underset{\mathcal{N}}{*}$ defined as
{\scriptsize\begin{align}
    &\mathscr{v}_1(s)\underset{\mathcal{N}}{*} \mathscr{v}_2(s) \bydef \mathcal{F}_{\mathcal{N}}^{-1}\left\{ \mathcal{F}_{\mathcal{N}}\{\mathscr{v}_1(s)\}.*\mathcal{F}_{\mathcal{N}}\{\mathscr{v}_2(s)\}\right\}, ~\text{where}~ \mathscr{v}_1(s), \mathscr{v}_2(s) \in \left\{\mathbb{R}_{\mathrm{con}},\ell^1_{\tau,\mathrm{con}}(\mathbb{N}_0),\ell^1_{e,\tau,\mathrm{con}},X_{e,\tau,\mathrm{con}}\right\},\label{eq : chebproduct}
\end{align}}
where $.*$ denotes pointwise multiplication between the Chebyshev coefficients.
We also define
\begin{align}
    \mathscr{L}_{p} \bydef \begin{bmatrix}
        0 & 0 \\ 
        0 & L_p
    \end{bmatrix},~ \mathscr{G}_p(\overline{\mathbf{W}}(s)) \bydef \begin{bmatrix}
        0 \\\mathscr{g}_p(\overline{\mathbf{U}}(s)) \\
        -\nu_3 \overline{\nu_5}(s) + \nu_4
    \end{bmatrix}.
\end{align}
We are now ready to compute the bounds. As shown in \cite{dominic_sh_periodic}, the bounds for periodic solutions can be (essentially) re-used. Our primary task to assign the correct size of products for Chebyshev sequences. That is, we choose the correct $\mathcal{N}$ to compute $\mathscr{v}_1(s)\mathscr{v}_2(s)$ using \eqref{eq : chebproduct}. We now compute the $Y_0^{[s]}$ bound.
\begin{lemma}\label{lem : bound Y_0 branches}
Let $Y_0^{[s]} > 0$ be defined as 
\begin{align}
    Y_0^{[s]} \bydef Y_{0,1}^{[s]} + \mathcal{L}_{\infty}Y_{0,2}^{[s]} + Y_{0,3}^{[s]} + \mathcal{L}_{\infty} Y_{0,4}^{[s]}
\end{align}
where
{\footnotesize\begin{align}
    &Y_{0,1}^{[s]} \bydef \left\| A_{c}^N(s)\underset{5Nc}{*} \left(\mathscr{L}_p\overline{\mathbf{W}}(s) + \mathscr{G}_p(\overline{\mathbf{W}})\right)  \right\|_{X_{\tau,\mathrm{con}}} \\
    &Y_{0,2}^{[s]} \bydef \left\|(\bpi^{\leq N_0} - \bpi^{\leq N}) L_p\overline{\mathbf{U}}(s) + \begin{bmatrix}
        (\Pi^{\leq 4N_0} - \Pi^{\leq N}) \left(\overline{\Psi}_p(s)  \underset{4Nc}{*}\overline{\Phi}_{p,\mathrm{inv}}(s)\right) \\ 0
    \end{bmatrix}\right\|_{1,\tau,\mathrm{con}} \\
    &Y_{0,3}^{[s]} \bydef \frac{\left\|A_{c}^N(s) \underset{9Nc}{*}\begin{bmatrix} 0 \\
    \overline{\Psi}_p(s)\underset{9Nc}{*}\overline{\Phi}_{p,\mathrm{inv}}(s)\underset{9Nc}{*}\left(1-\overline{\Phi}_p(s)\underset{9Nc}{*}\overline{\Phi}_{p,\mathrm{inv}}(s)\right) \\ 0 \end{bmatrix}\right\|_{X_{\tau,\mathrm{con}}}}{1-\|1-\overline{\Phi}_p(s)\underset{4Nc}{*}\overline{\Phi}_{p,\mathrm{inv}}(s)\|_{1,\tau,\mathrm{con}}} \\
    &Y_{0,4}^{[s]} \bydef    \frac{\left\|(\Pi^{\leq 8N_0} - \Pi^{ \leq N})\left(\overline{\Psi}_p(s)\underset{8Nc}{*}\overline{\Phi}_{p,\mathrm{inv}}(s)\underset{8Nc}{*}\left(1 - \overline{\Phi}_p(s)\underset{8Nc}{*}\overline{\Phi}_{p,\mathrm{inv}}(s)\right)\right)\right\|_{1,\tau,\mathrm{con}}}{1-\|1-\overline{\Phi}_p(s)\underset{4Nc}{*}\overline{\Phi}_{p,\mathrm{inv}}(s)\|_{1,\tau,\mathrm{con}}}.
\end{align}}
Then, it follows that $\underset{s \in [-1,1]}{\sup}\|A_{c}(s) F_{c}(\overline{\mathbf{W}}(s))\|_{X_{\tau,\mathrm{con}}} \leq Y_0^{[s]}$.
 \end{lemma}
 \begin{proof}
The proof follows the steps of Lemma \ref{lem : bound Y_0 periodic} except that we introduce the corresponding Chebyshev product. In particular, note that $A^N_{c}(s)$ is a polynomial of order $Nc$ with respect to $s$, and $\mathscr{G}_p(\overline{\mathbf{W}})$ is that of order $4Nc$. Hence, we use a product of size $5Nc$. Noting that $\overline{\Psi}_p(s), \overline{\Phi}_p(s),$ and $\overline{\Phi}_{p,\mathrm{inv}}(s)$ are of order $2N_c$ yields the reamining product sizes.
 \end{proof}
 Let us now compute $Z_2^{[s]}$.
 \begin{lemma}\label{lem : Z2 branches}
Let $R > 0$. Define $Z_2^{[s]} > 0$ as 
\begin{align}
    Z_2^{[s]} \bydef 2\nu_1\left(\|A_{c}^N\|_{\mathcal{B}(X_{e,\tau,\mathrm{con}})} + \mathcal{L}_{\infty}\right)(Z_{2,1}^{[s]}Z_{2,2}^{[s]} + Z_{2,3}^{[s]})
\end{align}
where
{\small\begin{align}
    &Z_{2,1}^{[s]} \bydef \biggl\|\nu_2^2 \overline{U}_1(s) \underset{4Nc}{*}\overline{U}_1(s)\underset{4Nc}{*}\overline{U}_1(s)\underset{4Nc}{*} \overline{U}_2(s) + \nu \nu_2 \overline{U}_1(s)\underset{4Nc}{*}\overline{U}_1(s)\underset{4Nc}{*}\overline{U}_1(s) \\
    &+ 3\nu \nu_2 \overline{U}_1(s)\underset{4Nc}{*}\overline{U}_1(s) - 3\nu_2 \overline{U}_1(s)\underset{4Nc}{*} \overline{U}_2(s) - \nu - \overline{U}_2(s)\biggr\|_{1,\tau,\mathrm{con}} \\
    &+\biggl(\left\|-1-3\nu_2 \overline{U}_1(s)+\nu_2^2 \overline{U}_1(s)\underset{3Nc}{*}\overline{U}_1(s)\underset{3Nc}{*}\overline{U}_1(s)\right\|_{1,\tau,\mathrm{con}} \\
    &+ \left\|6\nu\nu_2 \overline{U}_1(s) + 3 \nu \nu_2 \overline{U}_1(s)\underset{3Nc}{*}\overline{U}_1(s) - 3\nu_2 \overline{U}_2(s) + 3\nu_2^2 \overline{U}_1(s)\underset{3Nc}{*}\overline{U}_1(s)\underset{3Nc}{*} \overline{U}_2(s)\right\|_{1,\tau,\mathrm{con}}\biggr)R \\
    &+ \left(\|-3\nu_2 + 3\nu_2^2 \overline{U}_1(s)\underset{2Nc}{*}\overline{U}_1(s)\|_{1,\tau,\mathrm{con}} + \left\|3\nu \nu_2 + 3\nu \nu_2 \overline{U}_1(s) + 3\nu_2^2 \overline{U}_1(s)\underset{2Nc}{*}\overline{U}_1(s)\right\|_{1,\tau,\mathrm{con}}\right) R^2 \\
    &+ (\left\|3\nu_2^2 \overline{U}_1(s) \right\|_{1,\tau,\mathrm{con}} + \left\|\nu \nu_2 + \nu_2^2 \overline{U}_2(s)\right\|_{1,\tau,\mathrm{con}})R^3 + \nu_2^2 R^4,
    \end{align}}
    \begin{align}
    &Z_{2,2}^{[s]} \bydef \frac{\|\overline{\Phi}_{p,\mathrm{inv}}(s)\|_{1,\tau,\mathrm{con}}^3}{\left(1- \|1-\overline{\Phi}_p(s)\underset{4Nc}{*}\overline{\Phi}_{p,\mathrm{inv}}(s)\|_{1,\tau,\mathrm{con}} - R\|\overline{\Phi}_{p,\mathrm{inv}}(s)\|_{1,\tau,\mathrm{con}}\right)^3}, \\
    &Z_{2,3}^{[s]} \bydef \frac{(\|1 - \nu_2 \overline{U}_1(s)\underset{2Nc}{*}\overline{U}_1(s)\|_{1,\tau,\mathrm{con}} + 2\nu_2 \|\overline{U}_1(s)\|_{1,\tau,\mathrm{con}} R + \nu_2 R^2)\|\overline{\Phi}_{p,\mathrm{inv}}(s)\|_{1,\tau,\mathrm{con}}^2}{\left(1 - \|1-\overline{\Phi}_p(s)\underset{4Nc}{*}\overline{\Phi}_{p,\mathrm{inv}}(s)\|_{1,\tau,\mathrm{con}} - R \|\overline{\Phi}_{p,\mathrm{inv}}(s)\|_{1,\tau,\mathrm{con}}\right)^2}.
\end{align}
Then, it follows that $\underset{s \in [-1,1]}{\sup}\underset{\mathbf{U} \in B_R(\overline{\mathbf{U}})}{\sup}\|A_{c}(s)D^2F_{c}(\mathbf{W}(s)) \|_{\mathcal{B}(X_{e,\tau,\mathrm{con}},\mathcal{B}(X_{e,\tau,\mathrm{con}}))} \leq Z_2^{[s]}.$
 \end{lemma}
 \begin{proof}
 To begin, observe that
 {\scriptsize\begin{align}
     \|A_{c}(s)D^2F_{c}(\mathbf{W}(s)) \|_{\mathcal{B}(X_{e,\tau,\mathrm{con}},\mathcal{B}(X_{e,\tau,\mathrm{con}}))} &\leq \|A_{c}(s)\|_{\mathcal{B}(X_{e,\tau,\mathrm{con}})}\left\| D^2F_{c}(\mathbf{W}(s))\right\|_{\mathcal{B}(X_{e,\tau,\mathrm{con}},\mathcal{B}(X_{e,\tau,\mathrm{con}}))} \\
    &\leq (\|A^N_{c}(s)\|_{\mathcal{B}(X_{e,\tau,\mathrm{con}})} + \mathcal{L}_{\infty})\left\| D^2F_{c}(\mathbf{W}(s))\right\|_{\mathcal{B}(X_{e,\tau,\mathrm{con}},\mathcal{B}(X_{e,\tau,\mathrm{con}}))}.
 \end{align}}
 Now, by definition, observe that
 \begin{align}
     \left\| D^2F_{c}(\mathbf{W}(s))\right\|_{\mathcal{B}(X_{e,\tau,\mathrm{con}},\mathcal{B}(X_{e,\tau,\mathrm{con}}))} &= \left\| \begin{bmatrix}
         0 & 0 \\
         0 & D^2F_p(\nu_5(s),\mathbf{U}(s))
     \end{bmatrix}\right\|_{\mathcal{B}(X_{e,\tau,\mathrm{con}},\mathcal{B}(X_{e,\tau,\mathrm{con}}))} \\
     &= \|D^2F_p(\nu_5(s),\mathbf{U}(s))\|_{\mathcal{B}(\ell^1_{e,\tau,\mathrm{con}},\mathcal{B}(\ell^1_{e,\tau,\mathrm{con}}))}
 \end{align}
 From this point on, the proof follows from Lemma \ref{lem : Z2 periodic}. Our only task is to introduce the correct Chebyshev product sizes, which follows from the definition of the involved terms.
 \end{proof}
  Let us now compute the $Z_1^{[s]}$ bound. To begin, given $U \in \ell^2$, define $U(s)^{\star}$ as the dual in $\ell^2$ of $U$. More specifically, we have 
\begin{align}
    U^{\star}(V) \bydef (U,V)_2 ~~~~ \text{ for all } V \in \ell^2.
    \label{dual definition}
\end{align}
Using \eqref{dual definition}, we introduce 
 {\small\begin{align}
     &V_{p,1}^N(s) \bydef \Pi^{\leq N} \left(\overline{\Psi}_{p,1}(s) \underset{7Nc}{*} \overline{\Phi}_{p,\mathrm{inv}}(s)\underset{7Nc}{*}\overline{\Phi}_{p,\mathrm{inv}}(s)\right), ~ V_{p,2}^N(s) \bydef \Pi^{\leq N} \left(\overline{\Psi}_{p,2} \underset{3Nc}{*} \overline{\Phi}_{p,\mathrm{inv}}(s)\right), ~ \\
     &DG_{c}^N(\overline{\mathbf{W}}(s)) \bydef \begin{bmatrix}
         0 & (\Pi^{\leq N}\dot{U}_1(s))^* & (\Pi^{\leq N}\dot{U}_2(s))^* \\
         0 &\mathbb{V}_{p,1}^N(s) & \mathbb{V}_{p,2}^N(s) \\
         -\nu_3 & 0 & 0
     \end{bmatrix}, ~ DG_{c}(\overline{\mathbf{W}}(s)) \bydef \begin{bmatrix}
         0 & \dot{U}_1(s)^* & \dot{U}_2(s)^* \\
         0 & \mathbb{V}_{p,1}(s) & \mathbb{V}_{p,2}(s) \\
         -\nu_3 & 
         0 & 0
     \end{bmatrix}.
 \end{align}}
 We now state the following lemma.
 \begin{lemma}\label{lem : Z1 branches}
Let $Z_{\infty}^{[s]} > 0$ be defined as 
\begin{align}
&Z_{\infty}^{[s]} \bydef (\|A_{c}^N(s)\|_{\mathcal{B}(X_{e,\tau,\mathrm{con}})} + \mathcal{L}_{\infty})(Z_{\infty,1}^{[s]} + Z_{\infty,2}^{[s]} + Z_{\infty,3}^{[s]} + Z_{\infty,4}^{[s]}) 
\end{align}
where 
{\footnotesize\begin{align}
&Z_{\infty,1}^{[s]} \bydef \|\dot{U}_1(s) - \Pi^{\leq N}\dot{U}_1(s)\|_{1,\tau,\mathrm{con}} + \|\dot{U}_2(s) - \Pi^{\leq N}\dot{U}_2(s)\|_{1,\tau,\mathrm{con}} \\
&Z_{\infty,2}^{[s]} \bydef \|V_{p,1}^N(s) - \mathscr{V}_{p,1}(s)\|_{1,\tau,\mathrm{con}} \\
&Z_{\infty,3}^{[s]} \bydef  \frac{\left\|\overline{\Psi}_{p,1}(s)\underset{15N_c}{*}\overline{\Phi}_{p,\mathrm{inv}(s)}\underset{15N_c}{*}\overline{\Phi}_{p,\mathrm{inv}(s)}\underset{15N_c}{*}\left(1-\overline{\Phi}_p(s)\underset{15N_c}{*}\overline{\Phi}_p(s)\underset{15N_c}{*}\overline{\Phi}_{p,\mathrm{inv}}(s)\underset{15N_c}{*}\overline{\Phi}_{p,\mathrm{inv}}(s)\right)\right\|_{1,\tau,\mathrm{con}}}{1-\left\|\left(1-\overline{\Phi}_p(s)\underset{8N_c}{*}\overline{\Phi}_p(s)\underset{8N_c}{*}\overline{\Phi}_{p,\mathrm{inv}}(s)\underset{8N_c}{*}\overline{\Phi}_{p,\mathrm{inv}}(s)\right)\right\|_{1,\tau,\mathrm{con}}} \\
&Z_{\infty,4}^{[s]} \bydef \|V_{p,2}^N - \mathscr{V}_{p,2}\|_{1,\tau} +\frac{\left\|\overline{\Psi}_{p,2}(s)\underset{7N_c}{*}\overline{\Phi}_{p,\mathrm{inv}}(s)\underset{7N_c}{*}\left(1-\overline{\Phi}_p(s)\underset{7N_c}{*}\overline{\Phi}_{p,\mathrm{inv}}(s)\right)\right\|_{1,\tau,\mathrm{con}}}{1-\left\|1-\overline{\Phi}_p(s)\underset{4N_c}{*}\overline{\Phi}_{p,\mathrm{inv}}(s)\right\|_{1,\tau,\mathrm{con}}}
\end{align}}
Then, let $Z_1^{[s]} > 0$ be defined as
\begin{align}
    Z_1^{[s]} \bydef Z_{1,1}^{[s]} + Z_{1,2}^{[s]} + Z_{\infty}^{[s]}
\end{align}
where
\begin{align}
    &Z_{1,1}^{[s]} \bydef \left\|\mathscr{P}^{\leq N}\left(I_d - A_{c}(s)\underset{8N_c}{*}\left(\mathscr{L}_p + DG_{c}^N(\overline{\mathbf{W}}(s))\right)\right)\mathscr{P}^{\leq 2N}\right\|_{\mathcal{B}(X_{e,\tau,\mathrm{con}})} \\
    &Z_{1,2}^{[s]} \bydef  2\mathcal{L}_{\infty}(\|V_{p,1}^N(s)\|_{1,\tau,\mathrm{con}} + \|V_{p,2}^N(s)\|_{1,\tau,\mathrm{con}}).
\end{align}
Then, it follows that $\|I_d - A_{c}(s) DF_{c}(\overline{\mathbf{W}}(s))\|_{\mathcal{B}(X_{\tau,\mathrm{con}})} \leq Z_1^{[s]}$.
 \end{lemma}
 \begin{proof}
 After beginning in a similar way to Lemma \ref{lem : Z1 periodic}, we analyze the $Z_{\infty}$ bound. 
 \begin{align}
    &\left\|DG_{c}(\overline{\mathbf{W}}(s)) - DG^N_{c}(\overline{\mathbf{W}}(s))\right\|_{\mathcal{B}(X_{e,\tau,\mathrm{con}})}\\
    &= \left\|\begin{bmatrix}
         0 & \dot{U}_1(s)^* - (\Pi^{\leq N}\dot{U}_1(s))^* & \dot{U}_2(s)^* - (\Pi^{\leq N}\dot{U}_2(s))^* \\
         0 & \mathbb{V}_{p,1}(s) - \mathbb{V}_{p,1}^N(s) & \mathbb{V}_{p,2}(s) - \mathbb{V}_{p,2}^N(s) \\
         0 & 
         0 & 0
     \end{bmatrix} \right\|_{\mathcal{B}(X_{e,\tau,\mathrm{con}})} \\
     &\leq \|\dot{U}_1(s) - \Pi^{\leq N}\dot{U}_1(s)\|_{1,\tau,\mathrm{con}} + \|\dot{U}_2(s) - \Pi^{\leq N}\dot{U}_2(s)\|_{1,\tau,\mathrm{con}} + Z_{\infty,2}^{[s]} + Z_{\infty,3}^{[s]} + Z_{\infty,4}^{[s]}
 \end{align}
 where the last step followed from the estimate performed in lemma \ref{lem : Z1 periodic}. The first two terms can be bounded directly. This leads to the $Z_{\infty,1}^{[s]}$ bound.
 Note that $A_{c}(s)$ is a polynomial of order $N_c$ with respect to $s$. Also, $DG_{c}^N(\overline{W}(s))$ is a polynomial of order $7N_c$ with respect to $s$, so we use that many coefficients to represent the product. The rest of the proof follows from Lemma \ref{lem : Z1 periodic}.
 \end{proof}
 \subsection{Constructive existence proofs of branches of periodic solutions}
 In this section, we present our proofs of existence (and local uniqueness) of branches of periodic solutions to the 1D Thomas model. 
 \begin{theorem}[\bf Branch of Periodic Solutions in Thomas]\label{th : branch of periodic solution in thomas}
Let $\nu = 0.1764, \nu_1 = 8, \nu_2 = 1, \nu_3 = 0.28, \nu_4 = 21$. Also choose $\tau = 1.0$ and $d = 5$. Then there exists a unique solution to \eqref{eq : psuedo arclength system}. Moreover, this solution corresponds to a branch of periodic solutions in \eqref{eq:thomas}.
\end{theorem}
\begin{proof}
We perform the proof of the full branch in various segments. We denote each segment by $s_i$ for $i \in \{1,\dots,10\}$. In order to show that the branch is one continuous segment made up of the smaller segments, we use the argument from \cite{whitham_cadiot}. Suppose we prove segments $s_i$ and $s_{i+1}$ for all $r \in [r^{[i]}_{min},r^{[i]}_{max}]$. As $s \in [-1,1]$, we must show that
\begin{align}
    \overline{B_{r^{[i]}_{min}}(\overline{W}_i(-1))} \subset \overline{B_{r^{[i+1]}_{max}}(\overline{W}_{i+1}(1))}.
\end{align}
To do so, we derive an explicit condition. Observe that any element in $\overline{B_{r^{[i]}_{min}}(\overline{W}_i(-1))}$ can be written as $\overline{W}_i(-1) + \mathbf{H}$ for some $\mathbf{H} \in \overline{B_{r^{[i]}_{min}}(0)}$. So, notice that
{\small\begin{align}
    \|\overline{W}_i(-1) + \mathbf{H} - \overline{W}_{i+1}(1)\|_{X_{\tau}} \leq \|\overline{W}_i(-1) - \overline{W}_{i+1}(1)\|_{X_{\tau}} + \|\mathbf{H}\|_{X_{\tau}} 
    &\leq \|\overline{W}_i(-1) - \overline{W}_{i+1}(1)\|_{X_{\tau}} + r^{[i]}_{min}.
\end{align}}
Hence, if we can show 
\begin{align}
    \|\overline{W}_i(-1) - \overline{W}_{i+1}(1)\|_{X_{\tau}} + r^{[i]}_{min} \leq r^{[i+1]}_{max},\label{continuity condition}
\end{align}
then we have by uniqueness that we have continuity on the branch. We present the results in the following table. Note that values of $\nu_5$ start and $\nu_5$ end are approximate values.
\begin{table}[h!]
\centering
\small
\begin{tabular}{|c|c|c|c|c|c|c|c|c|c|}
\hline
~ & $\nu_5$ Start & $\nu_5$ End & $N_0$ & $N$ & $N_c$ & $R$ & $Y_0^{[s]}$ & $Z_2^{[s]}$ & $Z_1^{[s]}$ \\ \hline
$s_1$  & 67.59& 67.5272& 90& 80& 127& $9\times10^{-7}$&$1.31\times10^{-7}$ &$1.11\times10^6$ &0.2561\\ \hline
$s_2$  & 67.5272& 67.5247&100 &100 &63 &$1\times10^{-6}$ &$4.63\times10^{-8}$ &$9.333\times10^5$ &0.14221\\ \hline
$s_3$  &67.5247 &67.713 &120 &120 &63 & $7\times10^{-7}$&$2.51\times10^{-7}$ &$1.55\times10^6$ &0.1063\\ \hline
$s_4$  & 67.713& 68.0124& 110& 110& 127& $7\times10^{-7}$& $8.27\times10^{-8}$& $1.292\times10^6$& 0.13184\\ \hline
$s_5$  &68.0124 & 69.5133& 100& 90& 63& $7\times10^{-7}$& $4.555\times10^{-8}$& $2.902\times10^6$& 0.23503\\ \hline
$s_6$  &69.5133 &73.011 &160 & 160& 63& $7\times10^{-7}$&$5.93003\times10^{-8}$& $7.006\times10^6$& 0.0876\\ \hline
$s_7$  &73.011 &75.2063 & 150& 150& 63& $7\times10^{-7}$& $6.33\times10^{-8}$& $6.08\times10^6$&0.1092\\ \hline
$s_8$ &75.2063 &76.4& 110& 100& 31& $4\times10^{-7}$ &$8.47\times10^{-8}$& $3.2933\times10^6$ & 0.2236\\ \hline
$s_9$  & 76.4& 77.0071& 90& 90& 63& $5\times10^{-7}$&$6.74\times10^{-8}$ &$2.4575\times10^6$& 0.23357\\ \hline
$s_{10}$ & 77.0071& 77.196& 90& 90& 16& $5\times10^{-7}$& $1.13\times 10^{-7}$& $2.259503\times10^6$& 0.16872\\ \hline
\end{tabular}
\caption{Values of each segment's proof.}
\end{table}
We prove that each of these these values satisfy Theorem \ref{th: radii polynomial continuation} and the continuity condition \eqref{continuity condition}. 
\end{proof}
 \begin{figure}[H]
    \centering
    \begin{subfigure}[b]{0.3\textwidth}
        \centering
        \epsfig{figure=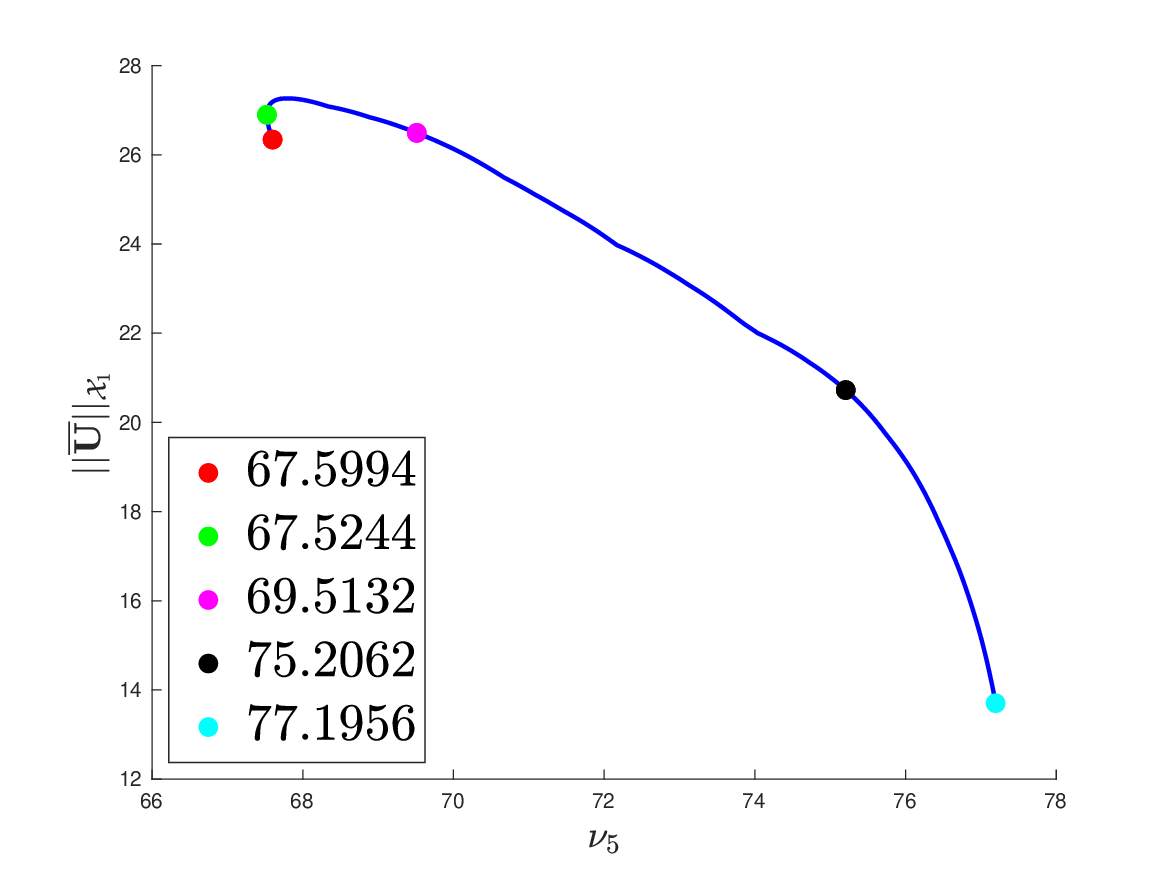, width=\textwidth}
        \caption{Approximation of a branch of periodic solutions in the Thomas model.}\label{fig : branch thomas}
    \end{subfigure}
    \hfill
    \begin{subfigure}[b]{0.3\textwidth}
    \centering
        \epsfig{figure=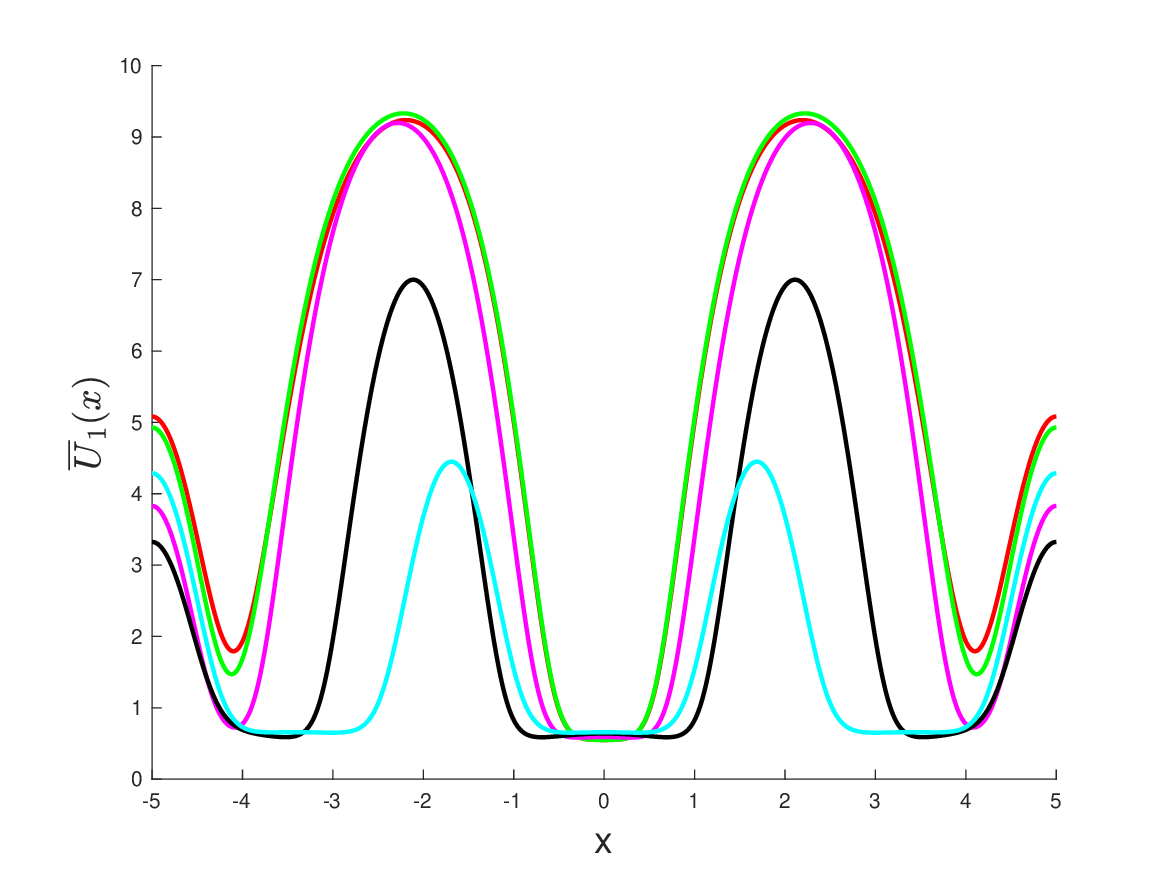, width=\textwidth}
        \caption{The value of the $\overline{U}_1$ plotted on $(-5,5)$ at the indicated points in Figure \ref{fig : branch thomas}}\label{fig : U1 thomas branch}
    \end{subfigure}
    \hfill
    \begin{subfigure}[b]{0.3\textwidth}
    \centering
        \epsfig{figure=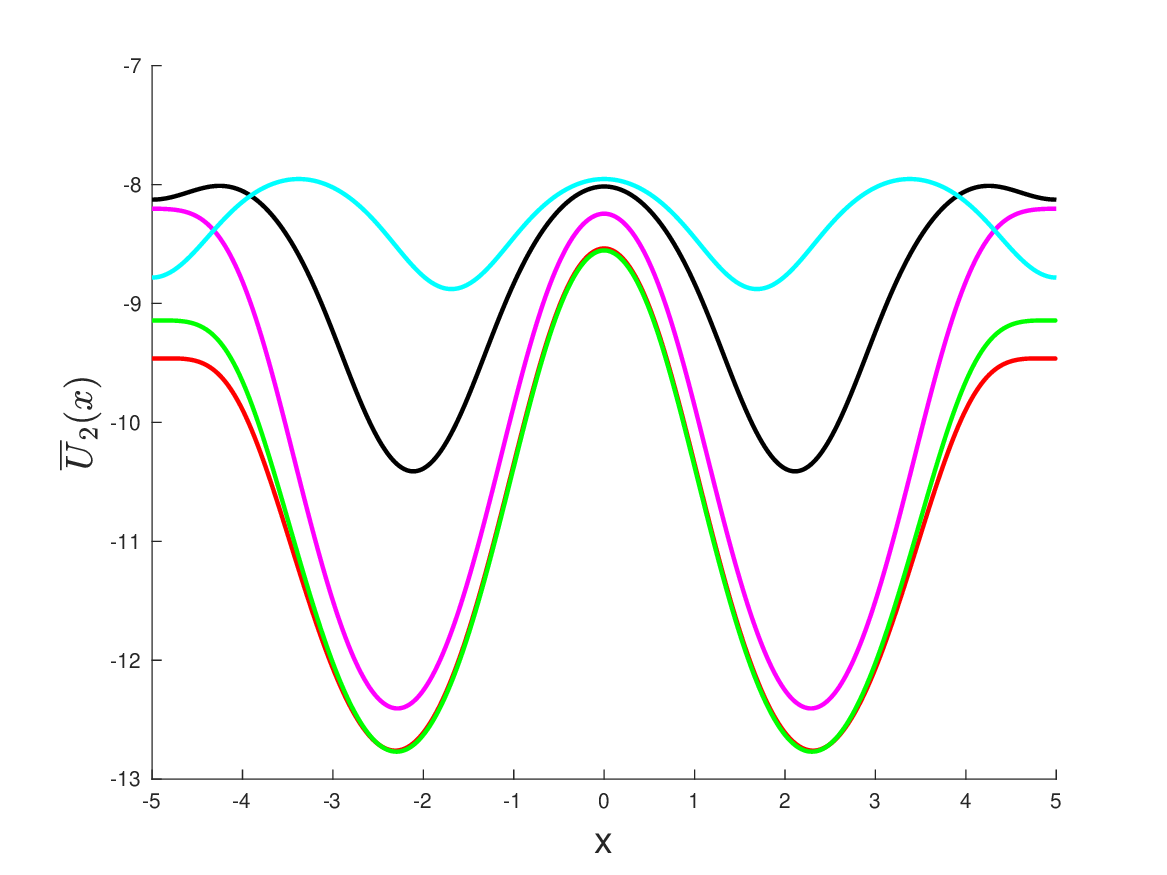, width=\textwidth}
        \caption{The value of $\overline{U}_2$ plotted on $(-5,5)$ at the indicated points in Figure \ref{fig : branch thomas}.}\label{fig : U2 Thomas branch}
    \end{subfigure}
\end{figure}
At the conclusion of the branch proven in Theorem \ref{th : branch of periodic solution in thomas}, we performed further continuation numerically. We conjecture that the branch terminates at the constant steady state solution, which is similar to what the authors of \cite{thomas_1d} obtained in their branch of periodic solutions. We also provide some additional information regarding the branch. The branch began at $\nu_5 \approx 67.59$ and ended at $\nu_5 \approx 77.196$. On $s_2$, we noticed the value of $\nu_5$ get as small as approximately $67.524$. This provides an approximate parameter range for $\nu_5$, which is the parameter we solved for. Note that all other parameters are held fixed, so our branch is valid for those specific choices.
We also conjecture that there is a saddle-node bifurcation occurring in the $s_2$ segment of the branch, but we do not verify this.
If one was interested in rigorously verifying the saddle-node bifurcation, similar approaches to those of \cite{jp_saddle_node1,jp_saddle_node,blanco_cadiot_fassler_saddle} could be used. If we are able to verify that our proof of the saddle-node is on the branch using some kind of continuity condition similar to \eqref{continuity condition}, then we could obtain a rigorous enclosure of a saddle-node bifurcation on the branch. Additionally, since saddle-nodes provide us with a relative extrema, we could verify that this extrema is a relative minimum. Hence, we could also identity a rigorous parameter range. We consider it a possible future project. Another potential future project would be to compute the (arc)length of the branch proven in Theorem \ref{th : branch of periodic solution in thomas}. Computing such a quantity is highly non-trivial. It would require us to compute 
\begin{align}
    \mathscr{L} \bydef \int_{-1}^1 \| \tilde{\mathbf{W}}'(s)\|_2^2 ds,\label{arclength}
\end{align}
where $'$ denotes the derivative with respect to $s$. The problematic is that one does not have access to $\tilde{\mathbf{W}}'(s)$. As a result, obtaining an exact arclength is difficult if not impossible. If one attempted to find an upper bound for the arclength, one approach could be to write
\begin{align}
    \int_{-1}^1 \| \tilde{\mathbf{W}}'(s)\|_2^2 ds &= \int_{-1}^1 \| \tilde{\mathbf{W}}'(s) - \overline{\mathbf{W}}'(s) + \overline{\mathbf{W}}'(s)\|_2^2 ds\\
    &\leq \int_{-1}^1 \|\overline{\mathbf{W}}'(s)\|_{2}^2 ds + \int_{-1}^1 \|\tilde{\mathbf{W}}'(s) - \overline{\mathbf{W}}'(s)\|_{2}^2 ds \\
    &\bydef \overline{\mathscr{L}} + \int_{-1}^1 \|\tilde{\mathbf{W}}'(s) - \overline{\mathbf{W}}'(s)\|_{2}^2 ds .
\end{align}
The term $\overline{\mathscr{L}}$ is now finite, and computable via rigorous numerics. What remains is control $\|\tilde{\mathbf{W}}'(s) - \overline{\mathbf{W}}'(s)\|_{2}$. The problematic is that we do not have access to such a relation with the derivatives. One would have to perform some additional estimates and attempt to find such a relation that would allow us to estimate \eqref{arclength}. Such a relation could be difficult to find. One possibility could be to change the norm on the Chebyshev coefficients in order to obtain a more convenient estimate. We consider this another future project.
\section{Conclusion}\label{sec : conclusion}
In this paper, we provided the necessary tools to rigorously compute localized patterns, periodic solutions, and branches of periodic solutions in the 1D Thomas model. For localized patterns, we adapted the methodology developed by the authors of \cite{unbounded_domain_cadiot} to the Thomas model to perform rigorous proofs. This required the estimates introduced by the authors of \cite{maxime_paper_continuation,olivier_kevin_paper} to handle the nonpolynomial nonlinearity. Furthermore, we improved one of these estimates in Lemma \ref{lem : phi bound stronger} which is a general adaptation that could be used beyond the Thomas model. We then performed the computer-assisted analysis to estimate the rigorous bounds, specific to the 1D Thomas model. We then moved to periodic solutions, where we again adapted previously developed works to the Thomas model in this case. This again required the estimates of \cite{maxime_paper_continuation,olivier_kevin_paper}, which were useful in handling the nonpolynomial nonlinearity directly. Finally, we used our work on periodic solutions along with the approach developed by the author of \cite{polynomial_chaos} to obtain a rigorous proof of a branch of periodic solutions. With all three of these cases taken care of, one has an approach for rigorously verifying some of the results obtained in \cite{thomas_1d}. This partially answers one of the questions asked the authors of \cite{thomas_1d}. Moreover, our approach is general in the sense that it goes beyond being applicable to the results of \cite{thomas_1d}. It could be applied generally to the Thomas model given a numerical approximation of a localized pattern, periodic solution, and branch of periodic solutions.
\par There are many aspects of this model that remain open for exploration. Firstly, our analysis in this paper was closely related to that of \cite{gs_cadiot_blanco} which studied the Gray-Scott model. In Section 5 of \cite{gs_cadiot_blanco}, the authors consider the case where the system can be reduced to a scalar PDE. A similar reduction exists for the Thomas model. Indeed, if $\lambda_5 = 
\nu_3 \nu - 1 = 0$, then we can reduce the system to a scalar PDE and consider this simpler case. While we briefly investigated this numerically, we were unable to obtain any approximate steady states (periodic or localized). Our Newton method would either diverge or converge to a constant solution. From the perspective of rigorous numerics, reducing to a scalar PDE would significantly reduce the computational complexity, and make any future studies in higher dimensions much more feasible. If solutions with $\lambda_5 = \nu \nu_3 - 1 = 0$ are conjectured numerically, and are of interest, it could be a future work to apply our techniques to them.
\par The authors of \cite{thomas_1d} also presented several bifurcations in the Thomas model. Furthermore, as discussed following Theorem \ref{th : branch of periodic solution in thomas}, we observed what we believe to be a saddle-node bifurcation along the branch proven in Theorem \ref{th : periodic solution in thomas}. While we did not attempt to prove saddle-node bifurcations, we refer to the work of \cite{jp_saddle_node1,jp_saddle_node}. We could have directly applied this approach to attempt to prove the existence of saddle-node bifurcations of periodic solutions. Additionally, finding a way to verify that this saddle-node bifurcation lies on the branch of periodic solutions proven would be of greater interest. It would further strengthen the result of Theorem \ref{th : branch of periodic solution in thomas} and reveal additional bifurcation structure of the branch itself. To take this a step further, in \cite{blanco_cadiot_fassler_saddle}, the authors presented a methodology for proving saddle-node bifurcations attached to localized patterns. This approach could be applied to \eqref{eq : gen_model} to obtain saddle-node bifurcations on unbounded domains. We consider it a future work to investigate the rigorous existence of these bifurcations. \par Proving other types of bifurcations, such as Hopf and cusp bifurcations are considered future work. The cusp bifurcation is of particular interest as, through the use of rigorous numerics via Newton-Kantorovich, it has yet to be applied to an infinite dimensional problem. Indeed, the authors of \cite{jp_cusps} provide the \emph{cusp map} compatible with a Newton-Kantorovich approach, but only perform a finite dimensional example. Applying such a map here would be the first time it has been done to a PDE and on unbounded domains. We consider this a future work of particular interest to the Thomas model. 
\par Finally, we make the same conclusion done by the authors of \cite{thomas_1d}, which is to consider a 2D investigation of localized patterns in the Thomas model. The approach of \cite{unbounded_domain_cadiot} has been applied in 2D in various works (cf. \cite{gs_cadiot_blanco,sh_cadiot,matthieu_lindsay}). We performed a brief numerical investigation in 2D using the approach discussed in Section \ref{sec : u0}, but were unsuccessful in identifying any numerical approximations for localized patterns. We would be interested in investigating the rigorous existence of 2D patterns in the Thomas model should they be found to exist, possiblity through collaboration with those currently studying this problem numerically.
\appendix
\renewcommand{\theequation}{A.\arabic{equation}}
\setcounter{equation}{0}
\section{Computation of \texorpdfstring{$\mathcal{Z}_2$}{Z2} and \texorpdfstring{$Z_2$}{Z2}}
In this appendix, we prove Lemmas \ref{lem : Z2 patterns} and \ref{lem : Z2 periodic}. We will begin with the case for localized patterns.
\subsection{Computation for Localized Solutions}\label{apen Z2}
In this appendix, we prove Lemma \ref{lem : Z2 patterns}. We first state some more formal definitions.
\begin{align}\label{def : qijs}
    &q_{0,1} \bydef \upsilon^{0,1}_{1,0} \overline{u}_1 + \upsilon^{0,1}_{2,0} \overline{u}_1^2 + \upsilon^{0,1}_{3,0} \overline{u}_1^3 + \upsilon^{0,1}_{4,0} \overline{u}_1^4 + \upsilon^{0,1}_{5,0}\overline{u}_1^5 + \upsilon^{0,1}_{6,0} \overline{u}_1^6, \\
    &q_{1,1} \bydef  \upsilon^{1,1}_{1,0}\overline{u}_1 + \upsilon^{1,1}_{2,0} \overline{u}_2^2 + \upsilon^{1,1}_{3,0} \overline{u}_1^3 + \upsilon^{1,1}_{4,0} \overline{u}_1^4 + \upsilon^{1,1}_{5,0} \overline{u}_1^5, \\
    &q_{1,0} \bydef \upsilon^{1,0}_{0,1}\overline{u}_2 + \upsilon^{1,0}_{1,0} \overline{u}_{1} + \upsilon^{1,0}_{1,1} \overline{u}_1 \overline{u}_2 + \upsilon^{1,0}_{2,0} \overline{u}_1^2 + \upsilon^{1,0}_{2,1} \overline{u}_1^2 \overline{u}_2 + \upsilon^{1,0}_{3,0} \overline{u}_1^3 + \upsilon^{1,0}_{3,1} \overline{u}_1^3\overline{u}_2\\
    &+\upsilon^{1,0}_{4,0}\overline{u}_1^4 + \upsilon^{1,0}_{4,1} \overline{u}_1^4 \overline{u}_2 + \upsilon^{1,0}_{5,0} \overline{u}_1^5 + \upsilon^{1,0}_{5,1} \overline{u}_1^5 \overline{u}_2, \\
    &q_{2,1} \bydef \upsilon^{2,1}_{1,0} \overline{u}_1 + \upsilon^{2,1}_{2,0} \overline{u}_1^2 + \upsilon^{2,1}_{3,0}\overline{u}_1^3 + \upsilon^{2,1}_{4,0} \overline{u}_1^4, \\
    &q_{2,0} \bydef \upsilon^{2,0}_{0,1} \overline{u}_2 + \upsilon^{2,0}_{1,0} \overline{u}_1 + \upsilon^{2,0}_{1,1} \overline{u}_1 \overline{u}_2 + \upsilon^{2,0}_{2,0} \overline{u}_1^2 + \upsilon^{2,0}_{2,1} \overline{u}_1^2\overline{u}_2 + \upsilon^{2,0}_{3,0} \overline{u}_1^3 + \upsilon^{2,0}_{3,1} \overline{u}_1^3 \overline{u}_2+\upsilon^{2,0}_{4,0} \overline{u}_1^4 + \upsilon^{2,0}_{4,1} \overline{u}_1^4 \overline{u}_2, \\
    &q_{3,0} \bydef \upsilon^{3,0}_{0,1} \overline{u}_2 + \upsilon^{3,0}_{1,0} \overline{u}_1 + \upsilon^{3,0}_{1,1} \overline{u}_1 \overline{u}_2 + \upsilon^{3,0}_{2,0} \overline{u}_1^2 + \upsilon^{3,0}_{2,1} \overline{u}_1^2 \overline{u}_2 + \upsilon^{3,0}_{3,0} \overline{u}_1^3 + \upsilon^{3,0}_{3,1} \overline{u}_1^3 \overline{u}_2, \\
    &q_{4,0} \bydef \upsilon^{4,0}_{1,0} \overline{u}_1 + \upsilon^{4,0}_{2,0} \overline{u}_1^2 + \upsilon^{4,0}_{0,1} \overline{u}_2 + \upsilon^{4,0}_{1,1} \overline{u}_1 \overline{u}_2 + \upsilon^{4,0}_{2,1} \overline{u}_1^2 \overline{u}_2
\end{align}
where
\begin{align}\label{def : rho and upsilon}
    &\mu_1 \bydef \nu_1 \nu_2, ~ \mu_2 \bydef \lambda_3 + 2\lambda_1 \lambda_3 \nu_2 - \lambda_4 \nu_2, ~ \mu_3 \bydef 2\lambda_1 \nu_1 \nu_2, ~ \mu_4 \bydef 2\lambda_1 \lambda_3 + 2\lambda_3 + 2\lambda_1^2 \lambda_3 \nu_2, \\
    &\mu_5 \bydef \lambda_1^2 \nu_1 \nu_2 - \nu_1, ~ \mu_6 \bydef \lambda_1 \lambda_4 + \lambda_4 + \lambda_1^2 \lambda_4 \nu_2
\end{align}
\begin{align}
    &\upsilon^{0,1}_{0,0} \bydef \mu_5 + 2\lambda_1 \mu_5 + \lambda_1^2 \mu_5 + 2\lambda_1^2 \nu_2 \mu_5 + 2\lambda_1^3 \nu_2 \mu_5 + \lambda_1^4 \nu_2^2 \mu_5 \\
    &\upsilon^{0,1}_{1,0} \bydef \mu_3 + 2\lambda_1 \mu_3 + \lambda_1^2 \mu_3 + 2\lambda_1^2 \nu_2 \mu_3 + 2\lambda_1^3 \nu_2 \mu_3 + \lambda_1^4 \nu_2^2 \mu_3 + 2\mu_5 + 2\lambda_1 \mu_5 + 4\lambda_1 \nu_2 \mu_5 \\
    &+ 6\lambda_1^2 \nu_2 \mu_5 + 4\lambda_1^3 \nu_2^2 \mu_5 \\
    &\upsilon^{0,1}_{2,0} \bydef \mu_1 + 2\lambda_1 \mu_1 + \lambda_1^2 \mu_1 + 2\lambda_1^2 \nu_2 \mu_1 + 2\lambda_1^3 \nu_2 \mu_1 + 2\mu_3 + 2\lambda_1 \mu_3 + 4\lambda_1 \nu_2 \mu_3 + 6\lambda_1^2 \nu_2 \mu_3 \\
    &+ 4\lambda_1^3 \nu_2^2 \mu_3 + \mu_5 + 2\nu_2 \mu_5 + 6\lambda_1 \nu_2 \mu_5 + 6\lambda_1^2 \nu_2^2 \mu_5 \\
    &\upsilon^{0,1}_{3,0} \bydef 2\mu_1 + 2\lambda_1 \mu_1 + 4\lambda_1 \nu_2 \mu_1 + 6 \lambda_1^2 \nu_2 \mu_1 + 4\lambda_1^3 \nu_2^2 \mu_1 + \mu_3 + 2\nu_2 \mu_3 + 6\lambda_1 \nu_2 \mu_3 + 6\lambda_1^2 \nu_2^2 \mu_3 \\
    &+ 2\nu_2 \mu_5 + 4\lambda_1 \nu_2^2 \mu_5 \\
    &\upsilon^{0,1}_{4,0} \bydef \mu_1 + 2\nu_2 \mu_1 + 6\lambda_1 \nu_2 \mu_1 + 6\lambda_1^2 \nu_2^2 \mu_1 + 2\nu_2 \mu_3 + 4\lambda_1 \nu_2^2 \mu_3 + \nu_2^2 \mu_5 \\
    &\upsilon^{0,1}_{5,0} \bydef 2\nu_2 \mu_1 + 4\lambda_1 \nu_2^2 \mu_1 + \nu_2^2 \mu_3 \\
    &\upsilon^{0,1}_{6,0} \bydef \nu_2^2 \mu_1
\end{align}
for the $\mathbb{h}_2$ term, 
\begin{align}
    &\upsilon^{1,1}_{0,0} \bydef \mu_3 + 2\lambda_1 \mu_3 + \lambda_1^2 \mu_3 + 2\lambda_1^2 \nu_2 \mu_3 + 2\lambda_1^2 \nu_2 \mu_3 + 2\lambda_1^3 \nu_2 \mu_3 + \lambda_1^4 \nu_2^2 \mu_3 \\
    &\upsilon^{1,1}_{1,0} \bydef 2\mu_1 + 4\lambda_1 \mu_1 + 2\lambda_1^2 \mu_1 + 4\lambda_1^2 \nu_2 \mu_1 + 4\lambda_1^2 \nu_2 \mu_1 + 4\lambda_1^3 \nu_2 \mu_1 + 2\lambda_1^4 \nu_2^2 \mu_1 + 2\mu_3 + 2\lambda_1 \mu_3 + 4\lambda_1 \nu_2 \mu_3 \\
    &+ 6\lambda_1^2 \nu_2 \mu_3 + 4\lambda_1^3 \nu_2^2 \mu_3 \\
    &\upsilon^{1,1}_{2,0} \bydef 4\mu_1 + 4\lambda_1 \mu_1 + 8\lambda_1 \nu_2 \mu_1 + 12\lambda_1^2 \nu_2 \mu_1 + 8\lambda_1^3 \nu_2^2 \mu_1 + \mu_3 + 2\nu_2 \mu_3 + 6\lambda_1 \nu_2 \mu_3 + 6\lambda_1^2 \nu_2^2 \mu_3 \\
    &\upsilon^{1,1}_{3,0} \bydef 2\mu_1 + 4\nu_2 \mu_1 + 12\lambda_1 \nu_2 \mu_1 + 12\lambda_1^2 \nu_2^2 \mu_1 + 2\nu_2 \mu_3 + 4\lambda_1 \nu_2^2 \mu_3 \\
    &\upsilon^{1,1}_{4,0} \bydef 4\nu_2 \mu_1 + 8\lambda_1 \nu_2^2 \mu_1 + \nu_2^2 \mu_3 \\
    &\upsilon^{1,1}_{5,0} \bydef 2\nu_2^2 \mu_1
\end{align}
for the $\mathbb{h}_1\mathbb{h}_2$ term,
{\footnotesize\begin{align}
    &\upsilon^{1,0}_{0,0} \bydef \mu_4 + 2\lambda_1 \mu_4 + \lambda_1^2 \mu_4 + 2\lambda_1^2 \nu_2 \mu_4 + 2\lambda_1^3 \nu_2 \mu_4 + \lambda_1^4 \nu_2^2 \mu_4 - 2\mu_6 - 2\lambda_1 \mu_6 - 4\lambda_1 \nu_2 \mu_6 - 6\lambda_1^2 \nu_2 \mu_6 - 4\lambda_1^3 \nu_2^2 \mu_6 \\
    &\upsilon^{1,0}_{0,1} \bydef 2\lambda_1 \mu_3 + \lambda_1^2 \mu_3 + 2\lambda_1^2 \nu_2 \mu_3 + 2\lambda_1^3 \nu_2 \mu_3 + \lambda_1^4 \nu_2^2 \mu_3 - 2\mu_5 - 2\lambda_1 \mu_5 - 4\lambda_1 \nu_2 \mu_5 - 6\lambda_1^2 \nu_2 \mu_5 - 4\lambda_1^3 \nu_2^2 \mu_5 \\
    &\upsilon^{1,0}_{1,0} \bydef 2\mu_2 + 4\lambda_1 \mu_2 + 2\lambda_1^2 \mu_2 + 4\lambda_1^2 \nu_2 \mu_2 + 4\lambda_1^3 \nu_2 \mu_2 + 2\lambda_1^4 \nu_2^2 \mu_2 - 2\mu_6 - 4\nu_2 \mu_6 - 12\lambda_1 \nu_2 \mu_6 - 12\lambda_1^2 \nu_2^2 \mu_6 \\
    &\upsilon^{1,0}_{1,1} \bydef 2\mu_1 + 4\lambda_1 \mu_1 + 2\lambda_1^2 \mu_1 + 4\lambda_1^2 \nu_2 \mu_1 + 4\lambda_1^3 \nu_2 \mu_1 + 2\lambda_1^4 \nu_2^2 \mu_1 - 2\mu_5 - 4\nu_2 \mu_5 - 12\lambda_1 \nu_2 \mu_5 - 12\lambda_1^2 \nu_2^2 \mu_5 \\
    &\upsilon^{1,0}_{2,0} \bydef 2\mu_2 + 2\lambda_1 \mu_2 + 4\lambda_1 \nu_2 \mu_2 + 6\lambda_1^2 \nu_2 \mu_2 + 4\lambda_1^3 \nu_2^2 \mu_2 - \mu_4 - 2\nu_2 \mu_4 - 6\lambda_1 \nu_2 \mu_4 - 6\lambda_1^2 \nu_2^2 \mu_4 - 6\nu_2 \mu_6 - 12\lambda_1 \nu_2^2 \mu_6 \\
    &\upsilon^{1,0}_{2,1} \bydef 2\mu_1 + 2\lambda_1 \mu_1 + 4\lambda_1 \nu_2 \mu_1 + 6\lambda_1^2 \nu_2 \mu_1 + 4\lambda_1^3 \nu_2^2 \mu_1 - \mu_3 - 2\nu_2 \mu_3 - 6\lambda_1 \nu_2 \mu_3 - 6\lambda_1^2 \nu_2^2 \mu_3 - 6\nu_2 \mu_5 - 12\lambda_1 \nu_2^2 \mu_5 \\
    &\upsilon^{1,0}_{3,0} \bydef -4\nu_2 \mu_4 - 8\lambda_1 \nu_2^2 \mu_4 - 4\nu_2^2 \mu_6 \\
    &\upsilon^{1,0}_{3,1} \bydef -4\nu_2 \mu_3 - 8\lambda_1 \nu_2^2 \mu_3 - 4\nu_2^2 \mu_5 \\
    &\upsilon^{1,0}_{4,0} \bydef -2\nu_2 \mu_2 - 4\lambda_1 \nu_2^2 \mu_2 - 3\nu_2^2 \mu_4 \\
    &\upsilon^{1,0}_{4,1} \bydef -2\nu_2 \mu_1 - 4\lambda_1 \nu_2^2 \mu_1 - 3\nu_2^2 \mu_3 \\
    &\upsilon^{1,0}_{5,0} \bydef =2\nu_2^2 \mu_2 \\
    &\upsilon^{1,0}_{5,1} \bydef -2\nu_2^2 \mu_1
\end{align}}
for the $\mathbb{h}_1$ term,
\begin{align}
    &\upsilon^{2,1}_{0,0} \bydef \mu_1 + 2\lambda_1 \mu_1 + \lambda_1^2 \mu_1 + 2\lambda_1^2 \nu_2 \mu_1 + 2\lambda_1^3 \nu_2 \mu_1 + \lambda_1^4 \nu_2^2 \mu_1 \\
    &\upsilon^{2,1}_{1,0} \bydef 2\mu_1 + 2\lambda_1 \mu_1 + 4\lambda_1 \nu_2 \mu_1 + 6\lambda_1^2 \nu_2 \mu_1 + 4\lambda_1^3 \nu_2^2 \mu_1 \\
    &\upsilon^{2,1}_{2,0} \bydef \mu_1 2\nu_2 \mu_1 + 6\lambda_1 \nu_2 \mu_1 + 6\lambda_1^2 \nu_2^2 \mu_1 \\
    &\upsilon^{2,1}_{3,0} \bydef 2\nu_2 \mu_1 + 4\lambda_1 \nu_2^2 \mu_1 \\
    &\upsilon^{2,1}_{4,0} \bydef \nu_2^2 \mu_1 
\end{align}
for the $\mathbb{h}_1^2\mathbb{h}_2$ term,
\begin{align}
    &\upsilon^{2,0}_{0,0} \bydef \mu_2 + 2\lambda_1 \mu_2 + \lambda_1^2 \mu_2 + 2\lambda_1^2 \nu_2 \mu_2 + 2\lambda_1^3 \nu_2 \mu_2 + \lambda_1^4 \nu_2^2 \mu_2 - \mu_6 - 2\nu_2 \mu_6 - 6\lambda_1 \nu_2 \mu_6 - 6\lambda_1^2 \nu_2^2 \mu_6 \\
    &\upsilon^{2,0}_{0,1} \bydef \mu_1 + 2\lambda_1 \mu_1 + \lambda_1^2 \mu_1 + 2\lambda_1^2 \nu_2 \mu_1 + 2\lambda_1^3 \nu_2 \mu_1 + \lambda_1^4 \nu_2^2\mu_1 - \mu_5 - 2\nu_2 \mu_5 - 6\lambda_1 \nu_2 \mu_5 - 6\lambda_1^2 \nu_2^2 \mu_5 \\
    &\upsilon^{2,0}_{1,0} \bydef 2\mu_2 + 2\lambda_1 \mu_2 + 4\lambda_1 \nu_2 \mu_2 + 6\lambda_1^2 \nu_2 \mu_2 + 4\lambda_1^3 \nu_2^2 \mu_2 -\mu_4 - 2\nu_2 \mu_5 - 6\lambda_1 \nu_2 \mu_4 \\
    &- 6\lambda_1^2 \nu_2^2 \mu_4 - 6\nu_2 \mu_6 - 12\lambda_1 \nu_2^2 \mu_6 \\
    &\upsilon^{2,0}_{1,1} \bydef 2\mu_1 + 2\lambda_1 \mu_1 + 4\lambda_1 \nu_2 \mu_1 + 6\lambda_1^2 \nu_2 \mu_1 + 4\lambda_1^3 \nu_2^2 \mu_1 - \mu_3 - 2\nu_2 \mu_3 - 6\lambda_1 \nu_2 \mu_3 \\
    &- 6\lambda_1^2 \nu_2^2 \mu_3 - 6\nu_2 \mu_5 - 12\lambda_1 \nu_2^2 \mu_5 \\
    &\upsilon^{2,0}_{2,0} \bydef -6\nu_2 \mu_4 - 12\lambda_1 \nu_2^2 \mu_4 - 6\nu_2^2 \mu_6 \\
    &\upsilon^{2,0}_{2,1} \bydef -6\nu_2 \mu_3 - 12\lambda_1 \nu_2^2 \mu_3 - 6\nu_2^2 \mu_5 \\
    &\upsilon^{2,0}_{3,0} \bydef -4\nu_2 \mu_2 -8\lambda_1 \nu_2^2 \mu_2 - 6\nu_2^2 \mu_4 \\
    &\upsilon^{2,0}_{3,1} \bydef -4\nu_2 \mu_1 - 8\lambda_1 \nu_2^2 \mu_1 -6\nu_2^2 \mu_3 \\
    &\upsilon^{2,0}_{4,0} \bydef -5\nu_2^2 \mu_2 \\
    &\upsilon^{2,0}_{4,1} \bydef -5\nu_2^2 \mu_1
\end{align}
for the $\mathbb{h}_1^2$ term,
\begin{align}
    &\upsilon^{3,0}_{0,0} \bydef -2\nu_2 \mu_6 - 4\lambda_1 \nu_2^2 \mu_6 - 2\nu_2 \\
    &\upsilon^{3,0}_{0,1} \bydef -2\nu_2 \mu_5 - 4\lambda_1 \nu_2^2 \mu_5 \\
    &\upsilon^{3,0}_{1,0} \bydef -2\nu_2 \mu_4 -4\lambda_1 \nu_2^2 \mu_4 - 4\nu_2^2 \mu_6 \\
    &\upsilon^{3,0}_{1,1} \bydef -2\nu_2 \mu_3 - 4\lambda_1 \nu_2^2 \mu_3 - 4\nu_2^2 \mu_5 \\
    &\upsilon^{3,0}_{2,0} \bydef -2\nu_2 \mu_2 - 4\lambda_1 \nu_2^2 \mu_2 - 4\nu_2^2 \mu_4 \\
    &\upsilon^{3,0}_{2,1} \bydef -2\nu_2 \mu_1 - 4\lambda_1 \nu_2^2 \mu_1 - 4\nu_2^2 \mu_3 \\
    &\upsilon^{3,0}_{3,0} \bydef -4\nu_2^2 \mu_2 \\
    &\upsilon^{3,0}_{3,1} \bydef -4\nu_2^2 \mu_1
\end{align}
for the $\mathbb{h}_1^3$ term, 
\begin{align}
    &\upsilon^{4,0}_{0,0} \bydef -\nu_2^2 \mu_6 \\
    &\upsilon^{4,0}_{1,0} \bydef - \nu_2^2 \mu_4 \\
    &\upsilon^{4,0}_{2,0} \bydef -\nu_2^2 \mu_2 \\
    &\upsilon^{4,0}_{0,1} \bydef -\nu_2^2 \mu_5 \\
    &\upsilon^{4,0}_{1,1} \bydef -\nu_2^2 \mu_3 \\
    &\upsilon^{4,0}_{2,1} \bydef -\nu_2^2 \mu_1
\end{align}
for the $\mathbb{h}_1^4$ term.
\begin{proof}
 First, we introduce some notation. 
\begin{align}\label{def : wjs}
&w_1(u_1) \bydef \mathbb{1}_{\om} + u_1 + \lambda_1 \mathbb{1}_{\om} + \nu_2 (u_1 + \lambda_1 \mathbb{1}_{\om})^2 \\
&D_{u_1} w_1(u_1) = \mathbb{1}_{\om} + 2\nu_2 (u_1+\lambda_1) \\
&w_2(\mathbf{u}) \bydef \lambda_3 u_1^2 + \lambda_4 u_1 - \nu_1 u_1 u_2 - \nu_1\lambda_1 u_2, \\
& D_{u_1} w_2(\mathbf{u}) = 2\lambda_3 u_1 + \lambda_4 \mathbb{1}_{\om} - \nu_1\lambda_1 u_2, \\
&D_{u_2} w_2(\mathbf{u}) \bydef -\nu_1 u_1 - \nu_1\lambda_1 \mathbb{1}_{\om}.
\end{align}
Then, observe that 
\begin{align}
    &\mathbb{g}(\mathbf{u}) \bydef -\frac{w_2(\mathbf{u})}{w_1(u_1)} - \lambda_6 u_1 - \lambda_7 u_2, \\
    &D_{u_1}\mathbb{g}(\mathbf{u}) \bydef -\frac{\mathbb{w}_1(u_1)D_{u_1} \mathbb{w}_2(\mathbf{u}) - \mathbb{w}_2(\mathbf{u}) D_{u_1} \mathbb{w}_1(u_1)}{\mathbb{w}_1(u_1)^2} - \lambda_6 \mathbb{1}_{\om}, \\
    &D_{u_2} \mathbb{g}(\mathbf{u}) \bydef -\frac{D_{u_2}\mathbb{w}_2(\mathbf{u})}{\mathbb{w}_1(u_1)} - \lambda_7 \mathbb{1}_{\om}
\end{align}
where $D_{u_j} \mathbb{w}_j$ and $\mathbb{w}_j$ are the multiplication operators associated to $D_{u_j} w_j$ and $w_j$ respectively for each $j \in \{1,2\}$. We begin by expanding the desired expression.
\begin{align}
    \|\mathbb{A}(D\mathbb{F}(\mathbf{u}) - D\mathbb{F}(\overline{\mathbf{u}}))\|_{\mathcal{H}} &= \|\mathbb{B}(D\mathbb{F}(\mathbf{u}) - D\mathbb{F}(\overline{\mathbf{u}}))\|_{\mathcal{H},2} \\
    &= \left\| \mathbb{B} \begin{bmatrix}
        D_{u_1}\mathbb{g}(\mathbf{u}) - D_{u_1}\mathbb{g}(\overline{\mathbf{u}}) & D\mathbb{g}_{u_2}(\mathbf{u}) - D\mathbb{g}_{u_2}(\overline{\mathbf{u}}) \\
        0 & 0 
    \end{bmatrix}\right\|_{\mathcal{H},2} \\
    &=  \left\| \mathbb{B}_{11}\begin{bmatrix}
        D_{u_1}\mathbb{g}(\mathbf{u}) - D_{u_1}\mathbb{g}(\overline{\mathbf{u}}) & D\mathbb{g}_{u_2}(\mathbf{u}) - D\mathbb{g}_{u_2}(\overline{\mathbf{u}})
    \end{bmatrix}\right\|_{\mathcal{H},2}.
\end{align}
Next, let $v \bydef (v_1,v_2) \in \mathcal{H}, \|v\|_{\mathcal{H}} = 1$. Then, we have
{\small\begin{align}
    \|\mathbb{A}(D\mathbb{F}(\mathbf{u}) - D\mathbb{F}(\overline{\mathbf{u}}))\|_{\mathcal{H}} &= \left\| \mathbb{B}_{11}(D_{u_1}\mathbb{g}(\mathbf{u}) - D_{u_1}\mathbb{g}(\overline{\mathbf{u}}))v_1 + (D\mathbb{g}_{u_2}\mathbb{B}_{11}(\mathbf{u}) - D\mathbb{g}_{u_2}(\overline{\mathbf{u}}))v_2\right\|_{2} \\
    &\leq  \left( \|\mathbb{B}_{11}(D_{u_1}\mathbb{g}(\mathbf{u}) - D_{u_1}\mathbb{g}(\overline{\mathbf{u}}))v_1\|_{2} + \|\mathbb{B}_{11}(D\mathbb{g}_{u_2}(\mathbf{u}) - D\mathbb{g}_{u_2}(\overline{\mathbf{u}}))v_2\|_{2}\right).
\end{align}}
Let us now examine each component. We begin with the second.
{\footnotesize\begin{align}
    \left\|\mathbb{B}_{11}(D\mathbb{g}_{u_2}(\mathbf{u}) - D\mathbb{g}_{u_2}(\overline{\mathbf{u}}))v_2\right\|_{2} &= 
        \left\|-\mathbb{B}_{11}\left(\frac{D_{u_2} \mathbb{w}_2(\mathbf{u})}{\mathbb{w}_1(u_1)} - \frac{D_{u_2} \mathbb{w}_2(\overline{\mathbf{u}})}{\mathbb{w}_1(\overline{u}_1)}\right) v_2\right\|_{2} \\
        &= \left\|\mathbb{B}_{11} \frac{\mathbb{w}_1(\overline{u}_1) D_{u_2}\mathbb{w}_2( \mathbf{u}) - \mathbb{w}_1(u_1) D_{u_2} \mathbb{w}_2(\overline{\mathbf{u}})}{\mathbb{w}_1(u_1)\mathbb{w}_1(\overline{u}_1)}v_2\right\|_{2} \\
        &\leq \left\| \frac{\mathbb{1}_{\om}}{\mathbb{w}_1(u_1)\mathbb{w}_1(\overline{u}_1)}\right\|_{2} \|\mathbb{B}_{11}\left(\mathbb{w}_1(\overline{u}_1) D_{u_2}\mathbb{w}_2( \mathbf{u}) - \mathbb{w}_1(u_1) D_{u_2} \mathbb{w}_2(\overline{\mathbf{u}})\right) v_2\|_{2}\label{before introducing q in Z2}
        \end{align}}
Now, let $z \in L^2_e, \|z\|_{2} = 1$. Then, 
\begin{align}
    \left\| \frac{\mathbb{1}_{\om}}{\mathbb{w}_1(u_1)\mathbb{w}_1(\overline{u}_1)}\right\|_{2} &= \left\| \frac{z}{w_1(u_1)w_1(\overline{u}_1)}\right\|_{2} \leq \left\| \frac{1}{w_1(u_1)}\right\|_{\infty}\left\|\frac{1}{w_1(\overline{u}_1)}\right\|_{\infty} \|z\|_{2} \leq \kappa_0^2.\label{kappa0 in second Z2 component}
\end{align}
Now, let $\mathbf{u} = \overline{\mathbf{u}} + \mathbf{h}$ where $\mbf{h} = (h_1,h_2) \in \overline{B_r(0)}\subset \mathcal{H}_e$. Returning to \eqref{before introducing q in Z2}, we estimate
\begin{align}
    &\|\mathbb{B}_{11}\left(\mathbb{w}_1(\overline{u}_1) D_{u_2}\mathbb{w}_2(\mathbf{u}) - \mathbb{w}_1(u_1) D_{u_2} \mathbb{w}_2(\overline{\mathbf{u}})\right) v_2\|_{2} \\
    &=  \|\mathbb{B}_{11}\left(\mathbb{w}_1(\overline{u}_1) D_{u_2}\mathbb{w}_2(\overline{\mathbf{u}} + \mathbf{h}) - \mathbb{w}_1(\overline{u}_1+h_1) D_{u_2} \mathbb{w}_2(\overline{\mathbf{u}})\right) v_2\|_{2}\\
    &= \|\mathbb{B}_{11}\left( (\lambda_1 \nu_1 \nu_2 \mathbb{1}_{\om}+ \nu_1 \nu_2 \overline{\mathbb{u}}_1) \mathbb{h}_1^2 + ((-\nu_1 + \lambda_1^2 \nu_1 \nu_2)\mathbb{1}_{\om} + 2\lambda_1 \nu_1 \nu_2 \overline{\mathbb{u}}_1 + \nu_1 \nu_2 \overline{\mathbb{u}}_1^2) \mathbb{h}_1\right) v_2\|_{2} \\
    &= \|\mathbb{B}_{11}\left( \mathbb{q}_1 \mathbb{h}_1^2 + \mathbb{q}_2 \mathbb{h}_1\right) v_2\|_{2}
\end{align}
where $\mathbb{q}_1$ and $\mathbb{q}_2$ are the multiplication operators associated to $q_1$ and $q_2$ defined in \eqref{def : qjs}. Then, observe that
\begin{align}
    \|\mathbb{B}_{11}\left( \mathbb{q}_1 \mathbb{h}_1^2 + \mathbb{q}_2 \mathbb{h}_1\right) v_2\|_{2} &\leq \|\mathbb{B}_{11} \mathbb{q}_1 \mathbb{h}_1^2 v_2\|_{2} + \|\mathbb{B}_{11} \mathbb{q}_2 \mathbb{h}_1 v_2\|_{2} \\
    &\leq \|\mathbb{B}_{11} \mathbb{q}_1 \|_{2}\|h_1 v_2\|_{2} + \|\mathbb{B}_{11} \mathbb{q}_2\|_{2} \|h_1^2v_2\|_{2} \\
    &\leq \kappa_1 \|\mathbb{B}_{11} \mathbb{q}_1 \|_{2}r + \kappa_1 \|l^{-1}\|_{\mathcal{M}_2}\|\mathbb{B}_{11} \mathbb{q}_2\|_{2} r^2\label{before Z21 and Z22}
\end{align}
where we used Lemma \ref{lem : extra kappa estimates}. Then, we use similar steps to those done in Lemma 3.4 of \cite{sh_cadiot} to obtain
\begin{align}
    &\|\mathbb{B}_{11} \mathbb{q}_1\|_{2} \leq \max\{\lambda_1 \nu_1 \nu_2, \left(\|B_{11}^N Q_1^2 (B_{11}^N)^*\|_{2} + \|Q_1\|_{1}^2\right)^{\frac{1}{2}}\} \bydef \mathcal{Z}_{2,2} \\
    &\|\mathbb{B}_{11} \mathbb{q}_2\|_{2} \leq \max\{|-\nu_1 + \lambda_1^2 \nu_1 \nu_2|, \left(\|B_{11}^N Q_2^2 (B_{11}^N)^*\|_{2} + \|Q_2\|_{1}^2\right)^{\frac{1}{2}}\} \bydef \mathcal{Z}_{2,2}
\end{align}
We then return to \eqref{before Z21 and Z22} and obtain
\begin{align}
    \kappa_1 \|\mathbb{B}_{11} \mathbb{q}_1 \|_{2}r + \kappa_1 \|l^{-1}\|_{\mathcal{M}_2}\|\mathbb{B}_{11} \mathbb{q}_2\|_{2} r^2 &\leq \kappa_1 (\mathcal{Z}^{[1]}_{2} + \|l^{-1}\|_{\mathcal{M}_2} \mathcal{Z}^{[2]}_{2} r)r.\label{final second term to combine}
\end{align}
Finally, we combine \eqref{kappa0 in second Z2 component} and \eqref{final second term to combine} to estimate
\begin{align}
    \left\|\mathbb{B}_{11}(D\mathbb{g}_{u_2}(\mathbf{U}) - D\mathbb{g}_{u_2}(\overline{\mathbf{u}}))v_2\right\|_{2} &\leq \kappa_0^2 \kappa_1 (\mathcal{Z}^{[1]}_{2} + \|l^{-1}\|_{\mathcal{M}_2} \mathcal{Z}^{[2]}_{2} r)r.\label{second component approximated Z2}
\end{align}
We now turn our attention to the first component. 
{\footnotesize\begin{align}
    &\|\mathbb{B}_{11}(D_{u_1}\mathbb{g}(\mathbf{u}) - D_{u_1}\mathbb{g}(\overline{\mathbf{u}}))v_1\|_{2}\label{step with kappa04} \\
    &= \left\|-\mathbb{B}_{11}\left( \frac{\mathbb{w}_1(u_1)D_{u_1} \mathbb{w}_2(\mathbf{u}) - \mathbb{w}_2(\mathbf{u}) D_{u_1} \mathbb{w}_1(u_1)}{\mathbb{w}_1(u_1)^2} - \frac{\mathbb{w}_1(\overline{u}_1)D_{u_1} \mathbb{w}_2(\overline{\mathbf{u}}) - \mathbb{w}_2(\overline{\mathbf{u}}) D_{u_1} \mathbb{w}_1(\overline{u}_1)}{\mathbb{w}_1(\overline{u}_1)^2}\right)v_1\right\|_{2} 
    \\
    &=  \left\|\mathbb{B}_{11}\frac{\mathbb{w}_1(\overline{u}_1)^2(\mathbb{w}_1(u_1)D_{u_1} \mathbb{w}_2(\mathbf{u}) - \mathbb{w}_2(\mathbf{u}) D_{u_1} \mathbb{w}_1(u_1)) - \mathbb{w}_1(u_1)^2(\mathbb{w}_1(\overline{u}_1)D_{u_1} \mathbb{w}_2(\overline{\mathbf{u}}) - \mathbb{w}_2(\overline{\mathbf{u}}) D_{u_1} \mathbb{w}_1(\overline{u}_1))}{\mathbb{w}_1(u_1)^2\mathbb{w}_1(\overline{u}_1)^2}v_1\right\|_{2} \\
    &\leq \kappa_0^4 \|\mathbb{B}_{11}(\mathbb{w}_1(\overline{u}_1)^2(\mathbb{w}_1(u_1)D_{u_1} \mathbb{w}_2(\mathbf{u}) - \mathbb{w}_2(\mathbf{u}) D_{u_1} \mathbb{w}_1(u_1)) - \mathbb{w}_1(u_1)^2(\mathbb{w}_1(\overline{u}_1)D_{u_1} \mathbb{w}_2(\overline{\mathbf{u}}) - \mathbb{w}_2(\overline{\mathbf{u}}) D_{u_1} \mathbb{w}_1(\overline{u}_1))v_1\|_{2}
    \end{align}}
where we used similar steps to those done in \eqref{kappa0 in second Z2 component}. 
Now, observe that
\begin{align}
    &D_{u_1} \mathbb{g}(\mathbf{u}) = -\frac{\mu_1 u_1^2 u_2 + \mu_2 u_1^2 + \mu_3 u_1 u_2 + \mu_4 u_1 + \mu_5 u_2 + \mu_6}{(1 + u_1 + \lambda_1 + \nu_2 (u_1 + \lambda_1)^2)^2}.
\end{align}
Recalling again that $\mathbf{u} = \overline{\mathbf{u}} + \mathbf{h}$, we then obtain
\begin{align}
    &\mathbb{w}_1(\overline{u}_1)^2(\mathbb{w}_1(u_1)D_{u_1} \mathbb{w}_2(\mathbf{u}) - \mathbb{w}_2(\mathbf{u}) D_{u_1} \mathbb{w}_1(u_1)) - \mathbb{w}_1(u_1)^2(\mathbb{w}_1(\overline{u}_1)D_{u_1} \mathbb{w}_2(\overline{\mathbf{u}}) - \mathbb{w}_2(\overline{\mathbf{u}}) D_{u_1} \mathbb{w}_1(\overline{u}_1)) \\
    &= (\upsilon^{0,1}_{0,0} + \mathbb{q}_{0,1}) \mathbb{h}_2 + (\upsilon^{1,1}_{0,0} + \mathbb{q}_{1,1})\mathbb{h}_1 \mathbb{h}_2 + (\upsilon^{1,0}_{0,0} + \mathbb{q}_{1,0})\mathbb{h}_1 + (\upsilon^{2,1}_{0,0} + \mathbb{q}_{2,1})\mathbb{h}_1^2\mathbb{h}_2 + (\upsilon^{2,0}_{0,0} + \mathbb{q}_{2,0})\mathbb{h}_1^2 \\
    &+ (\upsilon^{3,0}_{0,0} + \mathbb{q}_{3,0})\mathbb{h}_1^3 + (\upsilon^{4,0}_{0,0} + \mathbb{q}_{4,0})\mathbb{h}_1^4
\end{align}
where $\mathbb{q}_{i,j}$ are defined as in \eqref{def : qijs} and $\upsilon^{i,j}_{0,0}$ are defined as in \eqref{def : rho and upsilon} for each $i,j \in \mathcal{Q}$ defined as in \eqref{def : set Q}. Then, we have
\begin{align}
    &\left\|\mathbb{B}_{11}\sum_{(i,j) \in \mathcal{Q}}(\upsilon^{i,j}_{0,0} + \mathbb{q}_{i,j}) \mathbb{h}_1^i\mathbb{h}_2^j\right\| \leq \kappa_1 \sum_{(i,j) \in \mathcal{Q}} \|\mathbb{B}_{11}(\upsilon^{i,j}_{0,0} + \mathbb{q}_{i,j})\|_{2}\|l^{-1}\|_{\mathcal{M}_2}^{i + j - 1}r^{i + j}
\end{align}
where we used Lemma \ref{lem : extra kappa estimates} to obtain the previous estimate. Then, again using Lemma 3.4 from \cite{sh_cadiot}, we obtain
\begin{align}
    \|\mathbb{B}_{11}(\upsilon^{i,j}_{0,0} + \mathbb{q}_{i,j})\|_{2} \leq \max\left\{|\upsilon^{i,j}_{0,0}|, \left(\|B_{11}^N Q_{i,j}^2 (B_{11}^N)^*\|_{2} + \|Q_{i,j}\|_{1}^2\right)^{\frac{1}{2}}\right\} \bydef \mathcal{Z}_{2,i,j}.\label{def : Z2ijs}
\end{align}
Hence, using \eqref{def : Z2ijs}, we return to \eqref{step with kappa04} and obtain
\begin{align}
   \|\mathbb{B}_{11}(D_{u_1}\mathbb{g}(\mathbf{U}) - D_{u_1}\mathbb{g}(\overline{\mathbf{u}}))v_1\|_{2} \leq \left(\kappa_0^4 \kappa_1\sum_{(i,j) \in \mathcal{Q}} \mathcal{Z}_{2,i,j} \|l^{-1}\|_{\mathcal{M}_2}^{i + j - 1} r^{i + j - 1}\right) r.\label{first component approximated Z2}
\end{align}
We finally combine \eqref{second component approximated Z2} and \eqref{first component approximated Z2} to obtain the final result
{\footnotesize\begin{align}
    \|\mathbb{A}(D\mathbb{F}(\mathbf{h}+\overline{\mathbf{u}}) - D\mathbb{F}(\overline{\mathbf{u}}))\|_{\mathcal{H}} &\leq \kappa_0^2 \kappa_1 \left(\mathcal{Z}_{2,1} + \|l^{-1}\|_{\mathcal{M}_2} \mathcal{Z}_{2,2} r + \kappa_0^2 \sum_{(i,j) \in \mathcal{Q}} \mathcal{Z}_{2,i,j} \|l^{-1}\|_{\mathcal{M}_2}^{i + j - 1} r^{i + j - 1}\right) r \bydef \mathcal{Z}_2(r) r
\end{align}}
as desired.
\end{proof}
\subsection{Computation for Periodic Solutions}\label{apen : Z2 periodic}
In this appendix, we prove Lemma \ref{lem : Z2 periodic}. 
\begin{proof}
 Firstly, notice that
 \begin{align}
     \|A_pD^2F_p(\mathbf{U})\|_{\mathcal{B}(\ell^1_{e,\tau},\mathcal{B}(\ell^1_{e,\tau}))} &\leq \|A_p\|_{\mathcal{B}(\ell^1_{e,\tau})} \|D^2F_p(\mathbf{U})\|_{\mathcal{B}(\ell^1_{e,\tau},\mathcal{B}(\ell^1_{e,\tau}))} \\
     &\leq (\|A_p^N\|_{\mathcal{B}(\ell^1_{e,\tau})} + \mathcal{L}_{\infty}) \|D^2F_p(\mathbf{U})\|_{\mathcal{B}(\ell^1_{e,\tau},\mathcal{B}(\ell^1_{e,\tau}))}.
 \end{align}
Now, let $\mathbf{W} \bydef (W_1,W_2), \mathbf{V} \bydef (V_1,V_2) \in \ell^1_{e,\tau}$ such that $\|\mathbf{W}\|_{1,\tau} = \|\mathbf{V}\|_{1,\tau} = 1$. Then,
{\scriptsize\begin{align}
    &\|D^2F_p(\mathbf{U})\|_{\mathcal{B}(\ell^1_{e,\tau},\mathcal{B}(\ell^1_{e,\tau}))}
    \\
    &\leq  2\nu_1 \left\|\frac{\nu_2^2 U_1^3 *U_2 + \nu \nu_2 U_1^3 + 3\nu \nu_2 U_1^2 - 3\nu_2 U_1 *U_2 - \nu - U_2}{(1 + U_1 + \nu_2 U_1^2)^3} *W_1*V_1\right\|_{1,\tau} + \nu_1 \left\|\frac{1 - \nu_2 U_1^2}{(1 + U_1 + \nu_2 U_1^2)^2} *W_2*V_1\right\|_{1,\tau} \\
    &+ \nu_1 \left\| \frac{1 - \nu_2 U_1^2}{(1 + U_1 + \nu_2 U_1^2)^2} *W_1*V_2\right\|_{1,\tau}.\label{split Z2 second deriv}
\end{align}}
Let us now focus on the second and third terms in \eqref{split Z2 second deriv}. Firstly, we define $\Phi_p \bydef 1 + U_1 + \nu_2 U_1^2$. Then, using the fact that $\ell^1_{e,\tau}$ is a Banach algebra, we get (for $(k,j) = (2,1), (1,2)$)
{\footnotesize\begin{align}
    \left\| \frac{1 - \nu_2 U_1^2}{(1 + U_1 + \nu_2 U_1^2)^2} *W_k*V_j\right\|_{1,\tau} = \|(1 - \nu_2 U_1^2) *\Phi_p^{-2} *H_k *V_j\|_{1,\tau} &\leq \|1 - \nu_2 U_1^2\|_{1,\tau} \|\Phi_p^{-2}\|_{1,\tau}\|W_k\|_{1,\tau} \|V_j\|_{1,\tau}  \\
    &\leq \|1 - \nu_2 U_1^2\|_{1,\tau} \|\Phi_p^{-1}\|_{1,\tau}^2.
\end{align}}
We now use Lemma \ref{lem : ball phi bound} and the fact that $\mathbf{U} \bydef \overline{\mathbf{U}} + \mathbf{H}$ for some $\mathbf{H} \bydef (H_1,H_2) \in B_R(\overline{U})$ (hence $\|\mathbf{H}\|_{1,\tau} \leq R$) to get
{\scriptsize\begin{align}
    \underset{\mathbf{U} \in B_R(\overline{\mathbf{U}})}{\sup} \|1 - \nu_2 U_1^2\|_{1,\tau} \|\Phi_p^{-1}\|_{1,\tau}^2 &\leq \|1 - \nu_2 (\overline{U}_1 + H_1)^2\|_{1,\tau} \left(\frac{\|\overline{\Phi}_{p,\mathrm{inv}}\|_{1,\tau}}{1 - \|1-\overline{\Phi}_p*\overline{\Phi}_{p,\mathrm{inv}}\|_{1,\tau} - R \|\overline{\Phi}_{p,\mathrm{inv}}\|_{1,\tau}}\right)^2 \\
    &= \|1 - \nu_2 \overline{U}_1^2 - 2\nu_2 \overline{U}_1 *H_1 - \nu_2 H_1^2\|_{1,\tau}\frac{\|\overline{\Phi}_{p,\mathrm{inv}}\|_{1,\tau}^2}{\left(1 - \|1-\overline{\Phi}_p*\overline{\Phi}_{p,\mathrm{inv}}\|_{1,\tau} - R \|\overline{\Phi}_{p,\mathrm{inv}}\|_{1,\tau}\right)^2} \\
    &\leq \frac{(\|1 - \nu_2 \overline{U}_1^2\|_{1,\tau} + 2\nu_2 \|\overline{U}_1\|_{1,\tau} R + \nu_2 R^2)\|\overline{\Phi}_{p,\mathrm{inv}}\|_{1,\tau}^2}{\left(1 - \|1-\overline{\Phi}_p*\overline{\Phi}_{p,\mathrm{inv}}\|_{1,\tau} - R \|\overline{\Phi}_{p,\mathrm{inv}}\|_{1,\tau}\right)^2} \bydef Z_{2,3}.
\end{align}}
For the first term in \eqref{split Z2 second deriv}, observe that
\begin{align}
&\left\|\frac{\nu_2^2 U_1^3 *U_2 + \nu \nu_2 U_1^3 + 3\nu \nu_2 U_1^2 - 3\nu_2 U_1 *U_2 - \nu - U_2}{(1 + U_1 + \nu_2 U_1^2)^3} *W_1*V_1\right\|_{1,\tau} \\
&= \|(\nu_2^2 U_1^3 *U_2 + \nu \nu_2 U_1^3 + 3\nu \nu_2 U_1^2 - 3\nu_2 U_1 *U_2 - \nu - U_2) \Phi_p^{-3} *W_1 *V_1\|_{1,\tau} \\
&\leq \|\nu_2^2 U_1^3 *U_2 + \nu \nu_2 U_1^3 + 3\nu \nu_2 U_1^2 - 3\nu_2 U_1 *U_2 - \nu - U_2\|_{1,\tau} \|\Phi_p^{-3}\|_{1,\tau} \|W_1\|_{1,\tau} \|V_1\|_{1,\tau}\\
&\leq \|\nu_2^2 U_1^3 *U_2 + \nu \nu_2 U_1^3 + 3\nu \nu_2 U_1^2 - 3\nu_2 U_1 *U_2 - \nu - U_2\|_{1,\tau} \|\Phi_p^{-1}\|_{1,\tau}^3\\
&\leq  \frac{\|\nu_2^2 U_1^3 *U_2 + \nu \nu_2 U_1^3 + 3\nu \nu_2 U_1^2 - 3\nu_2 U_1 *U_2 - \nu - U_2\|_{1,\tau}\|\overline{\Phi}_{p,\mathrm{inv}}\|_{1,\tau}^3}{\left(1- \|1-\overline{\Phi}_p*\overline{\Phi}_{p,\mathrm{inv}}\|_{1,\tau} - R\|\overline{\Phi}_{p,\mathrm{inv}}\|_{1,\tau}\right)^3} \\
&\bydef \|\nu_2^2 U_1^3 *U_2 + \nu \nu_2 U_1^3 + 3\nu \nu_2 U_1^2 - 3\nu_2 U_1 *U_2 - \nu - U_2\|_{1,\tau} Z_{2,2}\label{introduce Z211}
\end{align}
where we again used Lemma \ref{lem : ball phi bound}. Now, we again expand in terms of $\overline{\mathbf{U}}$ and $\mathbf{H}$ to estimate
{\scriptsize\begin{align}
    &\|\nu_2^2 U_1^3 *U_2 + \nu \nu_2 U_1^3 + 3\nu \nu_2 U_1^2 - 3\nu_2 U_1* U_2 - \nu - U_2\|_{1,\tau} \\
    &=\|\nu_2^2 (\overline{U}_1+H_1)^3 *(\overline{U}_2+H_2) + \nu \nu_2 (\overline{U}_1+H_1)^3 + 3\nu \nu_2 (\overline{U}_1+H_1)^2 - 3\nu_2 (\overline{U}_1+H_1)* (\overline{U}_2+H_2) - \nu - \overline{U}_2-H_2\|_{1,\tau} \\
    &\leq \|\nu_2^2 \overline{U}_1^3* \overline{U}_2 + \nu \nu_2 \overline{U}_1^3 + 3\nu \nu_2 \overline{U}_1^2 - 3\nu_2 \overline{U}_1 *\overline{U}_2 - \nu - \overline{U}_2\|_{1,\tau} + \|(-1-3\nu_2 \overline{U}_1+\nu_2^2 \overline{U}_1^3) *H_2 \|_{1,\tau} \\
    &+  \|(6\nu\nu_2 \overline{U}_1 + 3 \nu \nu_2 \overline{U}_1^2 - 3\nu_2 \overline{U}_2 + 3\nu_2^2 \overline{U}_1^2 *\overline{U}_2)*H_1\|_{1,\tau}+ \|(-3\nu_2 + 3\nu_2^2 \overline{U}_1^2)*H_1 *H_2\|_{1,\tau} \\
    & + \|(3\nu \nu_2 + 3\nu \nu_2 \overline{U}_1 + 3\nu_2^2 \overline{U}_1)*H_1^2\|_{1,\tau} + \|3\nu_2^2 \overline{U}_1 *H_1^2 *H_2\|_{1,\tau} + \|(\nu \nu_2 + \nu_2^2 \overline{U}_2)*H_1^3\|_{1,\tau} + \nu_2^2 \| H_1^3 *H_2\|_{1,\tau} \\
    &\leq \|\nu_2^2 \overline{U}_1^3* \overline{U}_2 + \nu \nu_2 \overline{U}_1^3 + 3\nu \nu_2 \overline{U}_1^2 - 3\nu_2 \overline{U}_1* \overline{U}_2 - \nu - \overline{U}_2\|_{1,\tau} \\
    &+(\|-1-3\nu_2 \overline{U}_1+\nu_2^2 \overline{U}_1^3\|_{1,\tau} + \|6\nu\nu_2 \overline{U}_1 + 3 \nu \nu_2 \overline{U}_1^2 - 3\nu_2 \overline{U}_2 + 3\nu_2^2 \overline{U}_1^2* \overline{U}_2\|_{1,\tau})R \\
    &+ (\|-3\nu_2 + 3\nu_2^2 \overline{U}_1^2\|_{1,\tau} + \|3\nu \nu_2 + 3\nu \nu_2 \overline{U}_1 + 3\nu_2^2 \overline{U}_1\|_{1,\tau}) R^2 + (\|3\nu_2^2 \overline{U}_1 \|_{1,\tau} + \|\nu \nu_2 + \nu_2^2 \overline{U}_2\|_{1,\tau})R^3 + \nu_2^2 R^4 \\
    &\bydef Z_{2,1}.
\end{align}}
Returning to \eqref{introduce Z211}, we get
\begin{align}
    \left\|\frac{\nu_2^2 U_1^3 *U_2 + \nu \nu_2 U_1^3 + 3\nu \nu_2 U_1^2 - 3\nu_2 U_1 *U_2 - \nu - U_2}{(1 + U_1 + \nu_2 U_1^2)^3} *W_1*V_1\right\|_{1,\tau} \leq Z_{2,1} Z_{2,2}.
\end{align}
Therefore, we obtain
\begin{align}
\|D^2F_p(\mathbf{U})\|_{\mathcal{B}(\ell^1_{e,\tau},\mathcal{B}(\ell^1_{e,\tau}))} \leq 2\nu_1(Z_{2,1} Z_{2,2} + Z_{2,3}).
\end{align}
Multiplying by $\|A_p^N\|_{\mathcal{B}(\ell^1_{e.\tau})} + \mathcal{L}_{\infty}$, we obtain the final result.
As we now have all the needed quantities, we conclude the proof.
\end{proof}
\bibliographystyle{abbrv}
\bibliography{biblio}

@article{thomas_1d,
    author = {{Al Saadi}, Fahad and Worthy, Annette and Alriheieli, Haifaa and Nelson, Mark},
    title = "{Localised spatial structures in the {T}homas model}",
    journal = {Mathematics and Computers in Simulation},
    volume = {194},
    pages = {141-158},
    year = {2022}
}

@article{gen_rd,
    author = {W. Hao and C. Xue},
    title = {Spatial pattern formation in reaction–diffusion models: a computational approach},
    journal = {Journal of Mathematical Biology},
    year = {2020},
    volume = {80},
    pages = {521-543},
    doi = {https://doi.org/10.1007/s00285-019-01462-0}
}

@book{murray_book1,
  title={Mathematical Biology},
  author={J. D. Murray},
  volume={1},
  year={2003},
  publisher={Springer Nature Link}
}

@book{murray_book2,
  title={Mathematical Biology II},
  author={J. D. Murray},
  volume={2},
  year={2003},
  publisher={Springer Nature Link}
}

@article{original_thomas,
    author = {L. Thomas and A.J. Gibson and S.K. Bhattacharyya },
    title = {Spatial and temporal variations of the atmospheric sodium layer observed with a steerable laser radar},
    journal = {Nature},
    year = {1976},
    volume = {263},
    pages = {115-116},
    doi = {https://doi.org/10.1038/263115a0}
}

@article{thomas_1,
    author = {J.F. Hervagault and A. Friboulet and J.P. Kernevez and D. Thomas },
    title = {Spatiotemporal behaviors in immobilized enzyme systems},
    journal = {Biochimie},
    year = {1980},
    volume = {62},
    pages = {367-373},
    doi = {https://doi.org/10.1016/S0300-9084(80)80167-2}
}

@article{thomas_2,
    author = {J.P. Kernevez and E. Doedel and M.C. Duban and J.F. Hervagault and G. Joly and D. Thomas },
    title = {Spatio-{T}emporal {O}rganization in {I}mmobilized {E}nzyme {S}ystems},
    journal = {Rhythms in Biology and Other Fields of Application},
    year = {1983},
    volume = {49},
    pages = {50-75}
}

@article{thomas_3,
    author = {J.P. Kernevez and G. Joly and M. Duban and B. Bunow and D. Thomas },
    title = {Hysteresis, oscillations, and pattern formation in realistic immobilized enzyme systems},
    journal = {Journal of Mathematical Biology},
    year = {1979},
    volume = {7},
    issue = {1},
    pages = {41-56},
    doi={10.1007/BF00276413}
}

@article{thomas_sanderwanner,
    author = {Evelyn Sander and Thomas Wanner },
    title = {Pattern formation in a nonlinear model for animal coats},
    journal = {Journal of Differential Equations},
    year = {2003},
    volume = {191},
    issue = {10},
    pages = {143-174},
    doi={https://doi.org/10.1016/S0022-0396(02)00156-0}
}

@article{murray_2,
    author = {J.D. Murray},
    title = {A Pre-pattern formation mechanism for animal coat markings},
    journal = {Journal of Theoretical Biology},
    year = {1981},
    volume = {88},
    issue = {1},
    pages = {161-199},
    doi={https://doi.org/10.1016/0022-5193(81)90334-9}
}

@article{murray_1,
    author = {J.D. Murray},
    title = {On pattern formation mechanisms for lepidopteran wing patterns and mammalian coat markings},
    journal = {Philosophical Transactions B},
    year = {1981},
    volume = {295},
    issue = {1078},
    pages = {473-496},
    doi={https://doi.org/10.1098/rstb.1981.0155}
}

@article{automatic_differentiation,
    author = {Jean-Philippe Lessard and J.D. Mireles-James and J. Ransford},
    title = {Automatic differentiation for {F}ourier series and the radii polynomial approach},
    journal = {Physica D: Nonlinear Phenomena},
    year = {2016},
    volume = {334},
    pages = {174-186},
    doi = {https://doi.org/10.1016/j.physd.2016.02.007}
}

@article{lindsay_suspensionbrdige,
    author = {van der Aalst, Lindsey and van den Berg, Jan Bouwe and Lessard, Jean-Philippe},
    title = {Periodic localised traveling waves in the two-dimensional suspension bridge equation},
    journal = {Nonlinearity},
    year = {2025},
    volume = {38},
    number ={7},
    pages = {075029},
    doi = {10.1088/1361-6544/ade5e5}
}

@article{JB_symmetries_1,
  title={{Rigorously Computing Symmetric Stationary Sates of the Ohta-Kawasaki Problem in Three Dimensions}},
  author={van den Berg, Jan-Bouwe and Williams, JF},
  journal={SIAM J. Math Anal.},
  year={2019},
  publisher={Society for Industrial and Applied Mathematics}
}

@article{suspension_bridge,
    author = {van den Berg, Jan Bouwe and Breden, Maxime and Lessard, Jean-Philippe and Murray, Maxime},
    title = {Continuation of homoclinic orbits in the suspension bridge equation: A computer-assisted proof},
    journal = {Journal of Differential Equations},
    year = {2018},
    volume = {264},
    issue = {5},
    pages = {3086-3130},
    doi = {https://doi.org/10.1016/j.jde.2017.11.011}
}

@article{chebyshev_parameterization,
    author = {van den Berg, Jan Bouwe and Deschenes, Andrea and Lessard, Jean-Philippe and Mireles-James, Jason},
    title = {Stationary coexistence of hexagons and rolls via rigorous computations},
    journal = {SIADS},
    year = {2015},
    volume = {2},
    pages = {942-979},
    doi = {10.1137/140984506}
}

@article{jp_saddle_node,
    author = {Lessard, Jean-Philippe and Sander, Evelyn and Wanner, Thomas},
    title = {Rigorous continuation of bifurcation points in the diblock copolymer equation},
    journal = {Journal of Computational Dynamics},
    year = {2017},
    volume = {4(1\&2)},
    pages = {71-118},
    doi = {10.3934/jcd.2017003}
}

@article{jp_saddle_node1,
    author = {Lessard, Jean-Philippe },
    title = {Rigorous verification of saddle-node bifurcations in {ODE}s},
    journal = {Indagationes Mathematicae},
    year = {2016},
    volume = {27},
    pages = {1013-1026},
    doi = {https://doi.org/10.1016/j.indag.2016.06.012}
}

@article{jp_cusps,
title = {Cusp bifurcations: Numerical detection via two-parameter continuation and computer-assisted proofs of existence},
journal = {Discrete and Continuous Dynamical Systems - B},
volume = {30},
number = {6},
pages = {2135-2158},
year = {2025},
issn = {1531-3492},
doi = {10.3934/dcdsb.2024181},
url = {https://www.aimsciences.org/article/id/675c0ed9d0b64a5d1d3497ef},
author = {Jean-Philippe Lessard and Alessandro Pugliese}
}

@article{period_kuramoto,
  title={A Posteriori Verification of Invariant Objects of Evolution Equations: Periodic Orbits in the {K}uramoto--{S}ivashinsky {PDE}},
  author={Gameiro, Marcio and Lessard, Jean-Philippe},
  journal={SIAM Journal on Applied Dynamical Systems},
  volume={16},
  number={1},
  pages={687--728},
  year={2017},
  publisher={SIAM}
}

@article {MR3633778,
    AUTHOR = {Figueras, Jordi-Llu\'{\i}s and de la Llave, Rafael},
     TITLE = {Numerical computations and computer assisted proofs of
              periodic orbits of the {K}uramoto-{S}ivashinsky equation},
   JOURNAL = {SIAM J. Appl. Dyn. Syst.},
  FJOURNAL = {SIAM Journal on Applied Dynamical Systems},
    VOLUME = {16},
      YEAR = {2017},
    NUMBER = {2},
     PAGES = {834--852},
   MRCLASS = {65G40 (35B32 35Q53 35R20 47J15 65H20)},
  MRNUMBER = {3633778},
MRREVIEWER = {Ferenc Agoston Bartha},
       DOI = {10.1137/16M1073790},
       URL = {https://doi-org.proxy3.library.mcgill.ca/10.1137/16M1073790},
}

@article{Sander_equilibrium,
title = {Equilibrium validation in models for pattern formation based on {S}obolev embeddings},
journal = {Discrete and Continuous Dynamical Systems - B},
volume = {26},number = {1},pages = {603-632},
year = {2021},
issn = {1531-3492},
doi = {10.3934/dcdsb.2020260},
url = {/article/id/0d611d64-31c0-4483-bce5-ab7986e7d976},
author = {Evelyn Sander and Thomas Wanner},
}

@article{gabriel_pfc,
  title={{M}icroscopic patterns in the 2{D} phase-field crystal model},
  author={Boissoniere, Gabriel M.L. and Choksi, Rustum and Lessard, Jean-Philippe},
  journal={Nonlinearity},
  year={2022},
volume={35},
pages={1500-1520},
  publisher={London Mathematical Society}
}

@article {MR4379799,
    AUTHOR = {Cyranka, Jacek and Lessard, Jean-Philippe},
     TITLE = {Validated forward integration scheme for parabolic {PDE}s via
              {C}hebyshev series},
   JOURNAL = {Commun. Nonlinear Sci. Numer. Simul.},
  FJOURNAL = {Communications in Nonlinear Science and Numerical Simulation},
    VOLUME = {109},
      YEAR = {2022},
     PAGES = {Paper No. 106304, 32},
      ISSN = {1007-5704},
   MRCLASS = {65M70 (65G20)},
  MRNUMBER = {4379799},
       DOI = {10.1016/j.cnsns.2022.106304},
       URL = {https://doi-org.proxy3.library.mcgill.ca/10.1016/j.cnsns.2022.106304},
}

@article {MR3392647,
    AUTHOR = {Breden, Maxime and Desvillettes, Laurent and Lessard,
              Jean-Philippe},
     TITLE = {Rigorous numerics for nonlinear operators with tridiagonal
              dominant linear part},
   JOURNAL = {Discrete Contin. Dyn. Syst.},
  FJOURNAL = {Discrete and Continuous Dynamical Systems. Series A},
    VOLUME = {35},
      YEAR = {2015},
    NUMBER = {10},
     PAGES = {4765--4789},
      ISSN = {1078-0947},
   MRCLASS = {47J05 (34B08 42A10 47J07 65L10)},
  MRNUMBER = {3392647},
MRREVIEWER = {Dharmendra Kumar Gupta},
       DOI = {10.3934/dcds.2015.35.4765},
       URL = {https://doi-org.proxy3.library.mcgill.ca/10.3934/dcds.2015.35.4765},
}

@article{jay_jp_gs,
  title={Rigorous numerics for symmetric connecting orbits: even homoclinics of the {G}ray-{S}cott equation},
  author={van den Berg, Jan Bouwe and Breden, Maxime and Lessard, Jean-Philippe and Mireles-James, Jason and Mischaikow, Konstantin},
  journal={SIAM Journal on Mathematical Analysis},
  volume={43},
  number={4},
  pages={1557--1594},
  year={2011},
  publisher={SIAM}
}

@article{cyranka2018construction,
  title={A construction of two different solutions to an elliptic system},
  author={Cyranka, Jacek and Mucha, Piotr Bogus{\l}aw},
  journal={Journal of Mathematical Analysis and Applications},
  volume={465},
  number={1},
  pages={500--530},
  year={2018},
  publisher={Elsevier}
}

@article{polynomial_chaos,
author = {Breden, Maxime},
title = {{A Posteriori Validation of Generalized Polynomial Chaos Expansions}},
journal = {SIAM Journal on Applied Dynamical Systems},
volume = {22},
number = {2},
pages = {765-801},
year = {2023},
doi = {10.1137/22M1493197},
URL = {https://doi.org/10.1137/22M1493197},
eprint = {  https://doi.org/10.1137/22M1493197}
}

@article{marschal,
  title={{From the Lagrange Triangle to the Figure Eight Choreography: Proof of Marchal's Conjecture}},
  author={Calleja, Renato and Garcia-Azpeitia, Carlos and Henot, Olivier and Lessard, Jean-Philippe and Mireles-James, Jason},
  journal={arXiv:2406.17564},
  year={2024},
  publisher={arXiv}
}

@article{dominic_sh_periodic,
  title={{Proving periodic solutions and branches in the 2D Swift Hohenberg PDE with hexagonal and triangular symmetry}},
  author={Blanco, Dominic},
  journal={arXiv:2602.12491},
  year={2026},
  publisher={arXiv}
}

@article{miguel_soliton,
  title={{Computer-Assisted Proofs of Gap Solitons in Bose-Einstein
Condensates}},
  author={Ayala, Miguel and  Garcia-Azpeitia, Carlos and Lessard, Jean-Philippe},
  journal={arXiv:2503.04701
},
  year={2026},
  publisher={arXiv}
}

@article{matthieu_lindsay,
  title={{Existence proofs of traveling wave solutions on an infinite strip for the suspension bridge equation and proof of orbital st}},
  author={van der Aalst, Lindsey and Cadiot, Matthieu},
  journal={arXiv:2509.16693},
  year={2026},
  publisher={arXiv}
}

@article{olivier_kevin_paper,
  title={{Global Continuation of Stable Periodic Orbits in Systems of Competing Predators}},
  author={Church, Kevin E.M. and Dai, Jia-Yuan and Henot, Olivier and Lappicy, Phillipo},
  journal={arXiv:2504.03058v2},
  year={2026},
  publisher={arXiv}
}

@article{maxime_paper_continuation,
title = {Computer-assisted proofs for the many steady states of a chemotaxis model with local sensing},
journal = {Physica D: Nonlinear Phenomena},
volume = {466},
pages = {134221},
year = {2024},
author = {M. Breden and M. Payan}}

@misc{maxime_general,
  title={{Computer-assisted proofs for differential equations and dynamical systems}},
  author={Breden, Maxime},
  year={2025},
}

@article{continuation_1,
  title={{Global bifurcation diagrams of steady states of systems of PDEs via rigorous numerics: a 3-component reaction-diffusion system}},
  author={Breden, Maxime and Lessard, Jean-Philippe and Vanicat, Matthieu},
  journal={Acta Applicandae Mathematicae},
volume = {128},
number = {1},
pages = {113-152},
  year={2013},
doi = {10.1007/s10440-013-9823-6},
URL = {https://doi.org/10.1007/s10440-013-9823-6}
}

@article{continuation_2,
author = {Day, Sarah and Lessard, Jean-Philippe and Mischaikow, Konstantin},
title = {{Validated Continuation for Equilibria of PDEs}},
journal = {SIAM Journal on Numerical Analysis},
volume = {45},
number = {4},
pages = {1398-1424},
year = {2007},
doi = {10.1137/050645968},
URL = {https://doi.org/10.1137/050645968},
eprint = {  https://doi.org/10.1137/050645968}
}

@article{continuation_3,
title = {{Continuation of homoclinic orbits in the suspension bridge equation: A computer-assisted proof}},
journal = {Journal of Differential Equations},
volume = {264},
number = {5},
pages = {3086-3130},
year = {2018},
doi = {https://doi.org/10.1016/j.jde.2017.11.011},
url = {https://www.sciencedirect.com/science/article/pii/S0022039617306010},
author = {van den Berg, Jan Bouwe and Breden, Maxime and Lessard, Jean-Philippe and Murray, Maxime}
}

@article{unbounded_domain_cadiot,
author = {Cadiot, Matthieu and Lessard, Jean-Philippe and Nave, Jean-Christophe},
title = {{Rigorous Computation of Solutions of Semilinear PDEs on Unbounded Domains via Spectral Methods}},
journal = {SIAM Journal on Applied Dynamical Systems},
volume = {23},
number = {3},
pages = {1966-2017},
year = {2024},
doi = {10.1137/23M1607507},
URL = {https://doi.org/10.1137/23M1607507},
eprint = {  https://doi.org/10.1137/23M1607507}
}

@article{blanco_cadiot_fassler_saddle,
  title={Proving the existence of localized patterns and saddle node bifurcations in 1{D} activator-inhibitor type models},
  author={Blanco, Dominic and  Cadiot, Matthieu and Fassler, Daniel},
  journal={arXiv:2509.17099},
  year={2025},
  publisher={arXiv}
}

@misc{breden2025solutionsdifferentialequationsfreudweighted,
      title={Solutions of differential equations in {F}reud-weighted {S}obolev spaces}, 
      author={Maxime Breden and Hugo Chu},
      year={2025},
      eprint={2501.13672},
      archivePrefix={arXiv},
      primaryClass={math.CA},
      url={https://arxiv.org/abs/2501.13672}, 
}

@article{breden2024constructiveproofssemilinearpdes,
      title={Constructive proofs for some semilinear {PDE}s on ${H}^2(e^{|x|^2/4},\mathbb{R}^d)$}, 
      author={Maxime Breden and Hugo Chu},
      year={2025},
      journal = {Numerische Mathematik},
      volueme={157},
      issue = {6},
      pages={2097-2143},
      doi={https://doi.org/10.1007/s00211-025-01504-4}
}

@misc{ThomasProofs.jl,
  author = {Dominic Blanco},
  title  = {Thomas{P}roofs.jl},
  URL    = {https://github.com/dominicblanco/ThomasProofs.jl},
  note = {\url{ https://https://github.com/dominicblanco/ThomasProofs.jl}},
  year   = {2026},
  doi    = {10.5281/zenodo.19487919}
}

@article{sh_cadiot,
title = {{Stationary non-radial localized patterns in the planar Swift-Hohenberg PDE: Constructive proofs of existence}},
journal = {Journal of Differential Equations},
volume = {414},
pages = {555-608},
year = {2025},
issn = {0022-0396},
doi = {https://doi.org/10.1016/j.jde.2024.09.015},
url = {https://www.sciencedirect.com/science/article/pii/S0022039624005941},
author = {Matthieu Cadiot and Jean-Philippe Lessard and Jean-Christophe Nave}
}

@article{gs_cadiot_blanco,
  title={{Localized stationary patterns in the 2D Gray Scott model: computer-assisted proofs of existence}},
  author={Cadiot, Matthieu and  Blanco, Dominic},
  journal={Nonlinearity},
  volume={38},
number={4},
pages={045016},
  year={2025},
  publisher={London Mathematical Society}
}

@article{symmetry_blanco_cadiot,
  title={Proving symmetry of localized solutions and application to dihedral patterns in the planar {S}wift-{H}ohenberg {PDE}},
  author={Blanco, Dominic and  Cadiot, Matthieu},
  journal={arXiv:2509.10375},
  year={2025},
  publisher={arXiv}
}

@article{whitham_cadiot,
  title={Constructive proofs of existence and stability of solitary waves in the {W}hitham
and capillary–gravity {W}hitham equations},
  author={Matthieu Cadiot},
  journal={Nonlinearity},
  volume={38},
number={3},
pages={035021},
  year={2025},
  publisher={London Mathematical Society}
}

@misc{julia_blanco_fassler,
  author = {Blanco, Dominic and Cadiot, Matthieu and Fassler, Daniel},
  title  = {ProofActivatorInhibitor.jl},
  URL    = {https://github.com/dominicblanco/LocalizedPatternsActivatorInhibitor.jl},
  note = {https://github.com/dominicblanco/LocalizedPatternsActivatorInhibitor.jl},
  year   = {2025},
  doi    = {10.5281/zenodo.17170735}
}

@mastersthesis{marco_thesis,
    author ={Marco Mignacca} ,
    title ={Rigorous proof of one-dimensional steady states in a neural field model} ,
    school ={McGill University} ,
    year ={2025} 
}

@article{olivier_radial,
  title={Constructive proofs for localized radial solutions of semilinear elliptic systems on {$\mathbb{R}^d$}},
  author={{van den Berg}, Jan Bouwe and Hénot, Olivier and Lessard, Jean-Philippe},
  journal={Nonlinearity},
  volume={36},
number={12},
pages={6476-6512},
  year={2023},
  publisher={London Mathematical Society}
}

@article {MR1976080,
    AUTHOR = {Cabr\'{e}, Xavier and Fontich, Ernest and de la Llave, Rafael},
     TITLE = {The parameterization method for invariant manifolds. {II}.
              {R}egularity with respect to parameters},
   JOURNAL = {Indiana Univ. Math. J.},
  FJOURNAL = {Indiana University Mathematics Journal},
    VOLUME = {52},
      YEAR = {2003},
    NUMBER = {2},
     PAGES = {329--360},
      ISSN = {0022-2518},
   MRCLASS = {37D10},
  MRNUMBER = {1976080},
MRREVIEWER = {Weishi Liu},
       DOI = {10.1512/iumj.2003.52.2407},
       URL = {https://doi-org.proxy3.library.mcgill.ca/10.1512/iumj.2003.52.2407},
}

@article {MR2177465,
    AUTHOR = {Cabr\'{e}, Xavier and Fontich, Ernest and de la Llave, Rafael},
     TITLE = {The parameterization method for invariant manifolds. {III}.
              {O}verview and applications},
   JOURNAL = {J. Differential Equations},
  FJOURNAL = {Journal of Differential Equations},
    VOLUME = {218},
      YEAR = {2005},
    NUMBER = {2},
     PAGES = {444--515},
      ISSN = {0022-0396},
   MRCLASS = {37D10 (37M99)},
  MRNUMBER = {2177465},
       DOI = {10.1016/j.jde.2004.12.003},
       URL = {https://doi-org.proxy3.library.mcgill.ca/10.1016/j.jde.2004.12.003},
}

@article {MR1976079,
    AUTHOR = {Cabr\'{e}, Xavier and Fontich, Ernest and de la Llave, Rafael},
     TITLE = {The parameterization method for invariant manifolds. {I}.
              {M}anifolds associated to non-resonant subspaces},
   JOURNAL = {Indiana Univ. Math. J.},
  FJOURNAL = {Indiana University Mathematics Journal},
    VOLUME = {52},
      YEAR = {2003},
    NUMBER = {2},
     PAGES = {283--328},
      ISSN = {0022-2518},
   MRCLASS = {37D10},
  MRNUMBER = {1976079},
MRREVIEWER = {Weishi Liu},
       DOI = {10.1512/iumj.2003.52.2245},
       URL = {https://doi-org.proxy3.library.mcgill.ca/10.1512/iumj.2003.52.2245},
}

@book{plum_numerical_verif,
  title={Numerical verification methods and computer-assisted proofs for partial differential equations},
  author={Nakao, Mitsuhiro T and Plum, Michael and Watanabe, Yoshitaka},
  year={2019},
  publisher={Springer}
}

@phdthesis{plum_thesis_navierstokes,
    author       = {Wunderlich, Jonathan Matthias},
    year         = {2022},
    title        = {Computer-assisted Existence Proofs for {N}avier-{S}tokes Equations on an Unbounded Strip with Obstacle},
    doi          = {10.5445/IR/1000150609},
    publisher    = {{Karlsruher Institut für Technologie (KIT)}},
    pagetotal    = {215},
    school       = {Karlsruher Institut für Technologie (KIT)},
    language     = {english}
}

@misc{julia_olivier,
author = {Olivier Hénot},
title = {RadiiPolynomial.jl},
year = {2022},
URL = {  https://github.com/OlivierHnt/RadiiPolynomial.jl
},
eprint = { https://doi.org/10.1137/141000671
},

 note = {\url{  https://github.com/OlivierHnt/RadiiPolynomial.jl}}
  
}

@misc{julia_interval,
author = {L. Benet and D.P. Sanders},
title = {IntervalArithmetic.jl},
year = {2022},
URL = { https://github.com/JuliaIntervals/IntervalArithmetic.jl
},
eprint = { https://doi.org/10.1137/141000671
},
  note = {\url{ https://github.com/JuliaIntervals/IntervalArithmetic.jl}}
}

@article{AlSaadi_unifying_framework,
    author = {{Al Saadi}, Fahad and Champneys, Alan and Verschueren, Nicolas},
    title = "{Localized patterns and semi-strong interaction, a unifying framework for reaction–diffusion systems}",
    journal = {IMA Journal of Applied Mathematics},
    volume = {86},
    number = {5},
    pages = {1031-1065},
    year = {2021}
}

@article{Champneys_Bistability,
title = {Bistability, wave pinning and localisation in natural reaction–diffusion systems},
journal = {Physica D: Nonlinear Phenomena},
volume = {416},
pages = {132735},
year = {2021},
author = {Alan Champneys and Fahad {Al Saadi} and Victor F. Breña–Medina and Verônica A. Grieneisen and Athanasius F.M. Marée and Nicolas Verschueren and Bert Wuyts}}

@article{schankenberg_og,
title = {Simple chemical reaction systems with limit cycle behaviour},
journal = {Journal of Theoretical Biology},
volume = {81},
number = {3},
pages = {389-400},
year = {1979},
issn = {0022-5193},
url = {https://www.sciencedirect.com/science/article/pii/0022519379900420},
author = {J. Schnakenberg},
}

@article{selkov_schankenberg_og,
author = {Sel'kov, E. E.},
title = {Self-Oscillations in Glycolysis 1. A Simple Kinetic Model},
journal = {European Journal of Biochemistry},
volume = {4},
number = {1},
pages = {79-86},
year = {1968}
}

@article{brusselator_og,
    author = {Prigogine, I. and Lefever, R.},
    title = {Symmetry Breaking Instabilities in Dissipative Systems. II},
    journal = {The Journal of Chemical Physics},
    year = {1968}
}

@article{root_hair_og,
author = {Payne, R and Grierson, C},
title = {A theoretical model for ROP localisation by auxin in Arabidopsis root hair cells},
journal = {PLoS ONE},
year = {2009}
}

@article{GrayScott1985,
  author    = {P. Gray and S. K. Scott},
  title     = {Sustained oscillations and other exotic patterns of behavior in isothermal reactions},
  journal   = {Journal of Physical Chemistry},
  volume    = {89},
  number    = {1},
  pages     = {22--32},
  year      = {1985},
  doi       = {10.1021/j100247a003}
}

@article{SITEUR201481,
title = {Beyond Turing: The response of patterned ecosystems to environmental change},
journal = {Ecological Complexity},
volume = {20},
pages = {81-96},
year = {2014},
issn = {1476-945X},
doi = {https://doi.org/10.1016/j.ecocom.2014.09.002},
url = {https://www.sciencedirect.com/science/article/pii/S1476945X14000944},
author = {Koen Siteur and Eric Siero and Maarten B. Eppinga and Jens D.M. Rademacher and Arjen Doelman and Max Rietkerk}
}

@article{ZELNIK201727,
title = {Desertification by front propagation?},
journal = {Journal of Theoretical Biology},
volume = {418},
pages = {27-35},
year = {2017},
issn = {0022-5193},
doi = {https://doi.org/10.1016/j.jtbi.2017.01.029},
url = {https://www.sciencedirect.com/science/article/pii/S0022519317300280},
author = {Yuval R. Zelnik and Hannes Uecker and Ulrike Feudel and Ehud Meron}}

@article{continue1,
  author    = {van den Berg, Jan-Bouwe and Lessard, Jean-Philippe and Mischaikow, Konstantin},
  title     = {Global smooth solution curves using rigorous branch following},
  journal   = {Mathematics of Computation},
  volume    = {79},
  number = {271},
  pages     = {1565--1584},
  year      = {2010}
}

@article{continue2,
  author    = {Sander, Evelyn and Wanner, Thomas},
  title     = {Equilibrium validation in models for pattern formation based on Sobolev embeddings},
  journal   = {Discrete and Continuous Dynamical Systems - B},
  volume    = {26},
  issue = {1},
  pages     = {603-632},
  year      = {2020},
  doi = {10.3934/dcdsb.2020260}
}

@article{continue3,
  author    = {Arioli, Gianni and Koch, Hans},
  title     = {{Computer-Assisted Methods for the Study of Stationary Solutions in Dissipative Systems, Applied to the Kuramoto-Sivashinski Equation}},
  journal   = {Archive for Rational Mechanics and Analysis},
  volume    = {197},
  issue = {3},
  pages     = {1033-1051},
  year      = {2010},
  doi = {10.1007/s00205-010-0309-7}
}

@article{olivier_embed,
  author    = {Olivier Henot},
  title     = {On polynomial forms of nonlinear functional differential equations},
  journal   = {Journal of Computational Dynamics},
  volume = {8},
  issue = {3},
  pages     = {307-323},
  year      = {2021},
  doi = {10.3934/jcd.2021013}
}

@article{KAM,
  author    = {Figueras, Jordi-Lluis and Haro, Alex and Luque, Alejandro},
  title     = {Rigorous computer assisted application of KAM theory: a modern approach},
  journal   = {Foundations of Computational Mathematics},
  volume = {17},
  issue = {5},
  pages     = {1123-1193},
  year      = {2017}
}

@article{nakao_sug3,
  author    = {Nakao, M.T. and Watanabe, Y. and Yamamoto, N. and Nishida, T. and Kim, M.-N.},
  title     = {{Computer Assisted Proofs
of Bifurcating Solutions for Nonlinear Heat Convection Problems}},
  journal   = {J. Sci. Comput. 43},
  pages     = {388--401},
  year      = {2010}
}

@article{arioli_sug2,
  author    = {G. Arioli},
  title     = {{Computer Assisted Proof of Branches of Stationary and Periodic Solutions, and
Hopf Bifurcations, for Dissipative PDEs}},
  journal   = {Commun. Nonlinear Sci. Numer. Simul. },
  volume     = {105},
  year      = {2022}
}
\end{document}